\documentclass{amsart}
\usepackage[lofdepth, lotdepth]{subfig}
\usepackage{etex}
\usepackage{pstricks,  amssymb}
\usepackage{pict2e}
\usepackage{geometry}
\usepackage{rotating}
\usepackage{pst-node}
\usepackage{amsmath}
\usepackage{dynkin-diagrams}
\def\row#1/#2!{#1_{\IfStrEq{#2}{}{n}{#2}} & \dynkin{#1}{#2}\\}

\usepackage{xypic}
\usepackage[utf8]{inputenc}
\usepackage{tikz}
\usepackage{tikz-cd}
\usetikzlibrary{arrows}
\usetikzlibrary{decorations.pathreplacing,angles,quotes}

\usepackage{quiver}
\usepackage{hyperref}
\hypersetup{
	  colorlinks,
	  citecolor=black,
	  filecolor=black,
	  linkcolor=black,
	  urlcolor=black
  }
\usepackage[capitalize,nameinlink,noabbrev,nosort]{cleveref}
\usepackage{amsmath,amscd}
\usepackage{youngtab}
\usepackage[boxsize=.3 em]{ytableau}
\usepackage{verbatim}

\numberwithin{equation}{section}

\theoremstyle{plain}
\newtheorem{theorem}{Theorem}[section]

\newtheorem{prop}[theorem]{Proposition}
\newtheorem{cor}[theorem]{Corollary}
\newtheorem{lemma}[theorem]{Lemma}
\newtheorem{conj}{Conjecture}
\newtheorem{mainthm}{Theorem}

\theoremstyle{definition}
\newtheorem{defn}[theorem]{Definition}
\newtheorem{example}[theorem]{Example}

\theoremstyle{remark}
\newtheorem{remark}[theorem]{Remark}

\newcommand{\To}{\longrightarrow}

\newcommand{\Minor}{\operatorname{Minor}}

\newcommand{\pitrop}{\pi_{trop}}

\renewcommand{\v}{{\mathrm v}}
\newcommand{\Rmin}{{\R_{\min}}}
\newcommand{\Rmininf}{{\R_{{\min},\infty}}}
\newcommand{\Imin}{\mathcal {IM}}

\newcommand{\Toeplitz}{\operatorname{Toep}}
\newcommand{\ToeplitzU}{\operatorname{UToep}}

\renewcommand{\L}{L}

\newcommand{\RR}{\mathcal R}
\newcommand{\ep}{\varepsilon}

\newcommand{\bigA}{\mathbf A}
\newcommand{\bigB}{\mathbf B}
\newcommand{\x}{\times}
\newcommand{\Q}{Q}

\newcommand{\inv}{^{-1}}

\newcommand{\pr}{\mathrm{pr}}

\newcommand{\R}{\mathcal{R}}

\newcommand{\C}{\mathbb{C}}
\renewcommand{\Q}{\mathbb{Q}}
\newcommand{\Z}{\mathbb{Z}}

\newcommand{\crit}{d_{\operatorname{crit}}}

\newcommand{\lc}{{\operatorname{lc}}}

\newcommand{\summ}{\operatorname{sum}}

                                                                                                                                                                                                                                                                                                                                                                                                                                                                                                                                                                                                                                                                                                                                                                                                                                                                                                                                                                                                                                                                                                                                                                                                                                                                                                                                                                                                                                                                                                                                                                                                                                                                                                                                                                                                                                                                                                                                                                                                                                                                                                                                                                                                                                                                                                                                                                                                                                                                                                                                                                                                                                                                                                                                                                                                                                                                                                                                                                                                                                                                                                                                                                                                                                                                                                                                                                                                                                                                                                                                                                                                                                                                                                                                                                                                                                                                                                                                                                                                                                                                                                                                                                                                                                                                                                                                                                                                                                                                                                                                                                                                                                                                                                                                                                                                                                                                                                                                                                                                                                                                                                                                                                                                                                                                                                                                                                                                                                                                                                                                                                                                                                                                                                                                                                                                                                                                                                                                                                                                                                                                                                                                                                                                                                                                                                                                                                                                                                                                                                                                                                                                                                                                                                                                                                                                                                                                                                                                                                                                                                                                                                                                                                                                                                                                                                                                                                                                                                                                                                                                                                                                                                                                                                                                                                                                                                                                                                                                                                                                                                                                                                                                                                                                                                                                                                                                                                                                                                                                                                                               \newcommand{\Mdelpos}{{\mathcal M_{>0}}}
\newcommand{\Odelpos}{{\mathcal O_{>0}}}
                                                                                                                                                                                                                                                                                                                                                                                                                                                                                                                                                                                                                                                                                                                                                                                                                                                                                                                                                                                                                                                                                                                                                                                                                                                                                                                                                                                                                                                                                                                                                                                                                                                                                                                                                                                                                                                                                                                                                                                                                                                                                                                                                                                                                                                                                                                                                                                                                                                                                                                                                                                                                                                                                                                                                                                                                                                                                                                                                                                                                                                                                                                                                                                                                                                                                                                                                                                                                                                                                                                                                                                                                                                                                                                                                                                                                                                                                                                                                                                                                                                                                                                                                                                                                                                         
                                                                                                                                                                                                                                                                                                                                                                                                                                                                                                                                                                                                                                                                                                                                                                                                                                                                                                                                                                                                                                                                                                                                                                                                                                                                                                                                                                                                                                                                                                                                                                                                                                                                                                                                                                                                                                                                                                                                                                                                                                                                                                                                                                                                                                                                                                                                                                                                                                                                                                                                                                                                                                                                                                                                                                                                                                                                                                                                                                                                                                                                                                                                                                                                                                                                                                                                                                                                                                                                                                                                                                                                                                                                                                                                                                                                                                                                                                                                                                                                                                                                                                                                                                                                                                                         \newcommand{\Kdel} {{\mathcal C}}                                                                                                                                                                                                                                                                                                                                                                                                                                                                                                                                                                                                                                                                                                                                                                                                                                                                                                                                                                                                                                                                                                                                                                                                                                                                                                                                                                                                                                                                                                                                                                                                                                                                                                                                                                                                                                                                                                                                                                                                                                                                                                                                                                                                                                                                                                                                                                                                                                                                                                                                                                                                                                                                                                                                                                                                                                                                                                                                                                                                                                                                                                                                                                                                                                                                                                                                                                                                                                                                                                                                                                                                                                                                                                                                                                                                                                                                                                                                                                                                                                                                                                                                                                                                                                      \newcommand{\Kdelnn} {{\mathcal C_{\ge 0}}}
 \newcommand{\Kdelnnst}{{\mathcal C_{\ge 0}^{\rm{st}}}}
 \newcommand{\Kdelnnwk}{{\mathcal C_{\ge 0}^{\rm{wk}}}}
                                                                                                                                                                                                                                                                                                                                                                                                                                                                                                                                                                                                                                                                                                                                                                                                                                                                                                                                                                                                                                                                                                                                                                                                                                                                                                                                                                                                                                                                                                                                                                                                                                                                                                                                                                                                                                                                                                                                                                                                                                                                                                                                                                                                                                                                                                                                                                                                                                                                                                                                                                                                                                                                                                                                                                                                                                                                                                                                                                                                                                                                                                                                                                                                                                                                                                                                                                                                                                                                                                                                                                                                                                                                                                                                                                                                                                                                                                                                                                                                                                                                                                                                                                            
                                                                                                                                                                                                                                                                                                                                                                                                                                                                                                                                                                                                                                                                                                                                                                                                                                                                                                                                                                                                                                                                                                                                                                                                                                                                                                                                                                                                                                                                                                                                                                                                                                                                                                                                                                                                                                                                                                                                                                                                                                                                                                                                                                                                                                                                                                                                                                                                                                                                                                                                                                                                                                                                                                                                                                                                                                                                                                                                                                                                                                                                                                                                                                                                                                                                                                                                                                                                                                                                                                                                                                                                                                                                                                                                                                                                                                                                                                                                                                                                                                                                                                                                                                            \newcommand{\Kdelpos}{{\mathcal C_{>0}}}
 \newcommand{\Kdelposst}{{\mathcal C_{>0}^{\rm{st}}}}
                                                                                                                                                                                                                                                                                                                                                                                                                                                                                                                                                                                                                                                                                                                                                                                                                                                                                                                                                                                                                                                                                                                                                                                                                                                                                                                                                                                                                                                                                                                                                                                                                                                                                                                                                                                                                                                                                                                                                                                                                                                                                                                                                                                                                                                                                                                                                                                                                                                                                                                                                                                                                                                                                                                                                                                                                                                                                                                                                                                                                                                                                                                                                                                                                                                                                                                                                                                                                                                                                                                                                                                                                                                                                                                                                                                                                                                                                                                                                                                                                                                                                                                                                                                                                                                            \newcommand{\Kdelposwk}{{\mathcal C_{>0}^{\rm{wk}}}}

                                                                                                                                                                                                                                                                                                                                                                                                                                                                                                                                                                                                                                                                                                                                                                                                                                                                                                                                                                                                                                                                                                                                                                                                                                                                                                                                                                                                                                                                                                                                                                                                                                                                                                                                                                                                                                                                                                                                                                                                                                                                                                                                                                                                                                                                                                                                                                                                                                                                                                                                                                                                                                                                                                                                                                                                                                                                                                                                                                                                                                                                                                                                                                                                                                                                                                                                                                                                                                                                                                                                                                                                                                                                                                                                                                                                                                                                                                                                                                                                                                                                                                                                                                                                                                                             \newcommand{\Kdelz}{{\mathcal C_{\ge 0}}}                                                                                                                                                                                                                                                                                                                                                                                                                                                                                                                                                                                                                                                                                                                                                                                                                                                                                                                                                                                                                                                                                                                                                                                                                                                                                                                                                                                                                                                                                                                                                                                                                                                                                                                                                                                                                                                                                                                                                                                                                                                                                                                                                                                                                                                                                                                                                                                                                                                                                                                                                                                                                                                                                                                                                                                                                                                                                                                                                                                                                                                                                                                                                                                                                                                                                                                                                                                                                                                                                                                                                                                                                                                                                                                                                                                                                                                                                                                                                                                                                                                                                                                                                                                                                                                                                                                                                                                                                                                                                                                                                                                                                                                                                                                                                                                                                                                                                                                                                                                                                                                                                                                                                                                                                                                                                                                                                                                                                                                                                                                                                                                                                                                                                                                                                                                                                                                                                                                                                                                                                                                                                                                                                                                                                                                                                                                                                                                                                                                                                                                                                                                                                                                                                                                                                                                                                                                                                                                                                                                                                                                                                                                                                                                                                                                                                                                                                                                                                                                                                                                                                                                                                                                                                                                                                                                                                                                                                                                                                                                                                                                                                                                                                                                                                                                                                                                                                                                                                                                                                                                                                                                                                                                                                                                                                                                                                                                                                                                                                                                                                                                                                                                                                                                                                                                                                                                                                                                                                                                                                                                                                                                                                                                                                                                                                                                                                                                                                                                                                                                                                                                                                                                                                                                                                                                                                                                                                                                                                                                                                                                                                                                                                                                                                                                                                                                                                                                                                                                                                                                                                                                                                                                                                                                                                                                                                                                                                                                                                                                                                                                                                                                                                                                                                                                                                                                                                                                                                                                                                                                                                                                                                                                                                                                                                                                                                                                                                                                                                                                                                                                                                                                                                                                                                                                                                                                                  \newcommand{\boldalpha}{\boldsymbol{\alpha}}                                                                                                                                                                                                                                                                                                                                                                                                                                                                                                                                                                                                                                                                                                                                                                                                                                                                                                                                                                                                                                                                                                                                                                                                                                                                                                                                                                                                                                                                                                                                                                                                                                                                                                                                                                                                                                                                                                                                                                                                                                                                                                                                                                                                                                                                                                                                                                                                                                                                                                                                                                                                                                                                                                                                                                                                                                                                                                                                                                                                                                                                                                                                                                                                                                                                                                                                                                                                                                                                                                                                                                                                                                                                                                                                                                                                                                                                                                                                                                                                                                                                                                                                                                                                                         
\newcommand{\boldbeta}{\boldsymbol{\beta}}

                                                                                                                                                                                                                                                                                                                                                                                                                                                                                                                                                                                                                                                                                                                                                                                                                                                                                                                                                                                                                                                                                                                                                                                                                                                                                                                                                                                                                                                                                                                                                                                                                                                  \renewcommand{\S}{\mathbf S}

                                                                                                                                                                                                                                                                                                                                                                                                                                                                                                                                                                                                                                                                                                                                                                                                                                                                                                                                                                                                                                                                                                                                                                                                                                                                                                                                                                                                                                                                                                                                                                                                                                              \newcommand{\N}{\mathbb N} 
                                                                                                                                                                                                                                                                                                                                                                                                                                                                                                                                                                                                                                                                                                                                                                                                                                                                                                                                                                                                                                                                                                                                                                                                                                                                                                                                                                                                                                                                                                                                                                                                                                             \renewcommand{\R}{\mathbb R}

\newcommand{\MonSeq}{\operatorname{MonSeq}}

\newcommand{\SMonSeq}{\operatorname{SMonSeq}}

\newcommand{\KK}{\mathbf K}

\newcommand{\K}{\mathbf k}
\newcommand{\ValK}{\operatorname{Val}_{\KK}}

\newcommand{\Val}{\operatorname{Val}}
\newcommand{\Trop}{\operatorname{Trop}}

\newcommand{\val}{{\operatorname{val}}}                                                                                                                                                                                                                                                                                                                                                                                                                                                                                                                                                                                                                                                                                                                                                                                                                                                                                                                                                                                                                                                                                                                                                                                                                                                                                                                                                                                                                                                                                                                                                                                                                                                                                                                                                                                                                                                                                                                                                                                                                                                                                                                                                                                                                                                                                                                                                                                                                                                                                                                                                                                                                                                                                                                                                                                                                                                                                                                                                                                                                                                                                                                                                                                                                                                                                                                                                                                                                                                                                                                                                                                                                                                                                                                                                                                                                                                                                                                                                                                                                                                                                                                                                                                                                       

\newcommand{\Ktrop}{\R_{\min{}}}

\tikzset{labl/.style={anchor=south, rotate=90, inner sep=.5mm}}

\title{Tropical Toeplitz matrices and parametrizations}
\author{Konstanze Rietsch}
\thanks{The author was supported by EPSRC grant EP/V002546/1}

\begin{document}

\begin{abstract}
The set of infinite upper-triangular totally positive Toeplitz matrices $\Toeplitz_\infty(\R_{>0})$ has a classical parametrisation \cite{ASW,Edrei52} originally conjectured by Schoenberg, that involves pairs of sequences of positive real parameters $(\boldalpha,\boldbeta)$. These matrices (and their  parameters) are central for understanding characters of the infinite symmetric group \cite{Thoma}. On the other hand there is a very different parametrisation theorem that applies to the finite analogue  $\Toeplitz_{n+1}(\R_{>0})$ of this set \cite{rietschJAMS}. These finite Toeplitz  matrices and their parameters relate to quantum cohomology of flag varieties and mirror symmetry. 
In this paper we replace $\R_{>0}$ by a  semifield $\RR_{>0}$ with valuation $\Val:\RR_{>0}\to\Rmin$, to then construct tropical analogues for both parametrisation theorems. In the finite case we tropicalise using positive generalised Puiseaux series $\KK_{>0}$. This builds on work of Judd \cite{Judd:Flag} and L\"udenbach  \cite{Ludenbach}. In the infinite case we replace $\KK_{>0}$ by a new valued semifield of continuous functions $\Kdelpos$. We arrive at different natural infinite analogues of $\ToeplitzU_{n+1}(\Rmin)$, depending on a choice of topology on $\Kdelnn$. We then prove an asymptotic result relating the tropical parameters from the finite case to the tropicalisations of the Schoenberg parameters. Moreover, we show that our finite type tropical parametrisation map is given by Lusztig's weight map from the theory of canonical bases. This results in a surprising connection between the classical Edrei theorem with its Schoenberg parameters $(\boldalpha,\boldbeta)$ and Lusztig's canonical basis parametrisation.
\end{abstract}
\maketitle

\setcounter{tocdepth}{1}
\tableofcontents

\section{Introduction}
Consider a sequence $\mathbf c=(c_i)_{i=1}^\infty$ of real numbers, and let $c_0:=1$ and $c_{-n}=0$ for all $n\in\N$. We call the sequence $\mathbf c$ \textit{totally nonnegative} if the corresponding infinite upper-triangular Toeplitz matrix, 
\begin{equation}\label{e:InfToeplIntro}
u(\mathbf c)=(c_{i-j})_{i,j\in\N}=\begin{pmatrix}
1 &c_1 & c_2& c_3& c_4 &  & \quad\\
 & 1  & c_1& c_2 & c_3& \ddots  &\\
  &    & 1  & c_1 & c_2 &\ddots &\\
  &    &    &1 &c_1 &\ddots &\\
  &    &    &       &  \ddots  &\ddots &  \\
  &&&&&& 
\end{pmatrix},
\end{equation}
has all minors in $\R_{\ge 0}$. 
A classical theorem gives a parametrisation of the set $\ToeplitzU_{\infty}(\R_{\ge 0})$ 
of  such infinite totally nonnegative Toeplitz matrices. 
Namely, consider the following set of parameters, 
\[
\Omega_{S}=\{(\gamma,\boldalpha,\boldbeta)\in\R_{\ge 0}\x \R_{\ge 0}^\N\x \R_{\ge 0}^{\N}\mid \text{$\boldalpha$ and  $\boldbeta$ are monotonely decreasing with finite sum}\}.
\]  
\begin{theorem} [{\cite{ASW,Edrei52,Thoma}}]\label{t:EdreiIntro}
A sequence $\mathbf c=(c_i)_{i=1}^\infty$ of real numbers is a totally nonnegative sequence if and only if its generating function has a factorisation of the form
\[
1+c_1x+c_2x^2+c_3x^3+\dotsc =e^{\gamma x}\prod_{i=1}^\infty\frac{1+\beta_ix}{1-\alpha_ix},
\]
for some $(\gamma,\boldalpha,\boldbeta)\in\Omega_{S}$, where $\boldalpha=(\alpha_i)_{i\in\N}$ and $\boldbeta=(\beta_i)_{i\in\N}$.
\end{theorem}
The fact that the power series expansion above gives a totally nonnegative sequence for all $(\gamma,\boldalpha,\boldbeta)\in \Omega_S$ goes back to Schoenberg \cite{Schoenberg:48}, who conjectured the above theorem. For the other direction of the theorem, the result can be divided into two statements. The first, concerning the roots and poles of the generating function $\sum c_ix^i$, was proved in \cite{ASW}. The part that was left, identifying the remaining factor as a simple exponential function of degree~$1$, was then proved by Edrei \cite{Edrei52}, who thereby completed the overall proof of the theorem.  
This theorem was later proved again independently by Thoma \cite{Thoma} in connection with his work on characters of the infinite symmetric group. We will refer to \cref{t:EdreiIntro} simply as the Edrei theorem here, it is also sometimes referred to as the Edrei-Thoma theorem or the Thoma theorem in connection with $S_\infty$ character theory. And we will call the parameters `Schoenberg parameters'. For more background we refer to the books of Borodin and Olshanski \cite{BorodinOlshanskiBook} and Karlin \cite{Karlin}.

About fifty years after \cite{ASW,Edrei52} a different parametrisation theorem for totally nonnegative Toeplitz matrices that takes place in the finite setting was given in \cite{rietschJAMS}. We state it in its totally positive version below. Let $\ToeplitzU_{n+1}(\R_{>0})$ denote the set of $(n+1)\x(n+1)$ totally positive unipotent upper-triangular Toeplitz matrices, and let  $T_{SL_{n+1}}(\R_{>0})$ denote the positive subset of the diagonal torus in $SL_{n+1}(\R)$ (where the diagonal entries are all positive).  
\begin{theorem}[{\cite{rietschJAMS,rietschNagoya}}]\label{t:RealFinite}
For any $(n+1)\times (n+1)$ matrix $u$ let $\Delta_i(u)$ denote the minor with row set $[1,i]$ and column set $[n+2-i,n+1]$. 
We have a homeomorphism
\begin{eqnarray}\label{e:R03}
\ToeplitzU_{n+1}(\R_{>0})&\longrightarrow & T_{SL_{n+1}}(\R_{>0})
\end{eqnarray}
defined by sending $u\in \ToeplitzU_{n+1}(\R_{>0})$ to the  matrix $\Delta(u)\in T_{SL_{n+1}}(\R_{>0})$ with diagonal entries
\begin{equation}\label{e:IntroDelta}
d_1={\Delta_{1}(u)},\   
d_2=\frac{\Delta_{2}(u)}{\Delta_{1}(u)},\ 
\dotsc,\ d_n=\frac{\Delta_{n}(u)}{\Delta_{{n-1}}(u)},\ d_{n+1}=\frac{1}{\Delta_{n}(u)}.
\end{equation}
\end{theorem}
For an equivalent parametrisation, \eqref{e:R03} is naturally composed with the homeomorphism $T_{SL_{n+1}}(\R_{>0})\overset \sim\longrightarrow \R_{>0}^{n}$ sending the diagonal matrix $(\delta_{ij}d_i)_{i=1}^{n+1}$ to $(\frac{d_{2}}{d_1},\frac{d_3}{d_2},\dotsc,\frac{d_{n+1}}{d_n})$. The composition then gives a homeomorphism
\begin{eqnarray}\label{e:R03-q-param-version}
\mathbf q=(q_1,\dotsc, q_n):\ToeplitzU_{n+1}(\R_{>0})&\longrightarrow & \R_{>0}^{n},
\end{eqnarray}
that can be considered as giving a parametrisation $\ToeplitzU_{n+1}(\R_{>0})$ with $n$ positive parameters. We note that the $q_i$ have an interpretation as quantum parameters \cite{Kos:QCoh} in the small quantum cohomology ring of the flag variety $\mathcal Fl_{n+1}=SL_{n+1}(\C)/B$, which was related to Toeplitz matrices by Dale Peterson \cite{peterson}. This interpretation was an important ingredient in the original proof of the above theorem \cite{rietschJAMS}. 

These two theorems, \cref{t:EdreiIntro} and \cref{t:RealFinite}, appear to be very different. For example the Edrei theorem provides a parametrisation going one way, associating a totally nonnegative Toeplitz matrix to some parameters
\[
\Omega_{S}\longrightarrow \ToeplitzU_{\infty}(\R_{\ge 0}),
\]
while in \cref{t:RealFinite} the parametrisation goes the other way, 
\begin{eqnarray*}
\ToeplitzU_{n+1}(\R_{>0})&\longrightarrow & \R_{>0}^n,
\end{eqnarray*}
associating parameters to a totally positive Toeplitz matrix. We note that the inverse of this latter map is not at all algebraic. The other obvious difference between the two parametrisation theorems is the structure of the parameter spaces. One containing a pair of infinite sequences, the other just a single $\R_{>0}^n$.

The main goal of this paper is to tropicalise both parametrisation theorems, \cref{t:EdreiIntro} and \cref{t:RealFinite}, and to relate them to one another. Along the way, we prove a version of \cref{t:RealFinite} over a field of generalised Puiseaux series. We also relate our tropical parametrisation theorems to Lusztig's parametrisation of the canonical basis. We will now explain our results in some more detail. 

\subsection{} The totally positive part of $U_+$, the unipotent radical of the upper-triangular Borel subgroup of a reductive algebraic group $G$, was described by Lusztig in \cite{Lusztig94} using certain equivalent `positive charts'. He also introduced a tropicalised version of $U_+$ and used it to parametrise his canonical basis of the (Langlands dual) quantized universal enveloping algebra $\mathcal U^\vee_-$. Let us write $U_+(\Rmin)$ for the tropical analogue of $U_+$, where $\Rmin=(\R,+,\min)$ is the `tropical semifield' of real numbers. We consider this tropicalisation to be constructed following \cite{Lusztig94} out of the totally positive part of $U_+$ over some `positive semifield' $\mathcal R_{>0}$ of a ring $\mathcal R$  with semifield homomorphism $\Val:\mathcal R_{>0}\twoheadrightarrow\Rmin$. Namely, we set
\begin{equation}\label{e:IntroTropToep}
U_+(\Rmin)=U_+(\mathcal R_{>0})/\sim,
\end{equation}
with equivalence relation given by the valuation map applied to a/any positive coordinate chart. For a key  example of a positive chart, see \cref{l:mijFormula}. In \cite{Lusztig94}, $\mathcal R_{>0}$ is taken to be the semifield in $\Z((t))$ consisting of Laurent series with positive leading term with its usual valuation, leading to $U_+(\Z_{{\min}})$. We work instead with semifields where the valuation surjects onto $\Rmin$, but otherwise follow the same principle. We note that Lusztig's canonical basis is nevertheless parametrised by elements in $U_+(\Rmin)$, just that he considers the subset with $\Z_{>0}$-coordinates. Our choice for $\mathcal R_{>0}$ in the finite case is  the semifield of generalised Puiseaux  series with positive leading term, see \cref{ex:genPuiseux}. We denote this semifield by $\KK_{>0}$. It has a topology given by the $t$-adic norm. The natural tropicalisation of $\ToeplitzU_{n+1}(\KK_{>0})$ is  then simply the subset of $U_+(\Rmin)$ described by
\[
\ToeplitzU_{n+1}(\Rmin)=\ToeplitzU_{n+1}(\KK_{>0})/\sim.
\]
 Recall the definition of the minors $\Delta_i$ from \cref{t:RealFinite}. Our first result is an analogue of \cref{t:RealFinite} with $\R_{>0}$ replaced by $\KK_{>0}$. However, while $\R_{>0}\subset \KK_{>0}$, the induced topology on $\R_{>0}$ is the discrete topology, so that \cref{t:RealFinite} is in fact not fully a special case of \cref{t:IntroPuiseauxFinite}. 
\begin{mainthm}\label{t:IntroPuiseauxFinite}
The map $
\Delta_{>0}:\ToeplitzU_{n+1}(\KK_{>0})\longrightarrow T_{SL_{n+1}}(\KK_{>0})$
defined by sending $u\in\ToeplitzU_{n+1}(\KK_{>0})$ to the diagonal matrix $d$ with diagonal entries
\[
d_1= \Delta_{1}(u),\   
d_2=\frac{\Delta_{2}(u)}{\Delta_{1}(u)},\ 
\dotsc,
\ d_{n+1}=\frac {1}{\Delta_{n}(u)}
\]
is a bijection. Moreover, it is a homeomorphism for the $t$-adic topology.
\end{mainthm}
A precursor to this theorem is \cite[Theorem~5.1]{Judd:Flag}, which shows that the fibre of $t^\lambda\in T_{SL_{n+1}}(\KK_{>0})$ for a dominant coweight $\lambda$ is a singleton. The proof of \cref{t:IntroPuiseauxFinite} 
is inspired by the second proof of \cref{t:RealFinite}, from \cite{rietschNagoya}, which uses mirror symmetry to describe Toeplitz matrices as critical points of a `superpotential' associated to the type $A_n$ flag variety. We prove \cref{t:IntroPuiseauxFinite} via this method, making use of a positive critical point theorem over $\KK_{>0}$ given in \cite{JuddRietsch:24}.

Our next result is the description of the tropicalisation $\ToeplitzU_{n+1}(\Rmin)$ and the tropicalised version of the parametrisation theorem.  
Consider the positive charts of $U_+$  given by multiplying together simple root subgroups 
according to a reduced expression $\mathbf i$ of the longest element $w_0$ of the Weyl group $W=S_{n+1}$, see \cref{s:U+poscharts}. We identify the reduced expression $\mathbf i$ with its associated normal ordering of the positive roots $R_+=\{\alpha_{k\ell}=\ep_k-\ep_\ell\mid 1\le k<\ell\le n+1\}$ as in  \eqref{e:posrootsordering}. 
Thus for any reduced expression $\mathbf i$ we obtain a  coordinate chart with  coordinates $(M^{\mathbf i}_\alpha)_{\alpha\in R_+}\in\R^{R_+}$ on $U_+(\Rmin)$. Coordinate changes are piecewise-linear maps $\psi_{\mathbf i,\mathbf i'}:\R^{R_+}\to \R^{R_+} $. 

Let us now write $\mathfrak h_{\R,SL_{n+1}}$ for the diagonal Cartan subalgebra of $\mathfrak sl_{n+1}(\R)$ and make the identification  $\mathfrak h_{\R,SL_{n+1}}=\mathfrak h^*_{\R,PSL_{n+1}}$.
Lusztig in \cite{Lusztig94} showed that the following map 
\begin{equation}\label{e:LusztigWeight}
\begin{array}{cccc}
\mathcal L:&U_+(\Rmin) &\longrightarrow & \mathfrak h_{\R,PSL_{n+1}}^*\\
& (M^{\mathbf i}_{\alpha})_{\alpha\in R_+}&\mapsto &\sum_{\alpha\in\R_+} M^{\mathbf i}_\alpha\alpha,
\end{array}
\end{equation}
 which he used to describe the weights of canonical basis elements, is well-defined. We call it the Lusztig weight map. We show in \cref{l:dunilemma} that $\mathcal L$ is the tropicalisation of the map $\Delta_{>0}:U_+(\R_{>0})\to T_{SL_{n+1}}(\R_{>0})$ given by the formula \eqref{e:IntroDelta} in our parametrisation theorem.  The next theorem contains a combination of work from \cite{Judd:Flag} and \cite{Ludenbach} in strengthened form, making use of \cref{t:IntroPuiseauxFinite}. It is proved in   \cref{t:MinIdealIsToepl} and \cref{t:TropFinite}.
\begin{mainthm}\label{t:IntroTropFinite}
The subset $\ToeplitzU_{n+1}(\Rmin)$ of $U_+(\Rmin)$ is described as follows. For any element of $\ToeplitzU_{n+1}(\Rmin)$ its coordinates $(M_{\alpha}^{\mathbf i})_{\alpha\in R_+}$ are independent of $\mathbf i$, so that we may write it as $(M_{\alpha})_{\alpha\in R_+}$ without specifying a reduced expression, and we have
\[
\ToeplitzU_{n+1}(\Rmin)=\{(M_\alpha)_{\alpha\in R_+}\mid M_{\gamma}=\min(M_{\alpha},M_{\beta}) \text{ whenever $\gamma=\alpha+\beta$ in $R_+$}\}.
\]
Moreover, the restriction of the Lusztig weight map to $\ToeplitzU_{n+1}(\Rmin)$ 
\[
\mathcal L|_{\ToeplitzU_{n+1}(\Rmin)}: \ToeplitzU_{n+1}(\Rmin) \longrightarrow \mathfrak h^*_{\R,PSL_{n+1}}.
\]
is a bijection.
\end{mainthm}
We have formulated this theorem here in a way that it can be generalised beyond type $A$. 
In type $A$ we will prefer the following explicit description of the elements of $\ToeplitzU_{n+1}(\Rmin)$. Let us arrange the coordinates $(M_{\alpha})_{\alpha\in R_+}$ of an element of $\ToeplitzU_{n+1}(\Rmin)$ into a lower-triangular array and re-index them. For example if $n=3$ then  
\begin{equation*}
\begin{tikzpicture}[scale=0.6]
\node[fill=white] at    (0.5,2.55) {$M_{\alpha_{34}}$}  ;
\node[fill=white] at    (0.5,1.55) {$M_{\alpha_{24}}$}  ;
\node[fill=white] at    (0.5,0.55) {$M_{\alpha_{14}}$}  ;
\node[fill=white] at    (1.95,1.55) {$M_{\alpha_{23}}$}  ;
\node[fill=white] at    (4.95,1.55) {$=$}  ;
\node[fill=white] at    (1.95,0.55) {$M_{\alpha_{13}}$}  ;
\node[fill=white] at    (3.4,0.55) {$M_{\alpha_{12}}$}  ;
\draw[black] (-.2,0.1) -- (-.2,3.1) -- (1.25,3.1) -- (1.25,0.1) -- (4.1,0.1) -- (4.1,1.1)--(-.2,1.1)--(-.2,0.1)--(2.7,0.1)--(2.7,2.1)--(-.2,2.1);
 \end{tikzpicture}\quad
 \begin{tikzpicture}[scale=0.6]
\node[fill=white] at    (0.5,2.55) {$M_{31}$}  ;
\node[fill=white] at    (0.5,1.55) {$M_{21}$}  ;
\node[fill=white] at    (0.5,0.55) {$M_{11}$}  ;
\node[fill=white] at    (1.95,1.55) {$M_{22}$}  ;
\node[fill=white] at    (1.95,0.55) {$M_{12}$}  ;
\node[fill=white] at    (3.4,0.55) {$M_{13}$}  ;
\draw[black] (-.2,0.1) -- (-.2,3.1) -- (1.25,3.1) -- (1.25,0.1) -- (4.1,0.1) -- (4.1,1.1)--(-.2,1.1)--(-.2,0.1)--(2.7,0.1)--(2.7,2.1)--(-.2,2.1);
 \end{tikzpicture}.
\end{equation*}
Then the condition from \cref{t:IntroTropFinite} for $(M_{\alpha})_{\alpha\in R_+}$ to lie in $\ToeplitzU_{n+1}(\Rmin)$ is equivalent to the condition  
\[
M_{ij}=\min(M_{i+1,j},M_{i,j+1})\quad \text{for all $\ i+j\le n$.}
\]
We call such an array a \textit{min-ideal filling} of rank $n$, and we use the notation $\Imin_{n+1}$ for the rank $n$ min-ideal fillings, see \cref{d:minideal}. The first part of \cref{t:IntroTropFinite} translates into the statement that $\ToeplitzU_{n+1}(\Rmin)=\Imin_{n+1}$, and that this is a well-defined subset of $U_+(\Rmin)$. A version of this is nearly proved in \cite{Ludenbach}. Modulo this first half, much of the second half can be deduced from \cite{Judd:Flag}. In completing the proof of \cref{t:IntroTropFinite}, we have made use of \cite{Judd:Flag} for injectivity, but otherwise give new direct proofs. The relationship with these works is also explained in greater detail in \cref{r:JuddLudenbach}.   

A key ingredient of our proof of the equality $\ToeplitzU(\Rmin)=\Imin_{n+1}$ is a particular preferred positive chart for $U_+$, whose coordinates $(m_{ij})_{i+j\le n+1}$ we call the \textit{standard coordinates}. We prove the following explicit description  of totally positive Toeplitz matrices in terms of these coordinates. This description uses a quiver $Q_{n+1}$ that we construct in \cref{d:Qn+1}, see  \cref{f:Qn+1} for an illustration, and it applies to $\ToeplitzU_{n+1}(\mathcal R_{>0})$ over an arbitrary `positive semifield' $\mathcal R_{>0}$ lying inside a ring $\RR$.    
\begin{mainthm}\label{t:IntroQfinite} For the quiver $Q_{n+1}$, label the vertices $v_{ij}$ by the  expressions $\sum \frac{1}{m_{i-\ell,j-\ell}}$ in standard coordinates, and the arrows by the label at the head minus the label at the tail. Then 
\[
\ToeplitzU_{n+1}(\mathcal R_{>0})=\left\{(m_{ij})_{i,j}\in \mathcal R_{>0}^{\{i+j\le n+1\}}\left |\begin{array}{l} \text{$\prod$ incoming arrow labels $=\ \prod$ outgoing arrow labels,}\\
\text{for all vertices $v_{ij}$ with $1< i+j< n$.}
\end{array}\right.\right\}
\]
Moreover, for $(m_{ij})_{i,j}\in\ToeplitzU_{n+1}(\mathcal R_{>0})$, the arrow labels are all in $\mathcal R_{>0}$. 
\end{mainthm} 
This result is related to an earlier description of Toeplitz matrices from \cite{rietsch} via \cite[Theorem 4.3.1]{Ludenbach}, see also \cite[Section~2.5]{rietschICM}. We give a direct proof of \cref{t:IntroQfinite} in \cref{a:ToeplitzviaQ}. 

\subsection{}\label{s:IntroGeneralType}
\label{r:conjecture} 
Unipotent Toeplitz matrices  can be described  Lie theoretically as elements of the stabiliser in $U_+$ of the principal nilpotent $E=(\delta_{i,j-1})_{i,j}$ under conjugation. This generalises to any reductive algebraic group $G$ in a standard way, whereby $\ToeplitzU_G$ is the stabilizer inside the analogous unipotent subgroup $U^G_{+}$ of an analogous principal nilpotent $E$ in $\mathfrak g^*$, e.g.  \cite{Kostant:Toda,peterson}.  
We define
\begin{equation}\label{e:ConjGenTropToep}
\Imin_G:=\{(M_\alpha)_{\alpha\in R^G_+}\in \R^{R^G_+}\mid M_{\gamma}=\min(M_{\alpha},M_{\beta}) \text{ whenever $\gamma=\alpha+\beta$ in $R^G_+$}\},
\end{equation}
where $R^G_+$ is the set of positive roots for $G$. The elements of $\Imin_G$ are  fixed under the now general type piecewise-linear coordinate change maps $\Psi^G_{\mathbf i,\mathbf i'}:\R^{R^G_+}\to \R^{R^G_+}$. Namely, this invariance follows by tropicalisation from \cite[Lemma 5.2.(2)]{R:IMRN}. Thus $\Imin_G$ is a well-defined subset of $U^G_+(\Rmin)$. 
\begin{conj}\label{c:IntroConj}
The subset $\Imin_G$ of $U_+(\Rmin)$ is the tropicalisation of $\ToeplitzU_G$,
\[\ToeplitzU_G(\Rmin)=\Imin_G.\]
\end{conj} 
We note that we expect that \cref{t:IntroPuiseauxFinite} and \cref{t:IntroTropFinite} can be generalised to this setting following the arguments given in \cref{s:ToeplitzFinite} and \cref{s:TropToeplitzFinite}, assuming  \cref{c:IntroConj} and with \cref{t:RealFinite}  replaced by its general type analogue  proved in  \cite{LamRietsch}.

\subsection{} We now turn to the infinite setting. Consider the following projective limit of the $\Imin_{n+1}$,
\begin{center}
 \begin{tikzpicture}[scale=0.6]
 \node[fill=white] at    (-3,1.6) {$\Imin_{\infty}(\Rmin):=$}  ;
  \draw[decoration={brace,raise=5pt},decorate]
  (-.3,-0.2) -- node[right=6pt] {} (-.3,3.6);
 \draw[decoration={brace,mirror,raise=5pt},decorate]
  (13,-0.2) -- node[right=6pt] {} (13,3.6);
\node[fill=white] at    (0.5,2.6) {$M_{31}$}  ;
\node[fill=white] at    (0.5,1.6) {$M_{21}$}  ;
\node[fill=white] at    (0.5,0.6) {$M_{11}$}  ;
\node[fill=white] at    (1.8,1.6) {$M_{22}$}  ;\node[fill=white] at    (1.8,0.6) {$M_{12}$}  ;
\node[fill=white] at    (3,0.6) {$M_{13}$}  ;
\node[fill=white] at    (4.4,0.6) {$\cdots$}  ;
\node[fill=white] at    (14,0) {$.$}  ;
\node[fill=white] at   (9,1) {$M_{ij}=\min(M_{i,j+1},M_{i+1,j})$};
\node[fill=white] at   (9,2) {$\forall  i,j\in\N, \ \  M_{ij}\in\R\ \text{ and}$};
\draw[black] (5,-0.2) -- (5,3.5);
\draw[black] (-.2,0.1) -- (-.2,3.55) --(-.2,3.1)--(1.35,3.1)-- (1.15,3.1) --(1.15,3.3) -- (1.15,0.1) -- (4,0.1)--(3.6,0.1) -- (3.6,1.3)--(3.6,1.1)--(3.8,1.1)--(3.6,1.1)--(-.2,1.1)--(-.2,0.1)--(2.4,0.1)--(2.4,2.3)--(2.4,2.1)--(2.6,2.1)--(2.4,2.1)--(-.2,2.1);
 \end{tikzpicture}
 \end{center} 
This is a basic infinite analogue of min-ideal fillings. We also introduce the following more restrictive infinite limits of the $\Imin_{n+1}$. The first consists of infinite min-ideal fillings for which each row and column has a supremum in $\R$. We call these min-ideal fillings `asymptotically real' and write
\[
\Imin^{\R}_{\infty}(\Rmin)=\{(M_{ij})\in\Imin_{\infty}(\Rmin)\mid \text{$\sup\{M_{ik}\mid k\in\N\}$ and $\sup\{M_{kj}\mid k\in\N\}<\infty$ for every $i,j\in\N$}\}.
\]
The second consists of the infinite min-ideal fillings for which every row and column has a maximum. We call these `stable' and write 
\begin{equation*}\begin{array}{cl}
\Imin^s_\infty(\Rmin)&:=\left\{(M_{ij})_{i,j\in\N}\in\Imin_\infty(\Rmin)\mid \text{$(M_{ik})_k$ and $(M_{kj})_k$ stabilise as $k\to\infty$}\right\}.
\end{array}
\end{equation*}
We then introduce tropical parameter spaces starting with the basic one,
\begin{equation}
\Omega_\star(\Rmin)=\left\{(\mathbf A,\mathbf B)\in(\R\cup\{\infty\})^\N\times(\R\cup\{\infty\})^\N\left|\begin{array} {l}
 \text{$\mathbf A=(A_i)_i, \mathbf B=(B_j)_j$, weakly increasing,} \\
\text{$\min(A_i,B_i)\in\R$ for all $i\in\N$,}\\
\sup(\{A_i\})=\sup(\{B_j\})
\end{array}\right.
\right\}.
\end{equation}
Here we refer to weakly increasing sequences $\bigA$ and $\bigB$ with $\sup(\{A_i\})=\sup(\{B_j\})$ as `weakly interlacing'. Within $\Omega_\star(\Rmin)$ we consider the subset whose sequences are real,
\begin{equation}
\Omega^{\R}_\star(\Rmin)=\left\{(\mathbf A,\mathbf B)\in \Omega_\star(\Rmin)\mid A_i,B_i\in\R\ \text{ for all $i\in\N$ }
\right\},
\end{equation}
and an even more restrictive subset satisfying a condition that we call \textit{interlacing},
\begin{equation}
\Omega^{il}_\star(\Rmin)=\left\{(\mathbf A,\mathbf B)\in \Omega_\star(\Rmin)\left|\begin{array} {l}
 \text{for any $i\in\N$ there exists a $j$ such that $B_j\ge A_i$} \\
 \text{for any $j\in\N$ there exists an $i$ such that $A_i\ge B_j$}\end{array}\right.
\right\}.
\end{equation}
The following theorem is proved in \cref{s:TropEdrei}.
\begin{mainthm}\label{t:IntroTropEdrei}
For any pair $\bigA,\bigB$ of weakly increasing sequences in $\R\cup\{\infty\}$ with $\min(A_i,B_i)\in\R$ for all~$i$, the element $(M_{ij})\in\R^{\N\times\N}$ defined by $M_{ij}=\min(A_i,B_j)$ is an infinite min-ideal filling. We have the following commutative  diagram of maps,  
\begin{equation*}
\begin{tikzcd}
	{\Omega_\star(\Rmin)} && {\Imin_{\infty}(\Rmin)} \\
	{\Omega_\star^\R(\Rmin)} && {\Imin_{\infty}^\R(\Rmin)} \\
	{\Omega^{il}_\star(\Rmin)} && {\Imin^s_\infty}(\Rmin),
	\arrow["{\mathbb E}", from=1-1, to=1-3]
	\arrow[hook,from=2-1, to=1-1]
	\arrow[hook, from=2-3, to=1-3]
	\arrow["{\mathbb E^\R}", from=2-1, to=2-3]
	\arrow[hook,from=3-1, to=2-1]
	\arrow[hook, from=3-3, to=2-3]
	\arrow["{\mathbb E^s}", from=3-1, to=3-3]
\end{tikzcd} 
\end{equation*}
where  the horizontal maps are all bijections and given by $(\bigA,\bigB)\mapsto (\min(A_i,B_j))_{i,j\in\N}$. 
\end{mainthm}
 Our next result relates the parameters $(\bigA,\bigB)$ for asymptotically real min-ideal fillings from  \cref{t:IntroTropEdrei} to the parameters from the parametrisation of $\ToeplitzU_{n+1}(\Rmin)$ in \cref{t:IntroTropFinite}. We also show that the asymptotically real condition cannot be removed.  This theorem is proved in  \cref{s:VK}
\begin{mainthm}\label{t:IntroVKtyp}
Let $M^{(\infty)}=(M_{ij})_{i,j\in\N}\in\Imin^{\R}_\infty(\Rmin)$ and $M^{(n+1)}=(M_{ij})_{i+j\le n+1}$ its projection to $\Imin_{n+1}$. Suppose
\[
\mathcal L^{(n+1)}(M^{(n+1)})=\begin{pmatrix}
\lambda^{(n+1)}_1 & &&&\\
&\lambda^{(n+1)}_2 & &&\\
&&\ddots &&\\
&&&\lambda^{(n+1)}_n & \\
&&&&\lambda^{(n+1)}_{n+1}
\end{pmatrix}
\]
denotes the Lusztig weight of $M^{(n+1)}$. Then the limits
\begin{equation}\label{e:Introlimits}
\lim_{n\to\infty}\frac{\lambda^{(n+1)}_i}n=A_i\qquad \text{and}\qquad \lim_{n\to\infty}\frac{-\lambda^{(n+1)}_{n+1-j}}{n}=B_j,
\end{equation}
exist in $\R$ and together determine an element $(\bigA,\bigB)=((A_i)_i,(B_j)_j)\in\Omega^{\R}_\star(\Rmin)$. Moreover, we have that $\mathbb E^\R(\bigA,\bigB)=M^{(\infty)}$.
\end{mainthm}

\begin{remark}\label{r:VK}In what follows we will be concerned with detropicalising the parameters $(\bigA,\bigB)\in\Omega_\star(\Rmin)$ and relating them to the Schoenberg parameters $(\boldalpha,\boldbeta)$ from \cref{t:EdreiIntro}. Therefore \cref{t:IntroVKtyp} has the flavour of an asymptotic reconstruction of Schoenberg parameters. In connection with this we mention a very different work due to Vershik and Kerov  \cite{VershikKerovAsymptotic} that has a similar flavour. Namely, in Thoma's theorem \cite{Thoma} a normalised version of $\ToeplitzU_\infty(\R_{>0})$ parametrises irreducible characters of the infinite symmetric group~$S_\infty$, and therefore another finite analogue for infinite totally positive Toeplitz matrices is given by the irreducible characters of $S_n$. In this setting, Vershik and Kerov proved an asymptotic interpretation of the classical Schoenberg parameters $(\boldalpha,\boldbeta)$ using normalised arm and leg-lengths of Young diagrams corresponding to characters of~$S_n$. This gave the   inspiration to look for an  asymptotic expression of our Schoenberg style parameters $(\bigA,\bigB)$ in terms of the finite parameters that led to \cref{t:IntroVKtyp}.
\end{remark}
 \subsection{} 
We now begin describing some subtle results concerning infinite totally positive Toeplitz matrices over different valuative semifields with the goal of `detropicalising' \cref{t:IntroTropEdrei}. 
For now let $\mathcal R$ be any ring containing a `positive' semifield $\mathcal R_{>0}$ together with a surjective semifield homomorphism $\Val:\mathcal R_{>0}\to\Rmin$ extending continuously to  $\Val:\mathcal R_{\ge 0}\to\R\cup\{\infty\}$, see \cref{d:semifieldvaluation}.
 One example is still $(\KK,\KK_{>0},\Val)$ with the usual $t$-adic topology on $\KK_{\ge 0}$. 

The standard coordinates we used in  \cref{t:IntroQfinite} were carefully constructed to extend to the projective limit, so that we again have  standard coordinates  $(m_{ij})_{i,j\in\N}$ for $U_+^{(\infty)}$. Moreover, the description of $\ToeplitzU_{n+1}(\RR_{>0})$ from  \cref{t:IntroQfinite} extends  to a description of $\ToeplitzU_{\infty}(\RR_{>0})$, now via an infinite quiver $Q_\infty$, see \cref{f:InfQuiverExample} and \cref{c:infQcoords}. From here on, we describe totally positive Toeplitz matrices in terms of these standard coordinates.

We consider a straightforward infinite analogue of \eqref{e:IntroTropToep}, $
\ToeplitzU_{\infty,\mathcal R_{>0}}(\Rmin):=\ToeplitzU_{\infty}(\mathcal R_{>0})/\sim$ to be the tropicalisation of $\ToeplitzU_{\infty}(\RR_{>0})$. Using valuations of standard coordinates and \cref{t:IntroTropFinite} we have
\[
\ToeplitzU_{\infty,\mathcal R_{>0}}(\Rmin)\subseteq \Imin_{\infty}(\Rmin),
\]
see \cref{l:infminideal}.
Our first main detropicalisation result will be about relating the tropical parameters from \cref{t:IntroTropEdrei} to actual Schoenberg-type parameters (but over $\mathcal R_{>0}$). We define the following parameter space. Here $s_{i\times j}$ denotes the Schur function for the partition $(j,\dotsc,j)$ with $i$ parts. 
 \begin{equation}\label{e:IntroOmega} \Omega(\mathcal R_{>0}):=\left\{(\boldalpha,\boldbeta)\in\mathcal R_{\ge 0}^\N\times \mathcal R_{\ge 0}^\N\left |\begin{array}{l} \boldalpha_1\ge\boldalpha_2\ge\boldalpha_3\ge\dotsc, \\\boldbeta_1\ge\boldbeta_2\ge\boldbeta_3\ge \dotsc, \\\boldbeta_i+\boldalpha_i\ne 0 \text{ for all $i\in\N,$}\\  \text{$s_{i\times j}(\boldalpha), s_{i\times j}(\boldbeta)$ are well-defined (converge in $\mathcal R_{>0}$) for all $i,j\in\N$}\end{array}\right.\right\}.
 \end{equation} 

Corresponding to $\Omega(\mathcal R_{>0})$ we also  introduce  
\[\Omega(\Rmin)=\{(\mathbf A,\mathbf B)\in (\R\cup\{\infty\})^\N\times (\R\cup\{\infty\})^\N\mid \mathbf A,\mathbf B \text{ weakly increasing, $\min(A_i,B_i)\in\R$}\},
\] 
so that coordinatewise valuation defines a map $\Val: \Omega(\mathcal R_{>0})\to\Omega(\Rmin)$. Note that $\Omega(\Rmin)$ is bigger than the tropical parameter spaces from \cref{t:IntroTropEdrei}. We now have the following theorem in this general setting, which is part of \cref{t:restrtropparam} and a first step towards a detropicalisation of \cref{t:TropEdrei}. A key ingredient of \cref{t:IntroGenR} is the theory of supersymmetric functions, in particular work of  \cite{Macdonald:variations,GouldenGreene94}.
\begin{mainthm}\label{t:IntroGenR}
There is a well-defined map $\mathcal T:{\Omega(\RR_{>0})} \to {\ToeplitzU_\infty(\RR_{>0})}$   that sends $(\boldalpha,\boldbeta)$ to the totally positive Toeplitz matrix with generating function 
\[
1+\sum_{k=1}^\infty \mathbf c_k x^k=\frac{\prod_{j=1}^{\infty}(1+\boldbeta_j x)}{\prod_{i=1}^{\infty}(1-\boldalpha_i x)}.
\]
Moreover, if $(\mathbf m_{ij}(\boldalpha,\boldbeta))$ are the standard coordinates of $\mathcal T(\boldalpha,\boldbeta)$, then $\Val(\mathbf m_{ij}(\boldalpha,\boldbeta))=\min(A_i, B_j)$, 
in terms of the valuations $A_i=\Val(\boldalpha_i)$ and $B_j=\Val(\boldbeta_j)$.
In other words, we have a commutative diagram
\begin{equation}\label{e:IntroOmVal}
\begin{tikzcd}
	{(\boldalpha,\boldbeta)} & {\Omega(\RR_{>0})} && {\ToeplitzU_\infty(\RR_{>0})} & {(m_{ij})_{i,j}} \\
	{(\bigA,\bigB)} & {\Omega(\Rmin)} && {\Imin_\infty}(\Rmin) & {(M_{ij})_{i,j}}.
	\arrow[maps to, from=1-1, to=2-1]
	\arrow["\mathcal T", from=1-2, to=1-4]
	\arrow["\Val"', from=1-2, to=2-2]
	\arrow["\Val"',from=1-4, to=2-4]
	\arrow[maps to, from=1-5, to=2-5]
	\arrow["\widetilde{\mathbb E}", from=2-2, to=2-4]
\end{tikzcd}
\end{equation}
 where the map $\widetilde{\mathbb E}$ maps $(\bigA,\bigB)$ to $(M_{ij})_{i,j}=(\min(A_i,B_j))_{i,j}$. 
\end{mainthm}

Note that $\mathcal T$ will be injective but generally not surjective. On the other hand the map $\widetilde{\mathbb E}$ is surjective but not injective, compare  \cref{t:IntroTropEdrei}. We nevertheless consider $\mathcal T$ to be a `detropicalisation' of $\widetilde{\mathbb E}$ as long as both of the vertical maps are surjective. Namely in this case each identity $\widetilde{\mathbb E}(\bigA,\bigB)=(M_{ij})_{i,j}$ will have a lifting to an identity $\mathcal T((\boldalpha,\boldbeta))=(m_{ij})_{i,j}$. Note that this surjectivity may fail depending on $\mathcal R_{>0}$. For example, for $\mathcal R_{>0}=\KK_{>0}$ we show that 
$\Val(\Omega(\KK_{>0}))=\Omega^{div}_\star(\Rmin)$, consisting pairs $(\bigA,\bigB)\in\Omega_\star(\Rmin)$ where both sequences have supremum equal to $\infty$, and we even have a strict inclusion of 
$\widetilde{\mathbb E}(\Val(\Omega(\KK_{>0})))$ in $\Val(\ToeplitzU_\infty(\KK_{>0}))$. See \cref{p:OmegaDivInterp}.

We now introduce a new valued semifield $\mathcal C_{>0}$ that has better properties than $\KK_{>0}$ in the infinite setting. Namely, pick some $\delta>0$ and consider for $\mathcal R$ the ring $\mathcal C:=C^0((0,\delta])$ of continuous functions. The positive semifield $\Kdelpos$ inside $\mathcal C$ is defined as the subset of $\R_{>0}$-valued functions with the additional property that 
\[
\lim_{t\to 0}t^{-F}f(t)\in\R_{>0}\quad \text{for some $F\in\R$,}
\]
which we then take to be the valuation, $\Val(f):=F$. 
Our first result involving the semifield $\Kdelpos$ is  that all min-ideal fillings can be realised over it. 
\begin{mainthm}
We have the equality 
\[
\ToeplitzU_{\infty,\Kdelpos}(\Rmin)= \Imin_{\infty}(\Rmin).
\]
\end{mainthm}
This is proved in \cref{t:KdelTrop}. The proof of this theorem involves defining a topology on $\Kdelnn$ and showing that the map $\Val:\Omega(\Kdelpos)\to\Omega(\Rmin)$ is surjective. This result then also implies that over $\Kdelpos$ the map $\mathcal T:\Omega(\Kdelpos)\to \ToeplitzU_{\infty}(\Kdelpos)$ from \cref{t:IntroGenR} becomes a detropicalisation of $\widetilde {\mathbb E}:\Omega(\Rmin)\to \Imin_{\infty}(\Rmin)$.

Note that \cref{t:IntroGenR} still does not explain the additional conditions on our tropical parameter spaces in the `tropical Edrei theorem', \cref{t:IntroTropEdrei}. Our final result will be a detropicalisation of the bijections $\mathbb E$ and $\mathbb E^s$  from that theorem. It will involve constructing an alternative infinite analogue of $\ToeplitzU_{n+1}(\RR_{>0})$.

\subsection{} Recall that we assume $\mathcal R_{\ge 0}=\mathcal R_{>0}\cup\{0\}$ to be a topological semiring inside a ring $\RR$. We use standard coordinates to define the following set of what we call \textit{restricted} infinite Toeplitz matrices,
\begin{equation}\label{e:IntroToepRes}
\ToeplitzU^{res}_{\infty}(\mathcal R_{>0}):=\left\{u\in\ToeplitzU_\infty(\mathcal R_{>0})\left|\begin{array}{l} \text{for the standard coordinates $(m_{ij})_{i,j\in\N}$ of $u$,}\\
\exists\ \lim_{k\to\infty}m_{ik}\in\mathcal R_{\ge 0}\ \text{and}\ \lim_{k\to\infty}m_{kj}\in\mathcal R_{\ge 0}
\end{array}\right.\right\}.
\end{equation}
We also define an associated parameter space $\Omega^{res}(\mathcal R_{>0})$, as the inverse image of $\ToeplitzU^{res}_{\infty}(\mathcal R_{>0})$ under $\mathcal T$ from \cref{t:IntroGenR}. For $\R_{\ge 0}$ it turns out that $\ToeplitzU^{res}_{\infty}(\R_{\ge 0})=\ToeplitzU_{\infty}(\R_{\ge 0})$ as follows from \cite{Edrei53}. Namely, in an alternative proof of \cref{t:EdreiIntro}, Edrei constructed the Schoenberg parameters as limits of combinations of minors that turn out to agree with the standard coordinates. Therefore \eqref{e:IntroToepRes} is an equally valid generalisation of $\ToeplitzU_{\infty}(\R_{\ge 0})$.

We introduce two topologies on $\mathcal C_{\ge 0}=\mathcal C_{\ge 0}\cup\{0\}$ that we call the weak and strong topologies, see \cref{d:Kdeltop} and \cref{d:topCnn}, and we consider the restriction of $\mathcal T$,
\[
\mathcal T^{res}:\Omega^{res}(\mathcal R_{>0})\to 
\ToeplitzU^{res}_{\infty}(\mathcal R_{>0}),
\] 
 for $\mathcal R_{>0}=\Kdelposwk$ and $\mathcal R_{>0}=\Kdelposst$. 
We then have restricted tropicalisations,
$\ToeplitzU^{res}_{\infty,\RR_{>0}}(\Rmin):=\ToeplitzU^{res}_{\infty,\RR_{>0}}(\RR_{>0})/\sim$, that depend also on the chosen topology. This leads to the final result. 
\begin{mainthm}\label{t:IntroRestrC}
For $\RR_{> 0}=\Kdelposwk$ and $\Kdelposst$ the diagram \eqref{e:IntroOmVal} from \cref{t:IntroGenR} restricts in the following ways,
\begin{equation}\label{e:IntrorestrOmVal}
\begin{tikzcd}
	{\Omega^{res}(\Kdelposwk)} && {\ToeplitzU^{res}_\infty(\Kdelposwk)} \\
	{\Omega_\star(\Rmin)} && {\Imin_\infty}(\Rmin),
	\arrow["{\mathcal T^{res}}", from=1-1, to=1-3]
	\arrow[two heads,from=1-1, to=2-1]
	\arrow[two heads, from=1-3, to=2-3]
	\arrow["{\mathbb E}", from=2-1, to=2-3]
\end{tikzcd} \qquad \qquad
\begin{tikzcd}
	{\Omega^{res}(\Kdelposst)} && {\ToeplitzU^{res}_\infty(\Kdelposst)} \\
	{\Omega_\star^{il}(\Rmin)} && {\Imin^{s}_\infty}(\Rmin),
	\arrow["{\mathcal T^{res}}", from=1-1, to=1-3]
	\arrow[two heads, from=1-1, to=2-1]
	\arrow[two heads, from=1-3, to=2-3]
	\arrow["{\mathbb E^s}", from=2-1, to=2-3]
\end{tikzcd}
\end{equation}
where the restrictions $\mathbb E$ and $\mathbb E^s$ of $\widetilde{\mathbb E}$ are bijections by  \cref{t:IntroTropEdrei} and the vertical maps are surjective.
\end{mainthm}

Finally, we note that we have  natural inclusions of $\R_{>0}$ into $\Kdelposwk$ and $\Kdelposst$ as  constant functions, and these are topological semifield embeddings, as is shown in \cref{s:Kdel}. Our map 
$\Omega^{res}(\Kdelposwk)\to\ToeplitzU_{\infty}^{res}(\Kdelposwk)$ and its $\Kdelposst$ version naturally  `extend' the open dense parametrisation 
\[
\Omega_S^\circ(\R_{>0})=\{(0,\boldalpha,\boldbeta)\in\Omega_{S}\mid \alpha_i,\beta_i\ne 0 \}\longrightarrow\ToeplitzU_{\infty}(\R_{>0})
\]
coming from the original Edrei theorem, \cref{t:EdreiIntro}. At the same time they tropicalise to recover the parametrisation maps $\mathbb E$
and $\mathbb E^s$ from \cref{t:IntroTropEdrei}. This is summed up in the commutative diagrams, 
\begin{equation}\label{e:IntroEnding}
\begin{tikzcd}
	{\Omega_S^\circ(\R_{>0})} && {\ToeplitzU_\infty(\R_{>0})} \\
	{\Omega^{res}(\Kdelposwk)} && {\ToeplitzU^{res}_\infty(\Kdelposwk)} \\
	{\Omega_\star(\Rmin)} && {\Imin_\infty}(\Rmin),
	\arrow["{\mathcal T_{S}}", from=1-1, to=1-3]
	\arrow[hook,from=1-1, to=2-1]
	\arrow[hook, from=1-3, to=2-3]
	\arrow["{\mathcal T^{res}}", from=2-1, to=2-3]
	\arrow[two heads,from=2-1, to=3-1]
	\arrow[two heads, from=2-3, to=3-3]
	\arrow["{\mathbb E}", from=3-1, to=3-3]
\end{tikzcd} \qquad \qquad
\begin{tikzcd}
	{\Omega_S^\circ(\R_{>0})} && {\ToeplitzU_\infty(\R_{>0})} \\
	{\Omega^{res}(\Kdelposst)} && {\ToeplitzU^{res}_\infty(\Kdelposst)} \\
	{\Omega^{il}_\star(\Rmin)} && {\Imin^s_\infty}(\Rmin).
	\arrow["{\mathcal T_{S}}", from=1-1, to=1-3]
	\arrow[hook,from=1-1, to=2-1]
	\arrow[hook, from=1-3, to=2-3]
	\arrow["{\mathcal T^{res}}", from=2-1, to=2-3]
	\arrow[two heads,from=2-1, to=3-1]
	\arrow[two heads, from=2-3, to=3-3]
	\arrow["{\mathbb E^s}", from=3-1, to=3-3]
\end{tikzcd} 
 \end{equation}

\subsection*{Acknowledgements} I thank Alexei Borodin, Nicholas Shepherd-Barron and Lauren Williams for helpful conversations, as well as my former students Jamie Judd and Teresa L\"udenbach whose beautiful work this paper is building on. I also thank George Lusztig, whose work on total positivity has been very inspiring to me, and  Dale Peterson, thanks to whom I became interested in Toeplitz matrices.

\section{Notation and Setup}\label{s:Notation}
\subsection {Positive structure on a ring}\label{s:Semifields} Let $(\S,+,\cdot)$ be a set with abelian semigroup operations of `addition' and `multiplication' that together satisfy distributivity. We assume that multiplication has a unit element, denoted $1_\S$.  We call such an object $(\S,+,\cdot )$ a (commutative) \textit{semiring}. Since we always assume commutativity we will leave out this adjective.  If every element has a multiplicative inverse then we will call $\S$ a \textit{semifield}. These definitions are following \cite{BFZ}. Note that a semifield cannot include an additive identity element $0_\S$. For a semiring $\S$ with $0_\S$ element, we call $\S$ a \textit{semifield with $0$} if every element \textit{but}~$0_\S$ is invertible for multiplication. We may extend any semifield to a \textit{semifield with $0$} simply by adding an additive identity element $0_\S$. 

Let $\RR$ be a commutative ring of characteristic~$0$. We denote by $\RR^{\times}$ the group of invertible elements of $\RR$.
If $\S\subseteq \RR$ is a semiring with addition and multiplication inherited from $\RR$, then we call $\S$ a subsemiring of $\RR$. If $\S$ is a semifield, we call it a subsemifield. 

\begin{defn}
By a ring with a \textit{positive structure} we mean a pair $(\RR,\RR_{>0})$ where $\RR$ is a (commutative) ring $\RR$ of characteristic $0$ and $\RR_{>0}$ a subsemifield of $\RR$. We call call $\RR_{>0}$ the `positive' subsemifield, or simply the `positive part', of $\RR$. We then also set $\RR_{\ge 0}:=\RR_{>0}\cup\{0_\RR\}$.
\end{defn}

\begin{remark}\label{r:RcontainsQ} Note that if a ring $\mathcal R$ has a positive structure, then it must contain $\Q$. This is because $\mathcal R_{>0}$ must contain $\Z_{>0}$, but as it is a semifield it also contains $\Q_{>0}$ and thus $\mathcal R$ contains $\Q$. 
\end{remark}

\begin{example}\label{ex:C} A standard example for a ring with positive structure is $(\RR,\RR_{>0})=(\C,\R_{>0})$. 
Thus $(\R_{> 0},+,\cdot)$ is the default choice of \textit{positive subsemifield} of $\C$, with $(\R_{\ge 0},+,\cdot)$ the associated \textit{nonnegative subsemiring} of $\C$ (which is of course a subsemifield with $0$). Note that  $\Q_{>0}$ also also forms a subsemifield. The natural numbers $\N=\{1,2,\dotsc\}$ form a subsemiring, as do $\R_{\ge 1}$ and $\Q_{\ge 1}$. Moreover, $\Q_{>0}$ also forms a  subsemifield of the $p$-adic field $\Q_{p}$.   
\end{example}

Two key examples used in a Lie-theoretic context in \cite{Lusztig94} are the following. 
\begin{example}[Laurent series]\label{ex:Laurent} 
Consider the field $\mathcal L=\C((t))$ of Laurent series with its lowest-term valuation $\Val$. 
We set 
\[\mathcal L_{>0}:=\left\{\left .\sum_{k=k_0}^\infty a_k t^{k}\in \mathcal L\ \right | \  k_0\in\Z, \  a_{k_0}\in \R_{>0} \right\}
\]
and consider this the \textit{positive subsemifield} of $\mathcal L$. We also have $\mathcal L_{\ge 0}:=\mathcal L_{>0}\cup\{0\}$.
\end{example}

\begin{example}[tropical semirings]\label{e:tropicalSemifield}  We consider the semifields  $\Z_{\min{}}:=(\Z,\min,+)$ and   $\Rmin:=(\R,\min,+)$ and their associated `semifields with $0$' which are  $\Z_{\min{},\infty}:=(\Z\cup\{\infty\},\min, +)$ and $\R_{\min{},\infty}:=(\R\cup\{\infty\},\min, +)$. An example of a tropical semiring is $\N_{\min}:=(\N,\min{},+)$. These `tropical semirings' cannot arise as subsemirings of any ring $\RR$.  See \cite{Litvinov} for more related background.
\end{example}

\begin{defn}[Valuative positive structure]\label{d:semifieldvaluation} Suppose $\RR$ is a ring as above, and $\RR_{>0}$ a subsemifield. We call $\RR_{>0}$ a \textit{valued subsemifield} of $\RR$ if it is endowed with a semifield homomorphism $\Val:\RR_{>0}\to\Rmin$. Thus by  a ring with a \textit{valuative positive structure} we mean a triple $(\RR,\RR_{>0},\Val)$ where $\RR_{>0}$ is a valued  subsemifield of $\RR$ with valuation $\Val$.  We may additionally ask for $\RR_{\ge 0}$ to be a \textit{topological} semiring (for which the semiring operations are continuous). The topology may be inherited from $\RR$, or it may be defined directly. We call $(\RR,\RR_{>0},\Val)$ a ring with a \textit{topological valuative positive structure}, if $\RR_{\ge 0}$ is a topological semiring and $\Val:\RR_{\ge 0}\to\Rmininf$ is continuous. Here the topology of $\Rmininf$ is the  Euclidean topology on $\R$ with added neighbourhoods $(M,\infty]$ of the additional point, $\infty$.
\end{defn}

A positive structure $\RR_{>0}$ on a ring $\RR$ naturally gives rise to a partial ordering on $\RR$ and therefore, in particular, a partial ordering on $\RR_{>0}$. Namely we consider that $k_1< k_2$ if $k_2-k_1\in\RR_{>0}$. If $\RR_{>0}$ is a valuative positive semifield as above, then we note that $\Val:\RR_{>0}\to\Rmin$ is automatically an order-reversing map.
\begin{lemma}\label{l:Kdelposet} Suppose $k_1,k_2\in\RR_{>0}$ with $k_1\le k_2$, then $ \Val(k_2)\le \Val(k_1)$.
\end{lemma}
\begin{proof}
If $k_1< k_2$, then $k_2=k_1+k$ for a $k\in\RR_{>0}$. But then $\Val(k_2)=\Val(k_1+k)=\min(\Val(k),\Val(k_1))$, which immediately implies $\Val(k_2)\le \Val(k_1)$. 
\end{proof}

\begin{example} The leading-term valuation $\Val$ on $\mathcal L$ from \cref{ex:Laurent} restricts to a semifield homomorphsim  $\mathcal L_{>0}\to\Rmin$ turning $\mathcal L_{>0}$ into a valuative subsemifield  of $\mathcal L$ in the above sense. In contrast, the restriction of the $p$-adic valuation to $\Q_{>0}$ does not give a semifield homomorphism, so that $\Q_{>0}\subset \Q_p$ is not an example. Note that in \cref{d:semifieldvaluation} we do not require $\RR$ to be a valuation ring. See \cref{ex:continuous}. 
\end{example}

We will make use of two very different examples of  rings with valuative positive structure.  
\begin{example}[Generalised Puiseux series~\cite{Markwig:genPuiseux}]\label{ex:genPuiseux}
Let $\KK$ denote the field of {\it generalised 
Puiseux series}. These are series $\sum_{a\in S}c_a t^a$ in a variable $t$ whose exponent sets are described by
\[
\{ S\subset \R \mid  \operatorname{Cardinality}(S\cap \R_{\le N})<\infty \text{ for all $N\in \N$} \}.
\]
Thus an exponent set $S$ may be 
thought of as a \textit{strictly monotone sequence} which is either finite, or countable and tending to infinity. We write $(a_k) \in \SMonSeq_{\infty}$ if $(a_k)_k=(a_0, a_1,a_2,\dotsc )$ is such a  sequence, so that 
\begin{equation}\label{ex:K}
\KK=\left\{\left.\boldalpha=\sum_{(a_k) \in  \SMonSeq_{\infty}} c_{a_k} t^{a_k}\right| c_{a_k}\in\C \right\}.
\end{equation}
The field $\KK$ has the valuation, 
$%
\Val:\KK\to \Rmininf$,
defined by $\ValK(0)=\infty$ and
$\Val\left(\boldalpha\right)=c_{a_0}$
if $\boldalpha=\sum_{(a_k) \in  \SMonSeq_{\infty}} c_{a_k} t^{a_k}$ with $c_{a_0}\ne 0$. This field is complete for the $t$-adic norm $\|\boldalpha\|_t=e^{-\Val(\boldalpha)}$, and it is algebraically closed, see \cite{Markwig:genPuiseux}. 

We then call 
\begin{equation}\label{e:Kpos}
\KK_{>0}=\left\{\left.\sum_{(a_k) \in \SMonSeq_\infty} c_{a_k} t^{a_k} \in\KK\, \right | c_{a_0}\in \R_{>0} \right\}
\end{equation}
the \textit{positive subsemifield} of $\KK$, and  $\KK_{\ge 0}=\KK_{>0}\cup\{0\}$ the \textit{nonnegative subsemiring} of $\KK$. We may consider the $t$-adic topology on $\KK$, so that $(\KK,\KK_{>0},\Val)$ is an example of a ring with topological valuative positive structure.
\end{example}

Note that we have a natural inclusion map 
$(\C,\R_{>0})\to (\KK,\KK_{>0})$, but the restriction of the $t$-adic topology of $\KK$ gives the discrete topology on $\C$. Thus continuity arguments employed in the classical theory of total positivity cannot be straightforwardly extended to $(\KK,\KK_{>0})$.
\begin{example}[Continuous functions]\label{ex:continuous}
Fix $0<\delta<1$ and let $\Kdel:=C^0((0,\delta])$ be the ring of continuous $\R$-valued functions on $(0,\delta]$. Set
\[
\begin{array}{ccl}
\Kdelpos&:=&\left\{k:(0,\delta]\to \R_{> 0}\mid \text{ $k$ is continuous and there exists $K\in\R$ with $\lim_{t\to 0} t^{-K}k(t)\in\R_{>0}$}\right\}.
\end{array}
\]
Note that $\Kdel$ is not a valuation ring, and not an integral domain. But we will show in \cref{s:Kdel}  that $(\Kdel,\Kdelpos,\Val)$ is an example of a ring with a valuative positive structure. Furthermore, we will construct two useful topologies on $\Kdelnn$ in \cref{s:KdelTropToep}, turning $(\Kdel,\Kdelpos,\Val)$ into a ring with a topological valuative positive structure. 
\end{example}

\subsection{Notations for total positivity}\label{s:AlgGroup-notations}  We introduce the notations we will need from algebraic groups~\cite{Borel:LAG, Humphreys:LAG, Springer:book} and total positivity \cite{Lusztig94}.  
We consider a split reductive group $G$ over $\Q$, and will be interested in its $\RR$-valued points for a ring with a positive structure, $(\RR,\RR_{>0})$, as in \cref{s:Semifields}. We identify $G$ and  its subgroups with their $\RR$-valued points as convenient. Primarily $G$ will be either $GL_{n+1}(\RR), PGL_{n+1}(\RR)$ or $SL_{n+1}(\RR)$. The rank of $G$ is always denoted $n$ and we set $I=[1,n]$.

We choose a `pinning' for $G$, see \cite{Lusztig94}, consisting of two opposite Borel subgroups $B_+,B_-$, their unipotent radicals $U_+$ and $U_-$ and simple root homomorphisms $x_i:\RR\to U_+$ and $y_i:\RR\to U_-$. We have maximal torus $T=B_+\cap B_-$ with character lattice $X^*(T)$ and $\mathfrak h^*_\R:=X^*(T)\otimes_\Z\R$. We may add a superscript or subscript $G$ if we need to specify the group, for example writing $\mathfrak h_{G,\R}^*$ or $B_+^G$ instead of $\mathfrak h_{\R}^*$ and $B_+$. Let $R_+$ denote the set of positive roots and write $\{\omega_i\mid i\in I\}$ for the set of fundamental weights. Let $\Pi=\{\alpha_i\mid i\in I\}$ be the set of simple roots. We denote the coroot associated to a root $\alpha$ by $\alpha^\vee$ and write $\Pi^\vee=\{\alpha_i^\vee\mid i\in I\}$  for the set of simple coroots. Let $W$ denote the Weyl group of $G$ and $\{s_i\mid i\in I\}$ be the set of simple reflection generators of $W$. For every simple reflection $s_i$ we fix a representative in $G$ by setting
\[
\bar s_i:=y_i(-1)x_i(1)y_i(-1).
\] 
We also obtain a representative $\bar w$ for a general element $w\in W$ by replacing $s_i$ by $\bar s_i$ in a/any reduced expression $w=s_{i_1}\dotsc s_{i_k}$. We recall that if $s_is_js_i=s_js_is_j$ then we have the relation
\begin{equation}\label{e:braidbirationalA2}
x_i(z_1)x_j(z_2)x_i(z_3)=x_j\left(\frac{z_2z_3}{z_1+z_3}\right)x_i(z_1+z_3)x_j\left(\frac{z_1z_2}{z_1+z_3}\right).
\end{equation}

If $G$ is of type $A_n$ then we choose the standard pinning with $T$ consisting of diagonal matrices. and the simple root homomorphisms given explicitly by $x_i(a)=I_{n+1} + a (\delta_{j,i+1})_{i,j}$ and $y_i(a)=I_{n+1} + a (\delta_{i,j+1})_{i,j}$, where $I_{n+1}$ is  the $(n+1)\times (n+1)$ identity matrix. For $i,j\in I$ with $i\le j$ we set $\alpha_{i,j+1}=\alpha_i+\dotsc + \alpha_j$ which gives all the positive roots. For $G=GL_{n+1}$ or $SL_{n+1}$ we also write $\ep_j$ for the character of $T$ corresponding to the  $j$\textsuperscript{th} diagonal entry, so that $\alpha_i=\ep_i-\ep_{i+1}$ and $\alpha_{i,j+1}=\ep_i-\ep_{j+1}$.

In \cite{Lusztig94,Lusztig:TP2}, semifield analogues of $G,B_+,U_+,T$ for reductive algebraic groups are constructed.  
We recall the relevant definitions below.

\begin{defn}\label{d:nonneg} 
We let $U_+(\RR_{\ge 0})$ be the submonoid of $U_+(\RR)$ generated by $\{x_i(a)\mid a\in\RR_{\ge 0},i\in I\}$. For any split torus $T$ over $\RR$ the positive part $T(\RR_{>0})$ consists of the points $d\in T(\RR)$ satisfying $\chi(d)\in\RR_{>0}$ for all $\chi$ in the character group $X^*(T)$. We set $T(\RR_{\ge 0})=T(\RR_{>0})$, since characters cannot take the value $0$, and we note that $T(\RR_{>0})$ is a subgroup of $T(\RR)$. 
 For our choice of reductive algebraic group $G$,
\[
\begin{array}{ccl}
B^G_+(\RR_{\ge 0})&:=&T^G(\RR_{>0})U_+(\RR_{\ge 0})=U_+(\RR_{\ge 0})T^G(\RR_{>0}),\\
G(\RR_{\ge 0})&:=&U_-(\RR_{\ge 0})T^G(\RR_{>0})U_+(\RR_{\ge 0})=U_+(\RR_{\ge 0})T^G(\RR_{>0})U_-(\RR_{\ge 0}).
\end{array}
\]
 \end{defn}  
  
These are the totally nonnegative parts for $G$ and its key subgroups. We now recall Lusztig's construction of the totally positive parts. 
 Consider the longest element $w_0$ of the Weyl group $W$.  
 Let $\mathbf i=(i_1,\dotsc, i_N)$ encode a reduced expression $s_{i_1}\dotsc s_{i_N}$ of $w_0$, and consider the map $\phi_{\mathbf i}: \RR^{N}\to U_+(\RR)$ defined by
\begin{equation}\label{e:phii} 
\phi_{\mathbf i}((z_i)_{i=1}^N)=x_{i_1}(z_1)x_{i_2}(z_2)\dotsc x_{i_N}(z_N).
\end{equation}
The set $\phi_{\mathbf i}\left(\RR_{\ge 0}^{N}\right)$ 
is independent of the reduced expression $\mathbf i$ and equals to $U_+(\RR_{\ge 0})$, see  \cite[Lemma 2.10]{Lusztig94}. 
The same is true of the further restriction of $\phi_{\mathbf i}$ to $\RR_{>0}^{N}$, which additionally is injective, see \cite[Proposition~2.7]{Lusztig94}.

\begin{defn}\label{d:pos}
Let $U_+(\RR_{>0})$ be the image under $\phi_{\mathbf i}$ of $\RR_{>0}^{N}$ for a/any reduced expression $\mathbf i$ of $w_0$. That is, $U_+(\RR_{>0})=\phi_{\mathbf i}(\RR_{>0})$. We also define $U_-(\RR_{>0})$  completely analogously, with  $x_i$ factorisations replaced by $y_i$ factorisations.
Then we set  
\[
\begin{array}{ccl}
B^G_+(\RR_{>0})&:=&T^G(\RR_{>0})U_+(\RR_{>0})
=U_+(\RR_{>0})T^G(\RR_{>0}),\\
G(\RR_{> 0})&:=&U_-(\RR_{> 0})T^G(\RR_{>0})U_+(\RR_{> 0})=U_+(\RR_{> 0})T^G(\RR_{>0})U_-(\RR_{> 0}).
\end{array}
\]
We may also consider the  images  of $(\RR^\times)^{N}$ under the $\phi_{\mathbf i}$, which do depend on $\mathbf i$ and will be denoted $U^{\mathbf i}_+(\RR^\times)$. Note that $U^{\mathbf i}_+(\RR^\times)$, and therefore $U_+(\RR_{>0})$, lies in the intersection $U_+\cap B_-\dot w_0 B_-$ by an application of the Bruhat lemma. 
\end{defn}

We recall that these definitions are compatible with the classical theory of total positivity (for matrices over $\R$) in terms of positivity of minors that goes  back to Fekete, Polya, Schoenberg, Gantmacher and Krein. Namely, $G=GL_{n+1}$ and let $J,K$ of $[1,n+1]$ be subsets with equal cardinality $k$. Then for any $g\in GL_{n+1}(\RR)$ we let $g^J_K$ denote the submatrix of $g$ with rows $J$ and columns $K$, and set
\[
\Minor^{J}_{K}(g)=\det(g^{J}_K),
\]
with the convention that $\Minor^{\emptyset}_{\emptyset}=1$. Then $G(\RR_{>0})$ consists of the $g\in G(\RR)$ for which all $\Minor^{J}_{K}(g)$ lie in $\RR_{>0}$, and similarly for $U_+$, see \cite{Whitney,Lusztig94,BFZ}.

\subsection{Twist map} We now recall a subtle bijection naturally relating $U_+(\RR_{>0})$ and $U_-(\RR_{>0})$ introduced in \cite[Theorem~8.7]{Lusztig94} over $\R$, and shown to be given by subtraction-free rational functions over $\Q$ in \cite[3.4]{Lusztig97} and \cite[Proposition~5.1]{BZ}. Passing from $\Q$ to $\RR$ we may rewrite this results as follows in our setting.  
\begin{theorem}\label{t:involution}
Let $(\RR,\RR_{>0})$ be a ring with positive structure. Any $u\in U_+(\RR_{>0})$ has a unique factorisation $u=b_-\dot w_0 u'$
where $b_-\in B_-$ and $u'\in U_-(\RR_{>0})$, and this defines a bijection
\[
\begin{array}{lcr}
U_+(\RR_{>0})\overset{\Psi}\longrightarrow U_-(\RR_{>0}),&  &u\mapsto u'.
\end{array}
\]
Moreover, for any reduced expression  $\mathbf i$ of $w_0$ there is a  subtraction-free birational involution with integer coefficients,
$\Psi_{\mathbf i}:\RR^N\longrightarrow \RR^N$,
such that if $z=(z_i)_{i=1}^N\in \RR_{>0}^N$ and $\Psi_{\mathbf i}(z)=(z'_i)_{i=1}^N$, we have
\[
\begin{array}{lcr}
\Psi: u= x_{i_1}(z_1)x_{i_2}(z_2)\dotsc x_{i_N}(z_N)
&\mapsto&
 u'=y_{i_1}(z'_1)y_{i_2}(z'_2)\dotsc y_{i_N}(z'_N).
 \end{array}
\]
\end{theorem}
The map $\Psi$ (composed with transposition) is studied and generalised in \cite{BZ,BZ:Tensor} where it is called the \textit{twist map}, and an explicit `chamber ansatz' formula expressing the coordinates $z_i$ of $u=\phi_{\mathbf i}((z_i)_i)$ in terms of minors of $\Psi(u)$ is given. See also  \cite[Section 3.1]{BFZ} for the type $A$ case. . 

In our type $A$ setting we will be making use of the following lemma which inverts $\phi_{\mathbf i}$ directly for a particular reduced expression~$\mathbf i_0$. It was written down in slightly different conventions in \cite[Appendix~A]{Ludenbach:thesis}.

\begin{lemma}\label{l:mijFormula}
Suppose $u\in U^{\mathbf i_0}_+(\RR^\times)$ for the reduced expression  $\mathbf i_0=(n,n-1,\dotsc, 1;n,\dotsc, 2;  \dotsc; n,n-1;n)$ of  $w_0$, 
\begin{eqnarray}\label{e:idealfact}
u&=&x_{n}\left({m_{n,1}}\right)\cdots\cdots\cdots x_2\left({m_{2,1}}\right) x_1\left({m_{1,1}}\right)\\
&& x_{n}\left({m_{n-1,2}}\right)\cdots\cdots\cdots x_2\left({m_{1,2}}\right)\nonumber\\
&&\nonumber\\
&&\dots\nonumber\\
&&\nonumber\\
&& x_{n}\left({m_{2,n-1}}\right)x_{n-1}\left({m_{1,n-1}}\right)\nonumber\\
&& x_n\left({m_{1,n}}\right).\nonumber
\end{eqnarray} 
Then the $m_{ij}$ are given by the formula
\begin{equation}\label{e:mviaminors}
m_{ij}=\frac{\Minor_{[i]+j}^{[i]}(u)\Minor_{[i-1]+(j-1)}^{[i-1]}(u)}{\Minor^{[i]}_{[i]+(j-1)}(u)\Minor_{[i-1]+j}^{[i-1]}(u)},
\end{equation}
where  we use the notation $[k]+r=\{1+r,2+r,\dotsc, k+r\}$ for a shifted index set.
\end{lemma}

\begin{proof}
It is straightforward to check that $\Minor^{[i]}_{[i]+j}(u)=\prod_{k=1}^i \prod_{r=1}^j m_{k,r}$. The formula follows immediately.
\end{proof}

\subsection{}\label{s:posstructure}
In this section we explain the notion of positive structure on a variety that we will use, 
that can be applied to all of the examples considered so far and relates closely to \cite{BK:GeometricCrystalsII,FG,Lusztig:TP2}.

\begin{defn}
\label{d:admissible} Suppose $\S$ is any semifield. We call a map $\psi:\S^m_{>0}\to \S^{m'}_{>0}$  \textit{admissible} 
if it is a subtraction-free rational map with integer coefficients. Suppose $m=m'$ and $\psi$ has an admissible inverse. Then we call $\psi$ a  \textit{bi-admissible} transformation.
\end{defn}

\begin{defn}\label{d:posstructure} 
Consider a variety $X$ over $\RR$, where $\RR$ has a positive subsemifield $\RR_{>0}$.  
We let a \textit{positive structure} on  $X$ be a subset $X(\RR_{>0})$ of $X(\RR)$ together with a collection of  bijections $\RR_{>0}^N\to X(\RR_{>0})$ called `positive charts' that are related to one another by bi-admissible transformations of $\RR_{>0}^N$. We may call it regular if the charts arise as restrictions of algebraic maps from a torus $\mathbb G_m^N$.

By a \textit{compatible} map we mean a morphism of $\RR$-varieties such that the associated map on $\RR$-valued points restricts to a map $X(\RR_{>0})\to X'(\RR_{>0})$. We call a compatible map \textit{admissible} if in terms of any choice of positive charts the corresponding map $\RR_{>0}^m\to \RR_{>0}^{m'}$ is admissible.   

Two varieties $X$ and $X'$ with positive structures are called isomorphic if there is an admissible map $X\to X'$ that has an admissible inverse.
\end{defn}

This notion of positive structure 
arises naturally from a `regular positive atlas' in the sense of ~\cite{FG} or a positive variety in the sense of \cite{BK:GeometricCrystalsII} by considering the $\RR$-valued points. It is very close to the notion of positive structure from \cite[1.3]{Lusztig:TP2}, a  difference being that we replace the ambient field $K$ by an ambient ring~$\RR$. However, since $\RR$ contains $\Q$, see \cref{r:RcontainsQ}, this generalisation is a straightforward one to make.

\begin{example} For a torus $T$ we consider a positive chart  to be any bijection $\chi_{>0}:\RR_{>0}^{\dim(T)}\to T(\RR_{>0})$ given by a basis of the co-character lattice of~$T$. We consider $T(\RR_{>0})$ to define a positive structure for $T$ with this collection of positive charts. 
\end{example}
\begin{example}
The totally positive part $U_+(\RR_{>0})$ of $U_+(\RR)$ by its definition comes with the restrictions $\phi_{\mathbf i,>0}: \RR_{>0}^{R_+}\to U_+(\RR_{>0})$ of the maps $\phi_{\mathbf i}$ from \eqref{e:phii}. 
We define a \textit{positive chart for $U_+$} to be any bijection from $\RR_{>0}^N$ to $U_+(\RR_{>0})$ that is related to one of the $\phi_{\mathbf i,>0}$ by a bi-admissible transformation. Similarly for $U_-$ and its defining charts. This determines positive structures on $U_+$ and $U_-$.
\end{example}

\begin{remark}\label{r:posstructureThm}
\cref{t:involution} can be interpreted as saying that $\Psi\circ\phi_{\mathbf i,>0}$ is one of the positive charts for the positive structure of $U_-$. 
\end{remark}

\begin{example}
For $B_+$ and $B_-$ we have the products of the positive charts for $T$ and $U_+$, respectively $T$ and $U_-$ (using $B_{\pm}\cong T\times U_{\pm}$). For example, we have $\chi_{>0}\times\phi_{\mathbf i,>0}:(l,z)\mapsto \chi(l)\phi_{\mathbf i}(z)$ for $B_+$. We consider any chart $\phi_{>0}:\RR_{>0}^{\dim T+N}\to B_{\pm}(\RR_{>0})$ related to one of these product charts by an admissible transformation to be a  \textit{positive chart} for $B_{\pm}$, provided it obeys the condition that the composition,
\[
\RR_{>0}^{N+\dim T}\overset\sim\longrightarrow B_{\pm}(\RR_{>0}) \overset \ep\longrightarrow \RR_{>0},
\]
is still a Laurent monomial map for any character $\ep$ of $T$ (extended trivially along $U_{\pm}$ to $B_{\pm}$). This constraint is equivalent to imposing that the projection onto the diagonal $B_+\to T$ be admissible, compare also \cite[Section~3.4]{BK:GeometricCrystals}.
\end{example}
\subsection{Upper-triangular Toeplitz matrices}
We introduce the following notations concerning upper-triangular Toeplitz matrices. We will initially be mostly be concerned with unipotent upper-triangular Toeplitz matrices. 
\begin{defn}
Let $(\RR,\RR_{>0})$ be a ring with positive structure. We write $\ToeplitzU_{n+1}(\RR)$ for the  unipotent upper-triangular Toeplitz matrices in $GL_{n+1}$. Then the `totally nonnegative' part is defined as 
\[
\ToeplitzU_{n+1}(\RR_{\ge 0}):=\ToeplitzU_{n+1}(\RR)\cap U_+(\RR_{\ge 0}).
\]
Inside $\ToeplitzU_{n+1}(\RR_{\ge 0})$ we have the `totally positive' part of $\ToeplitzU_{n+1}(\RR)$ which is defined to be 
\[
\ToeplitzU_{n+1}(\RR_{>0}):=\ToeplitzU_{n+1}(\RR)\cap U_+(\RR_{>0}).
\]  
We may also consider all the upper-triangular Toeplitz matrices in $GL_{n+1}$ (dropping the unipotent assumption), which we denote by $\Toeplitz_{n+1}$.  We set
\[
\begin{array}{lcc}
\Toeplitz_{n+1}(\RR_{\ge 0})&:=&\Toeplitz_{n+1}(\RR)\cap B^{GL_{n+1}}_+(\RR_{\ge 0}),\\
\Toeplitz_{n+1}(\RR_{>0})&:=&\Toeplitz_{n+1}(\RR)\cap B^{GL_{n+1}}_+(\RR_{>0}).
\end{array}
\] 
$\ToeplitzU_{n+1}(\RR_{\ge 0})$ and $\Toeplitz_{n+1}(\RR_{\ge 0})$ are monoids and $\ToeplitzU_{n+1}(\RR_{>0})$ and $\Toeplitz_{n+1}(\RR_{>0})$ are semigroups (without unit).   
\end{defn}
The analogue $\Toeplitz_G(\RR)$ of the upper-triangular Toeplitz matrices for a general reductive algebraic group $G$ can be defined as the stabiliser of a standard principal nilpotent, and has important relations to the Toda lattice and quantum cohomology of (Langlands dual) flag variety, \cite{Kostant:Toda, peterson, Kos:QCoh}.  

\section{Parametrization theorem for $\Toeplitz_{n+1}(\KK_{>0})$}\label{s:ToeplitzFinite}
We will now work over the field $\KK$ of generalised Puiseaux series together with its $t$-adic topology and its positive subsemifield $\KK_{>0}$, see \cref{ex:genPuiseux}. Suppose $G$ is the group $GL_{n+1}(\KK)$ with its maximal torus $T=T^{GL}$, and the `anti-principal' minors $\Delta_i:=\Minor^{[i]}_{[n-i+2,n+1]}$ where $i=1,\dotsc, n+1$. The main goal of this section is to prove the following theorem. 

\begin{theorem} \label{t:PuiseauxFinite}
The map $
\Delta_{>0}^{GL}:\Toeplitz_{n+1}(\KK_{>0})\longrightarrow T^{GL}(\KK_{>0})$
defined by sending $b\in\Toeplitz_{n+1}(\KK_{>0})$ to the diagonal matrix $d$ with diagonal entries
\[
d_1= \Delta_{1}(b),\   
d_2=\frac{\Delta_{2}(b)}{\Delta_{1}(b)},\ 
\dotsc,
\ d_{n+1}=\frac {\Delta_{n+1}(b)}{\Delta_{n}(b)}
\]
is a homeomorphism.
\end{theorem}

This theorem can be viewed as a strengthening of \cite[Theorem~5.1]{Judd:Flag}. It is an analogue over $\KK$ of \cref{t:RealFinite}.

\begin{remark}\label{r:SL-PGL-FiniteParam} As a corollary of \cref{t:PuiseauxFinite} we obtain a homeomorphism, 
\begin{equation}\label{e:SLFiniteParam}
\Delta_{>0}^{SL}:\ToeplitzU_{n+1}(\KK_{>0})\overset\sim\longrightarrow T^{SL}(\KK_{>0}),
\end{equation}
by restriction, and another homeomorphism, 
\begin{equation}\label{e:PGLFiniteParam}
\Delta_{>0}^{PGL}:\ToeplitzU_{n+1}(\KK_{>0})\overset\sim\longrightarrow T^{PGL}(\KK_{>0}),
\end{equation}
which arises from taking the quotient by $\KK_{>0}$ on both sides of $\Delta_{>0}^{GL}$.  Note that $\Delta_i$ is homogeneous of degree $n+1-i$ for the rescaling action by $\KK_{>0}$, so that $\Delta^{GL}_{>0}$ overall has degree~$1$. Note also that we can uniquely and continuously normalise the determinants of elements of $T^{GL}(\KK_{>0})$ to $1$ using the well-defined $(n+1)^{st}$-root map on $\KK_{>0}$ to obtain a surjective homomorphism $T^{GL}(\KK_{>0})\to T^{SL}(\KK_{>0})$ that identifies $T^{SL}(\KK_{>0})$ with $T^{PGL}(\KK_{>0})$ as sets. 

Finally, we may compose $\Delta^{PGL}$ further with the natural isomorphism $
T^{PGL}(\KK_{>0})\to (\KK_{>0})^n$ given by the  negative simple roots. The composition gives a homeomorphism 
\begin{equation}\label{e:q-paramToeplitz}
(\mathbf q_1,\dotsc,\mathbf q_n):\ToeplitzU_{n+1}(\KK_{>0})\overset\sim\longrightarrow \KK_{>0}^{n}
\end{equation}
whose components $\mathbf q_i$ have an interpretation in terms of the `quantum parameters' for the quantum cohomology of the flag variety $SL_{n+1}/B_+$. 
Namely, this homeomorphism \eqref{e:q-paramToeplitz} represents the direct generalisation of the original parametrisation of $\ToeplitzU_{n+1}(\R_{>0})$ constructed in \cite{rietschJAMS} using properties of the quantum cohomology ring $qH^*(SL_{n+1}/B_+)$. The formula for the quantum parameters in terms of Toeplitz minors, $\mathbf q_i=\frac{\Delta_{i-1}\Delta_{i+1}}{\Delta_i^2}$, is due to Kostant, see~\cite{Kos:QCoh}.
\end{remark}

\begin{remark}\label{r:posstructureToeplitz}
We can in principle use this theorem to endow $\Toeplitz_{n+1}$  with a \textit{positive structure} in the sense of \cref{d:posstructure} (albeit not a regular one). Namely, we identify $T^{GL}(\KK_{>0})$ with $\KK_{>0}^{n+1}$ so that the inverse of $\Delta^{GL}_{>0}$ gives a bijection $\KK_{>0}^{n+1}\to\Toeplitz(\KK_{>0})$. We would then define a  positive chart for $\Toeplitz_{n+1}$ to be any  bijection $\KK_{>0}^{n+1}\to\Toeplitz(\KK_{>0})$ obtained by Laurent monomial transformations from this one. Applying the analogous procedure to the bijections~\eqref{e:SLFiniteParam} and \eqref{e:PGLFiniteParam} leads to two distinct positive structures on $\ToeplitzU_{n+1}$ reflecting the fact that the construction of the homeomorphism from $T^{PGL}(\KK_{>0})$ to $T^{SL}(\KK_{>0})$ used $(n+1)^{st}$ roots.   
\end{remark}

\subsection{}\label{s:LieMirror}
The proof of \cref{t:PuiseauxFinite} involves constructing $b=(\Delta_{>0}^{GL})\inv(d)$ as critical point of a function $W_d$. This is along the lines of the proof in~\cite{rietschNagoya} in the setting where $(\RR,\RR_{>0})=(\C,\R_{>0})$, and is also the approach taken in \cite{Judd:Flag} for particular parameter values.

\begin{defn}\label{d:GLmirror} Let $G=GL_{n+1}(\RR)$. Define $Z:=B_+\cap B_-\bar w_0 B_-$. For any element $b\in Z$ there is a unique factorisation $b=u_Ld\bar w_0  u_R$, where $u_L,u_R\in U_-$ and $d\in T$. Let $f_k^*(u):=\Minor^{[k-1]\cup \{k+1\}}_{[k]}(u)=u_{k+1,k}$. We define
\begin{equation*}
\begin{array}{cccl}
W: & Z(\RR)&\longrightarrow & \RR\\
& b
&\mapsto& \sum_{k\in I} f_k^*(u_L)+ \sum_{k\in I} f_k^*(u_R),
\end{array}
\end{equation*}
where $u_L$ and $u_R$ are given by the factorisation $b=u_Ld\bar w_0 u_R$. 

Let us also define $\mathbf d:  Z\to  T$ to be the map sending $b$ to the diagonal factor $d$
  in $b=u_Ld\bar w_0  u_R$. Explicitly, we have that $\mathbf d(b)$ is the diagonal matrix  with entries
\begin{equation}\label{e:di}
d_1=\Delta_{1}(b),\   
d_2=\frac{\Delta_{2}(b)}{\Delta_{1}(b)},\ 
\dotsc,\ d_n=\frac{\Delta_{n}(b)}{\Delta_{{n-1}}(b)},\ d_{n+1}= \frac{{\Delta_{n+1}(b)}}{\Delta_{n}(b)}.
\end{equation}
 We denote the fiber $\mathbf d\inv(d)$ by $Z_d$, and the restriction of $W$ to the fiber $Z_d$ will be denoted $W_d$. We will call $W$ or $W_d$ the \textit{superpotential} and $\mathbf d$ the \textit{canonical weight} map. 
\end{defn}
\begin{remark}
These maps first appeared in the theory of geometric crystals by Berenstein and Kazhdan \cite{BeKa,BK:GeometricCrystalsII}, where $\mathbf d$ is called the highest weight map, and $W$ is called the `decoration' of the geometric crystal. They were constructed again independently in \cite{rietsch}, where $W$ was proposed to be the superpotential for the flag  variety, and shown to recover in a particular chart the Laurent polynomial superpotential proposed earlier by Givental~\cite{Givental:QToda}.  
\end{remark}

\begin{defn}\label{d:duni}
We note that the fiber $Z_d$ of $\mathbf d$ lies in $SL_{n+1}$ whenever $d\in T^{SL_{n+1}}$, in fact, $Z^{SL_{n+1}}=\mathbf d\inv(T^{SL_{n+1}})$. We restrict $\mathbf d$ to the unipotent part of $Z$, defining  
\[
\mathbf d_{uni}:U_+\cap B_-\bar w_0 B_-\to T^{SL_{n+1}}.
\]
This map is again given by \eqref{e:di}, now with $\Delta_{n+1}(b)=1$.
\end{defn}
We have the following relationship between the superpotential $W$ and $\Toeplitz_{n+1}$, which follows from {\cite[Theorem~4.1]{rietsch}}. 
\begin{theorem}[\cite{rietsch}]   
\label{t:Zcrit}
 Consider the union of critical loci of the $W_d$, 
\[Z^{crit}(\C):=\{b\in Z(\C)\mid \mbox{$b$ is a critical point for $W_d$ where $d=\mathbf d(b)$}\}.\]
Then $Z^{crit}(\C)=\Toeplitz_{n+1}(\C)\cap Z(\C)$.
\end{theorem}
Let $Z^{crit}(\R_{>0}):= Z^{crit}(\C)\cap B_+(\R_{>0})$. Then the above theorem implies that $Z^{crit}(\R_{>0})=\Toeplitz_{n+1}(\R_{>0})$.
 We now have the following parametrization theorem over $\R_{>0}$. 
\begin{theorem}[\cite{rietschJAMS,rietschNagoya}]   
\label{t:ZcritposR}
The map $\mathbf d$ restricts to a homeomorphism
\[
\Toeplitz_{n+1}(\R_{>0})\longrightarrow T(\R_{>0}).
\]
\end{theorem}
This theorem says in particular that if $d\in T(\R_{>0})$ then $W_d$ has a unique critical point in $B_+(\R_{>0})$. Namely, $Z_{d}(\C)\cap\Toeplitz(\R_{>0})$ contains a single point and this is the totally positive critical point of $W_d$. 
\subsection{}\label{s:Wpositive} We return to our ring with positive structure $(\RR,\RR_{>0})$. Since $Z=B_+\cap B_-w_0 B_-$, we let $Z(\RR_{>0}):=B_+(\RR_{>0})$ as a set. However, the difference between $Z(\RR_{>0})$ and $B_+(\RR_{>0})$ is encapsulated in the existence of the canonical weight map $\mathbf d$. We consider its restriction $\mathbf d_{>0}: Z(\RR_{>0}) \to  T(\RR_{>0})$  to be a part of the structure of $Z(\RR_{>0})$, and therefore give a more restrictive definition of `positive chart' for $Z$. (Note that $\mathbf d_{>0}$ is well-defined by the explicit formula in \eqref{e:di}).

\begin{defn}\label{d:poschartZ}
A \textit{positive chart} for $Z$ is a positive chart $\phi_{>0}$ for $B_+$ for which additionally the composition $\mathbf d\circ\phi_{>0}:\RR_{>0}^{\dim Z}\to T(\RR_{>0})$ is given by a Laurent monomial formula (in terms of a basis of characters of~$T$). All the positive charts for $B_+$ of the form $\chi\times \phi_{\mathbf i,>0}$ have this property -- where $\chi$ is any positive chart for~$T$. This follows from the fact that $TU_+^{\mathbf i}$ is contained in $Z$, see  \cref{d:pos}.

By a \textit{fibered positive chart} for $Z$ we mean a positive chart for $Z$ of the form $\phi_{>0}:\RR_{>0}^{\dim T}\times\RR_{>0}^{N}\to Z(\RR_{>0})$, together with a positive chart for $T$ identifying $\RR_{>0}^{\dim T}$ with $T(\RR_{>0})$, such that for every  $d\in T(\RR_{>0})$ the restriction gives a well-defined bijection with the associated fiber,
 \[
 \phi_{d,>0}: \{d\}\times \RR_{>0}^{N}\longrightarrow Z_d(\RR_{>0}).
 \]
We may also write $\phi_{d,>0}$  for the map $\RR_{>0}^{N}\longrightarrow Z_d(\RR_{>0})$ (ignoring the factor $\{d\}$), and call this the positive chart for $Z_d$ associated to $\phi_{>0}$. 
\end{defn}

 We consider $Z$ to have a positive structure defined by the above notion of positive chart. The fibered positive charts will be a useful subset of positive charts.

\begin{lemma}\label{l:Winposchart}
The superpotential $W$ restricts to a map
\[
W_{>0}:Z(\RR_{>0})\to \RR_{>0}
\]
which in any positive chart of $Z(\RR_{>0})$ is given by a Laurent polynomial with positive integer coefficients. 
\end{lemma}
\begin{proof} Fix a positive chart for $Z$. For $b=u_L d\bar w_0  u_R$ we have the identity 
\begin{equation}\label{e:minorofb}\Minor^{[n-k]\cup\{k+1\}}_{[n-k+2,n+1]}(b)=d_{1}\dotsc d_{k}
f_k^*(u_L).
\end{equation}
 We also have that the $d_i$ are Laurent monomials and the nontrivial minors of $b$ are Laurent polynomials with positive integer coefficients in terms of the positive chart. Therefore \eqref{e:minorofb} implies that $f_k^*(u_L)$ is a Laurent polynomial with positive integer coefficients. Let us write $b=\ell u$ for $\ell\in T(\RR_{>0})$ and $u\in U_+(\RR_{>0})$, so that $u=\ell\inv u_L\bar w_0 d u_R$. Then $u_R=\Psi(u)$ for the map $\Psi$ from \cref{t:involution}. This implies that $u_R\in U_-(\RR_{>0})$ and that $f_k^*(u_R)$ is a Laurent polynomial with positive integer coefficients in terms of the original positive chart.
\end{proof}

\subsection{}\label{s:JR} We now explain some results about positive Laurent polynomials and their critical points 
following \cite{JuddRietsch:24}. Let us write $(\K,\K_{>0})$ for a field with positive structure, as we will be applying these results in the finite-dimensional case, where we work over $\R,\C$ or the field of  generalised Puiseaux series $\KK$.  

\begin{defn}
Let $\L=\sum\gamma_i x^{\v_i}$ be a Laurent polynomial in $r$ variables and with coefficients in $\K$. Here each $\v_i$ is a lattice point $\v_i=(\v_{i,1},\dotsc,\v_{i,r})\in\Z^r$ in $\R^r$, and $x^{\v_i}$ is the corresponding Laurent monomial $x^{\v_i}=x_1^{\v_{i,1}}\dotsc x_r^{\v_{i,r}}$.  The Newton polytope of $\L$ is defined to be the convex hull in $\R^r$ of the set of lattice points $\{\v_1,\dots, \v_m\}$.  Supposing $\K$ has a positive subsemifield $\K_{>0}$, we call $\L$ a \textit{positive Laurent polynomial} over $\K$ if the coefficients $\gamma_i$ all lie in $\K_{>0}$. 
\end{defn}

We call $p\in(\K\setminus\{0\})^r$ a critical point of $\L$ if the vector-valued function $G(x)=\sum\gamma_i x^{\v_i}\v_i$ vanishes at~$p$. This definition is equivalent to the vanishing of all the partial derivatives of the Laurent polynomial $L$, as the components of $G(x)$ are just $x_j\frac{\partial}{\partial x_j}f$. 
 We now have the following existence and uniqueness theorem in the $\K=\KK$ setting. A related version of the existence part of this result (requiring some further assumptions) was proved earlier in \cite[Proposition 4.7]{FOOO:I}. 

\begin{theorem}[{\cite[Theorem 1.1, Corollary 8.5]{JuddRietsch:24}}]\label{t:JuddR} Let $\K=\KK$, the field of generalised Puiseaux series. Let $\L=\sum \gamma_i x^{\v_i}$ be a positive Laurent polynomial. Then $\L$ has a unique critical point in $(\KK_{>0})^r$ if and only if the Newton polytope of $\L$  is full-dimensional with zero in the interior. We call this point \emph {the positive critical point} of $\L$. 
\end{theorem}

This theorem was inspired by the analogous result over $\K=\R$ from \cite{Galkin}. We observe that the above Theorem has the following corollary. 

\begin{cor}\label{c:unique} Fix $m$ exponent vectors $\mathrm v_i$, where $i=1,\dotsc m$. Let $\L_{\gamma}=\sum \gamma_i x^{\v_i}$ be a positive Laurent polynomial in $r$ variables, where $\gamma$ denotes the  vector $\gamma=(\gamma_i)_i$ of  coefficients, $\gamma_i\in\KK_{>0}$. The following statements are equivalent. 
\begin{enumerate}
\item The Laurent polynomial $\L_{\gamma}$ has a unique critical point in $(\KK_{>0})^r$ for some choice of $\gamma\in\KK_{>0}^m$.
\item The Laurent polynomial $\L_\gamma$ has a unique critical point in $(\KK_{>0})^r$ for any choice of $\gamma\in\KK_{>0}^m$.
\item The Laurent polynomial $\L_{\gamma}$ has a unique critical point in $(\R_{>0})^r$ for some choice of $\gamma\in\R^m_{>0}$.
\end{enumerate}
Namely, these statements are all equivalent to $L_\gamma$ having a Newton polytope that is full-dimensional with $0$ in the interior.
\end{cor}

\begin{proof}
Assume (1) holds. If $\L_\gamma$ has a unique positive critical point then by \cref{t:JuddR} its Newton polytope is full-dimensional with $0$ in the interior. On the other hand, this Newton polytope does not depend on the $\gamma_i$. Therefore $\L_\gamma$  has full-dimensional Newton polytope with zero in the interior for any choice of $\gamma$. Now Theorem~\ref{t:JuddR} implies (2) and since (2) clearly implies (1) these two statements are equivalent. (1) and (2) imply that the Newton polytope of $\L_{\gamma}$ is full-dimensional with $0$ in its interior for any $\gamma$, including $\gamma\in(\R_{>0})^m$. With the result of \cite{Galkin} this implies (3). Finally, the proofs of Lemmas~8.3 and 8.4 in \cite{JuddRietsch:24} work verbatim over $\R$ and show that if a Laurent polynomial with coefficients in $\R_{>0}$ has a unique positive critical point, then its Newton polytope must be full-dimensional with $0$ in the interior. This shows that (3) implies (1) and (2). \end{proof}

We can now ask about the dependence of the positive critical point of $\L_{\gamma}$ on the choice of $\gamma$.

\begin{prop} \label{p:continuity}
Consider the field $\KK$ of generalised Puiseaux series with its $t$-adic topology. Let $\L_{\gamma}=\sum_{i=0}^m \gamma_i x^{\v_i}$ be a positive Laurent polynomial in $r$ variables depending on $\gamma=(\gamma_i)_i\in(\KK_{>0})^m$ such that $\L_\gamma$ has a unique positive critical point for a/any $\gamma$. Then the map 
\[ 
p_{crit}:\KK_{>0}^m\to\KK_{>0}^r
\]
sending $\gamma=(\gamma_i)_i$ to the unique critical point of $\L_\gamma$ is continuous for the $t$-adic topology. 
\end{prop}

The proof of this proposition is highly dependent on the notations and constructions from \cite{JuddRietsch:24} and will be proved in a separate note, \cite{R:CritCont}. We record the following corollary.

\begin{cor}\label{c:continuity}
Let $\L_d=\sum_{i=0}^m \gamma_i(d) x^{\v_i}$ be a family of Laurent polynomials over $\KK$ in $r$ variables, where each $\gamma_i$ is continuous map from $\KK_{>0}^s\to\KK_{>0}$. If $\L_d$ has a unique positive critical point for a/any $d\in\KK_{>0}^s$, then the map
\[ 
p_{crit}:\KK_{>0}^s\to\KK_{>0}^r
\]
sending $d=(d_i)_i$ to the unique critical point of $\L_d$ is continuous for the $t$-adic topology. 
\end{cor}

\begin{proof}
The coefficients $\gamma_i$ taken together define a continuous map $(\gamma_i)_{i=1}^{m}:\KK_{>0}^{s}\to\KK_{>0}^{m}$. The map $p_{crit}:\KK_{>0}^s\to\KK_{>0}^r
$ is the composition of this continuous map $\gamma$ with the map shown to be continuous in  \cref{p:continuity}.
\end{proof}
\subsection{}
We now apply the results from \cref{s:JR} in our setting and finish the proof of \cref{t:PuiseauxFinite}. Recall the notations from 
\cref{s:LieMirror}. 
\begin{prop}\label{p:WdNP}
Consider a fibered positive chart $\phi_{>0}$ sending $(d,z)\in T(\K_{>0})\times \K_{>0}^N$ to $\phi_{>0}(d,z)\in Z_{d}(\K_{>0})$. Let $d_1,\dotsc, d_{n+1}$ denote the coordinates of $d$. We obtain a family of Laurent polynomials $\mathcal W_{d}(z)=(W\circ\phi_{>0})(d,z)$ 
whose coefficients are positive Laurent polynomials in  $q_i=\frac{d_{i+1}}{d_i}$, and we write 
\[
\mathcal W_{q_\bullet}(z):=\mathcal W_{d}(z)=W\circ\phi_{>0}(d,z),
\] 
where $q_\bullet=(q_1,\dotsc, q_n)\in\K_{>0}^n$. The Newton polytope of the Laurent polynomial $\mathcal W_{q_\bullet}$ is full-dimensional with $0$ in the interior and independent of the choice of $q_\bullet$.  
\end{prop}

\begin{remark}
An example of a fibered positive chart is given  in \cite[Section~9]{rietsch}. The associated Laurent polynomial $\mathcal W_d$ is also computed in  \cite{rietsch}. Namely, it is shown to recover the Laurent polynomial superpotential for $SL_{n+1}/B_+$ proposed by Givental~\cite{Givental:QToda}.     
\end{remark}
Without reference to this specific chart \cref{p:WdNP} can be proved as follows. 
\begin{proof}[Proof of \cref{p:WdNP}]
By \cref{l:Winposchart} we have that the composition $W\circ\phi_{>0}(d,z)$ is a Laurent polynomial with positive integer coefficients. Therefore $\mathcal W_d$ is a Laurent polynomial whose coefficients are Laurent polynomials with positive integer coefficients in the coordinates of $d$. We may change coordinates unimodularly from $(d_1,\dotsc, d_{n+1})$ to $(d_1,q_1,\dotsc, q_n)$. It follows from \cref{d:GLmirror} that we have $W(b)=W(zb)$ for any central element $z$, so that $\mathcal W_d=\mathcal W_{zd}$ is independent of rescaling of $d$. Therefore the coefficients of $\mathcal W_d$ must actually be positive Laurent polynomials in the parameters $q_i(d)$. This implies the first claim. From the positivity property it also follows that the Newton polytope of $\mathcal W_d$ does not depend on the choice of $d\in T(\K_{>0})$. Now suppose $\K_{>0}=\R_{>0}$ and choose $d$ to be the identity matrix. Then $\mathcal W_d$ has a unique \textit{real} positive critical point by the combination of Theorems~\ref{t:Zcrit} and \ref{t:ZcritposR}. (Namely this critical point represents the unique $b\in\Toeplitz_{n+1}(\R_{>0})$ with $\mathbf d(b)$ equal to the identity matrix). \cref{c:unique} implies that the Newton polytope of $\mathcal W_d$ is full-dimensional with $0$ in the interior. 
\end{proof}

The proof of the following corollary will complete the proof of \cref{t:PuiseauxFinite}. Note that the restriction  $\mathbf d|_{\Toeplitz_{n+1}(\KK_{>0})}$ is precisely the map $\Delta^{GL}_{>0}$ from \cref{t:PuiseauxFinite}. 
\begin{cor}\label{c:WdNP}
Let $\K=\KK$ and $\phi_{>0}$ a fibered positive chart for $Z$. For any $d\in T(\KK_{>0})$ the Laurent polynomial $\mathcal W_{d}(z)=W\circ\phi_{>0}(d,z)$ has a unique critical point $p_{crit}(d)$ in $\KK_{>0}^N$. This critical point represents a Toepliz matrix, namely 
\[
\phi_{>0}(d,p_{crit}(d))\in \Toeplitz_{n+1}(\KK_{>0}).
\]
Moreover, we have a homeomorphism
\[
\begin{array}{cccc}
b_{crit}:& T(\KK_{>0})&\to& \Toeplitz_{n+1}(\KK_{>0})\\
& d &\mapsto & \phi_{>0}(d,p_{crit}(d)),
\end{array}
\]
and the map $b_{crit}$ is the inverse of $\mathbf d|_{\Toeplitz_{n+1}(\KK_{>0})}$. 
\end{cor}

\begin{proof}
The existence and uniqueness of the positive critical point $p_{crit}(d)$ follows from \cref{p:WdNP} and \cref{t:JuddR}. A proof that the critical point conditions for the superpotential are equivalent to the Toeplitz conditions for the matrix $\phi_{>0}(d,z)$ is given in \cite{rietschNagoya} together with \cite[Section~9]{rietsch}, algebraically over $\Q$. 
By \cref{p:WdNP} and \cref{c:continuity} we have that the map $b_{crit}$ is a well-defined continuous map. Its inverse is by construction precisely given by $\mathbf d$, which is also a continuous map. Therefore $b_{crit}$ is a homeomorphism. 
\end{proof}

Theorem~\ref{t:PuiseauxFinite} is now proved. We finish this section by explicitly connecting up the \cref{c:WdNP} with the unipotent case from \cref{r:SL-PGL-FiniteParam}. Consider the compositions
\[\begin{tikzcd}
	& {T^{SL}(\KK_{>0})} \\
	{\mathbf q:\ToeplitzU_{n+1}(\KK_{>0})} && {T^{PGL}(\KK_{>0})} & {\KK_{>0}^n.}
	\arrow["\mathrm{pr}",from=1-2, to=2-3]
	\arrow["{\Delta^{SL}_{>0}}", from=2-1, to=1-2]
	\arrow["\sim", from=2-3, to=2-4]
	\arrow["\Delta^{PGL}_{>0}", from=2-1, to=2-3]
\end{tikzcd}\]
The far right-hand isomorphism is given by choosing coordinates $q_i=\frac{d_{i+1}}{d_i}$ for $T^{PGL}(\KK_{>0})$. 
All of the maps in the above commutative diagram are homeomorphisms. The inverse of the map `$\mathrm{pr}$' is a Puiseaux monomial map in the coordinates of $T^{PGL}(\KK_{>0})$. Moreover, the inverses of $\Delta_{>0}^{SL}$ and $\Delta_{>0}^{PGL}$ can be described as `critical point maps' 
in the same way as $\Delta_{>0}^{GL}$ is in \cref{c:WdNP}. Namely, the inverse of $\Delta_{>0}^{SL}$  is given by the restriction of the map $b_{crit}$ from \cref{c:WdNP}, and for $\Delta_{>0}^{PGL}$ we have the following interpretation, noting that  the variety of totally positive upper-triangular Toeplitz matrices in $PGL_{n+1}$ (the stabiliser in $PGL_{n+1}$ of the principal nilpotent $E=\sum e_i$) is naturally identified with $\ToeplitzU_{n+1}$. 
\begin{cor}\label{c:PGLcritptmap}
Let $\mathcal W_{q_\bullet}$ be the Laurent polynomial defined in \cref{p:WdNP}, and $\bar p_{crit}:\KK_{>0}^{n}\to \KK_{>0}^N$ the map sending $q_\bullet$ to the unique positive critical point of $\mathcal W_{q_\bullet}$. 
The fibered positive chart $\phi_{>0}$ for $Z$ gives rise to a fibered positive chart for $Z^{PGL}$ via $\phi^{PGL}_{>0}(q_\bullet,z):=[\phi_{>0}(d,z)]$,
where $[g]$ denotes the element in $PGL_{n+1}(\KK)$ represented by the matrix $g$, and  $d\in T(\KK_{>0})$ is chosen such that $q_i=\frac{d_{i+1}}{d_i}$ for all $i$. Then
\begin{equation}
b_{crit}^{PGL}(q_\bullet):=\phi^{PGL}_{>0}(q_\bullet,\bar p_{\crit}(q_\bullet))
\end{equation} 
defines the inverse of $\Delta^{PGL}_{>0}$.
\end{cor}
\begin{proof}
This corollary is a straightforward consequence of \cref{c:WdNP} together with \cref{p:WdNP}. 
\end{proof}

\section{Parameterization theorem for $\ToeplitzU_{n+1}(\R_{\min{}})$}\label{s:TropToeplitzFinite}

In Section~\ref{s:ToeplitzFinite} we considered totally positive Toeplitz matrices over the field of generalised Puiseaux series. In this next section we tropicalise the results from \cref{s:ToeplitzFinite}. Namely, we wish to  replace $\KK_{>0}$ by the semifield $\R_{\min}$. We begin by recalling the definition of $U_+(\R_{\min{}})$ following Lusztig \cite{Lusztig94,Lusztig97,Lusztig:QuantumGroupsBook}, who first  introduced this set in a positive integral version, $U_+(\Z^{\ge 0}_{\min{}})$, in order to parametrise the canonical basis of the (Langlands dual) quantized universal enveloping algebra $\mathcal U_q^-$. 

\subsection{}\label{s:tropdefs}
Recall the definition of a positive structure from \cref{s:posstructure}. We describe the tropicalisation construction in the form that we will apply it.

\begin{defn}[Tropicalisation]\label{d:tropgen} Let $(\RR,\RR_{>0},\Val)$ be a ring with valuative positive structure. Let us assume additionally that $\Val:\RR_{>0}\to\Ktrop$  is surjective for simplicity (otherwise $\Rmin$ should be replaced with the image of $\Val$ later on). Let $X$ be a variety with a positive structure as in \cref{d:posstructure}.  Suppose $\phi_{>0}: \RR_{>0}^N\to X(\RR_{>0})$ is one of the positive charts. Let $u=\phi_{>0}(z_1,\dotsc, z_N)$ and $u'=\phi_{>0}(z'_1,\dotsc, z'_N)$. We define an equivalence relation on $X(\RR_{>0})$ by
\[
u\sim u'\quad  \iff\quad \Val(z_i)=\Val(z'_i)\quad \text{ for $i=1\dotsc, N$}.
\]  
It follows from the fact that different positive charts are related by admissible transformations that the equivalence relation $\sim$ is well-defined and independent of the choice of chart $\phi_{>0}$. We define $X(\Ktrop)$ as the quotient space $X(\RR_{>0})/\sim$. We use the word `tropicalising' to mean `quotienting by $\sim$'. Note that with this definition comes a quotient map $X(\RR_{>0})\to X(\Rmin)$ that we may call $\pitrop$. 

For every positive chart $\phi_{>0}$ of $X$ there is a \textit{tropical chart}, which is given by the bijection 
\[
\begin{array}{cccc}
\phi_{trop}:&X(\Ktrop)&\longrightarrow &\Ktrop^{N}\\
&{[\phi_{>0}(z_1,\dotsc,z_N)]}&\mapsto&(\Val(z_1),\dotsc, \Val(z_N)).
\end{array}
\]
The bi-admissible transformations between different positive charts of $X(\RR_{>0})$ tropicalise to bi-admissible transformations of $\Ktrop^N$. Note that an admissible map over $\Ktrop$ is a piecewise-linear map constructed out of $+,-$ and $\min$.

More generally we have that any admissible map  $X\to X'$ between spaces with positive structures tropicalises to a map $X(\Ktrop)\to X'(\Ktrop)$ which, in terms of any choice of tropical charts, is given by an admissible map over $\Rmin$.
 \end{defn}

\begin{defn}\label{d:subspacetrop} For any subset $Y$ of $X$ we may set $Y(\RR_{>0}):=Y\cap X(\RR_{>0})$ and define  $Y(\Ktrop)$ to be the image of the composition
\[
Y(\RR_{>0})\hookrightarrow X(\RR_{>0}) \rightarrow X(\Ktrop).
\]  
Note that $Y$ is not assumed to have a positive structure of its own, and in this case $Y(\Ktrop)$ does not have tropical charts. It is just a subset of $X(\Ktrop)$ (which can be described in the different tropical charts of $X(\Ktrop)$). 
\end{defn}

Recall now the notations from Section~\ref{s:Semifields}. In particular, we have the definition of $U_+(\RR_{>0})$ and its positive charts, see \cref{s:posstructure}. We obtain the tropical version $U_+(\Ktrop)$ of $U_+$ by applying \cref{d:tropgen}. We now tropicalise $\ToeplitzU_{n+1}$ by viewing it as embedded in $U_+$.

\begin{defn}[Tropical Toeplitz matrices]\label{d:tropToeplitz}
Let $(\RR,\RR_{>0},\Val)$ be a ring with a valuative positive structure.  We define $\ToeplitzU_{n+1}(\R_{\min})$ to be  the image of the composition 
\[
\ToeplitzU_{n+1}(\RR_{>0})\hookrightarrow U_+(\RR_{>0})\to  U_+(\R_{\min}),
\]
following \cref{d:tropgen} and \cref{d:subspacetrop}. 
\end{defn}

The remainder of this section is devoted to describing concretely the set $\ToeplitzU_{n+1}(\R_{\min{}})$ and giving the tropical version of Theorem~\ref{t:PuiseauxFinite}.

\subsection{}\label{s:U+poscharts}
Recall  the notations from \cref{s:AlgGroup-notations}, as well as the basic positive charts of $U_+$ associated to reduced expressions of $w_0$, see \cref{s:Semifields}. For every reduced expression there is an associated normal  ordering of the set of positive roots $R^+$,  see \cite{Bourbaki:L4}. Namely, encode the reduced expression $w_0=s_{i_1}\dotsc s_{i_N}$ via the sequence $\mathbf i=(i_1,\dotsc, i_N)$ of indices.  Then 
\begin{equation}\label{e:posrootsordering}
\alpha^{\mathbf i}_1:=\alpha_{i_1},\ \ \alpha^{\mathbf i}_2:=s_{i_1}\alpha_{i_2},\ \ \dotsc, \ \ \alpha^{\mathbf i}_n=s_{i_1}\dotsc s_{i_{n-1}}\alpha_{i_n}.
\end{equation}
A special role will be played by the reduced expression from  Lemma~\ref{l:mijFormula},
\[
\mathbf i_0=(n,n-1,\dotsc, 1; n,\dotsc, 2;\dotsc; n,n-2;n).
\] 
One can naturally arrange the positive roots into an upper-triangular array according to the weight spaces that they correspond to in $\mathfrak u_+$, and label the diagonal boundary by the $\ep_i$ so that the positive roots can be read off using $\alpha_{ij}=\ep_i-\ep_j$. Let us rotate this picture, as this will be convenient later on. For example, the positive roots of $SL_4$ ordered according to the reduced expression $\mathbf i_0=(3,2,1;3,2;3)$ can be read off column by column downwards from the top in following table
\begin{equation*}
\begin{tikzpicture}[scale=0.6]
\node[fill=white] at    (0.5,3.55) {$\ep_4$};
\node[fill=white] at    (0.5,2.55) {$\alpha_{34}$}  ;
\node[fill=white] at    (0.5,1.55) {$\alpha_{24}$}  ;
\node[fill=white] at    (0.5,0.55) {$\alpha_{14}$}  ;
\node[fill=white] at    (1.8,2.55) {$\ep_3$};
c\node[fill=white] at    (1.8,1.55) {$\alpha_{23}$}  ;\node[fill=white] at    (1.8,0.55) {$\alpha_{13}$}  ;
\node[fill=white] at    (3.0,1.55) {$\ep_2$};
\node[fill=white] at    (3,0.55) {$\alpha_{12}$}  ;
\node[fill=white] at    (4.2,0.55) {$\ep_1$};
\node[fill=white] at (5.5,1.5) {} ;
\node[fill=white] at (15,1.5) { $\alpha^{\mathbf i_0}_1=\alpha_{34},\ \alpha^{\mathbf i_0}_2=\alpha_{24},\ \alpha^{\mathbf i_0}_3=\alpha_{14},\ \alpha^{\mathbf i_0}_4=\alpha_{23},\ \alpha^{\mathbf i_0}_5=\alpha_{13},\ \alpha^{\mathbf i_0}_6=\alpha_{12}$.} ;
\draw[black] (-.2,0.1) -- (-.2,3.1) -- (1.15,3.1) -- (1.15,0.1) -- (3.6,0.1) -- (3.6,1.1)--(-.2,1.1)--(-.2,0.1)--(2.4,0.1)--(2.4,2.1)--(-.2,2.1);
 \end{tikzpicture}
\end{equation*}

We may graphically depict an element of $U_+(\RR_{>0})$ in terms of the positive chart  $\phi_{>0}^{\mathbf i_0}$ as a lower-triangular array with entries in $\RR_{>0}$. For example, 
\begin{equation}\label{e:coordinates}
\begin{tikzpicture}[scale=0.6]
\node[fill=white] at    (0.5,2.55) {$z_1$}  ;
\node[fill=white] at    (0.5,1.55) {$z_2$}  ;
\node[fill=white] at    (0.5,0.55) {$z_3$}  ;
\node[fill=white] at    (1.8,1.55) {$z_4$}  ;
\node[fill=white] at    (1.8,0.55) {$z_5$}  ;
\node[fill=white] at    (3,0.55) {$z_6$}  ;
\node[fill=white] at (5.5,1.5) { $\longleftrightarrow$} ;
\node[fill=white] at (15,1.5) { $\phi_{>0}^{\mathbf i_0}((z_i)_i)=x_{3}(z_1)x_{2}(z_2)x_{1}(z_3)x_{3}(z_4)x_{2}(z_5)x_{3}(z_6)$.} ;
\draw[black] (-.2,0.1) -- (-.2,3.1) -- (1.15,3.1) -- (1.15,0.1) -- (3.6,0.1) -- (3.6,1.1)--(-.2,1.1)--(-.2,0.1)--(2.4,0.1)--(2.4,2.1)--(-.2,2.1);
 \end{tikzpicture}
\end{equation}
\begin{defn}[The map $\mathcal L$]{\label{Lweightmap}} We have the following map $\mathcal L:U_{+}(\R_{\min})\to \mathfrak h_{PSL_{n+1}}^*(\R)$ that we call the Lusztig weight map,
\[
\begin{array}{rccc}
\mathcal L:&U_{+}(\R_{\min})&\to &\mathfrak h_{PSL_{n+1}}^*(\R),\\
&[\phi_{>0}^{\mathbf i}((z_i)_i)],&\mapsto &\sum_{i=1}^N Z_i\alpha^{\mathbf i}_i,
\end{array}
\]
where $Z_i=\Val(z_i)$. This definition is the extension from $\Z_{>0}$ to $\R$ of the weight map from \cite[2.8]{Lusztig:cbJAMS}.  If $[\phi_{>0}^{\mathbf i}((z_i)_i)]\in U_{+}(\Z^{> 0}_{\min})$ then its image under $\mathcal L$ lies in the root lattice and can be interpreted as the weight of the canonical basis element of $\mathcal U_-$ parameterised by $[\phi_{>0}^{\mathbf i}((z_i)_i)]$ as in \cite{Lusztig94}. 
\end{defn}

\begin{remark}\label{r:tropicaltransformation} We recall from  \cite{Lusztig:cbJAMS} that the map $\mathcal L$  is well-defined independently of the choice of reduced expression $\mathbf i$, as can be checked using the piecewise-linear transformation which is the tropicalisation of~\eqref{e:braidbirationalA2}. Namely, corresponding to the braid relation $\mathbf i=(\dotsc, i_{k-1}, i_{k}, i_{k+1},\dotsc)=(\dotsc, i,j, i,\dotsc)\mapsto \mathbf i'=(\dotsc, j,i, j,\dotsc)$,  we record the tropicalisation of  \eqref{e:braidbirationalA2} to be the transformation given by
\begin{equation}\label{e:tropicalbraidA2}
\begin{array}{lcl}
Z'_{k-1}&=&Z_k+Z_{k+1}-\min(Z_{k-1},Z_{k+1}),\\
Z'_k= &=& \min(Z_{k-1},Z_{k+1}),\\
Z'_{k+1}&=&Z_{k-1}+Z_k-\min(Z_{k-1},Z_{k+1}),
\end{array}
\end{equation}
and $Z'_\ell=Z_\ell$ otherwise. This transforms the tropical chart for $\mathbf i$ into the one for $\mathbf i'$. 
 \end{remark}
The  map $\mathbf d_{uni}$ from \cref{d:duni} is a geometrisation of the Lusztig weight map by the following 
lemma. 
\begin{lemma}\label{l:dunilemma} Let $G=SL_{n+1}$. Recall the map  $\mathbf d_{uni}:U_+\cap B_-\dot w_0 B_-\to T^{SL_{n+1}}$ sending $u$ to the diagonal matrix with entries
\begin{equation}\label{e:di-again}
d_1=\Delta_{1}(u),\   
d_2=\frac{\Delta_{2}(u)}{\Delta_{1}(u)},\ 
\dotsc,
\ d_{n+1}= \frac{1}{\Delta_{n}(u)}.
\end{equation}
The tropicalisation $\Trop(\mathbf d_{uni}):U_+(\R_{\min})\to\mathfrak h_{SL_{n+1}}(\R)$ is precisely Lusztig's weight map $\mathcal L$ (after canonically identifying $\mathfrak h_{SL_{n+1}}$ with $\mathfrak h_{PSL_{n+1}}^*$). 
\end{lemma}
\begin{remark}\label{r:JuddvsLWeight} This lemma is related to the result going back to \cite{Judd:Flag} about $\Trop(\mathbf d)$ at critical points in terms of 	`ideal fillings'. However, in \cite{Judd:Flag} there was no connection with $U_+(\Rmin)$. See also  \cref{r:JuddLudenbach} for a more detailed comparison. The map $\mathbf d$ was also tropicalised in different coordinate  charts and with an interpretation of `highest weight map' in a crystal basis context in  \cite{BK:GeometricCrystalsII}.
\end{remark}
Before proving the lemma we make an adjustment to our notation. 
\begin{defn}\label{d:standardcoordFin} From now on, instead of labelling the $\phi^{\mathbf i_0}_{>0}$ coordinates of $U_+(\RR_{>0})$ by $z_1,z_2,\dotsc, z_N$, we will instead label them $m_{ij}$ indexed by the set 
 \begin{equation}\label{e:Sn+1}
  \mathcal S_{\le n+1}:=\{(i,j)\mid i+j\le n+1\},
  \end{equation}  
 with $i$ indicating the row (counting from the bottom) and $j$ the column (left to right) of  the coordinate in the triangular tableau \eqref{e:coordinates}.
 We will take the chart $\phi^{\mathbf i_0}$ to be our standard chart, and call the coordinates $(m_{ij})_{i,j}$ the {\it standard coordinates} for $U_+$. 

We represent an element of $U_+(\R_{\min})$ analogously  using the associated tropical chart $\phi^{\mathbf i_0}_{trop}$. Namely, we denote the tropical coordinate associated to $m_{ij}$ by $M_{ij}$, and we also arrange the $M_{ij}$  into a lower-triangular tableau, in the same way as the $m_{ij}$, to represent a tropical point $[u]\in U_+(\Rmin)$ in terms of its (tropical) standard coordinates. 
\end{defn}
\begin{example} \label{ex:UtropFin}
For $n=3$ the coordinate labellings translate as follows  
\begin{center}\label{e:indexationchange}
\begin{tikzpicture}[scale=0.6]
\node[fill=white] at    (12.5,2.55) {$\tiny{m_{31}}$}  ;
\node[fill=white] at    (12.5,1.55) {$m_{21}$}  ;
\node[fill=white] at    (12.5,0.55) {$m_{11}$}  ;
\node[fill=white] at    (13.8,1.55) {$m_{22}$}  ;\node[fill=white] at    (13.8,0.55) {$m_{12}$}  ;
\node[fill=white] at    (15,0.55) {$m_{13}$}  ;
\draw[black] (11.8,0.1) -- (11.8,3.1) -- (13.15,3.1) -- (13.15,0.1) -- (15.6,0.1) -- (15.6,1.1)--(11.8,1.1)--(11.8,0.1)--(14.4,0.1)--(14.4,2.1)--(11.8,2.1);
\node[fill=white] at    (0.5,2.55) {$\tiny{z_1}$}  ;
\node[fill=white] at    (0.5,1.55) {$z_2$}  ;
\node[fill=white] at    (0.5,0.55) {$z_3$}  ;
\node[fill=white] at    (1.8,1.55) {$z_4$}  ;\node[fill=white] at    (1.8,0.55) {$z_5$}  ;
\node[fill=white] at    (3,0.55) {$z_6$}  ;
\node[fill=white] at (7.5,1.5) { $\longleftrightarrow$} ;
\draw[black] (-.2,0.1) -- (-.2,3.1) -- (1.15,3.1) -- (1.15,0.1) -- (3.6,0.1) -- (3.6,1.1)--(-.2,1.1)--(-.2,0.1)--(2.4,0.1)--(2.4,2.1)--(-.2,2.1);
 \end{tikzpicture},
\end{center}
giving  $\phi_{>0}^{\mathbf i_0}((m_{ij})_{i,j})=x_{3}(m_{31})x_{2}(m_{21})x_{1}(m_{11})x_{3}(m_{22})x_{2}(m_{12})x_{3}(m_{13})$.
The tropical point $[u]$ for $u=\phi_{>0}^{\mathbf i_0}((m_{ij})_{i,j})$ has tropical coordinates $(M_{ij})_{i,j}$ given by $M_{ij}=\Val(m_{ij})$. To describe the Lusztig weight map in these coordinates, we line up our graphical representation of $R_+$, shown below on the left, with our tableau of tropical standard coordinates for $[u]\in U_+(\R_{\min{}})$, shown on the right,
\begin{center}
\begin{tikzpicture}[scale=0.6]
\node[fill=white] at    (0.5,3.55) {$\ep_4$};
\node[fill=white] at    (0.5,2.55) {$\alpha_{34}$}  ;
\node[fill=white] at    (0.5,1.55) {$\alpha_{24}$}  ;
\node[fill=white] at    (0.5,0.55) {$\alpha_{14}$}  ;
\node[fill=white] at    (1.8,2.55) {$\ep_3$};
c\node[fill=white] at    (1.8,1.55) {$\alpha_{23}$}  ;\node[fill=white] at    (1.8,0.55) {$\alpha_{13}$}  ;
\node[fill=white] at    (3.0,1.55) {$\ep_2$};
\node[fill=white] at    (3,0.55) {$\alpha_{12}$}  ;
\node[fill=white] at    (4.2,0.55) {$\ep_1$,};
\draw[black] (-.2,0.1) -- (-.2,3.1) -- (1.15,3.1) -- (1.15,0.1) -- (3.6,0.1) -- (3.6,1.1)--(-.2,1.1)--(-.2,0.1)--(2.4,0.1)--(2.4,2.1)--(-.2,2.1);
 \end{tikzpicture}\qquad\qquad\qquad\quad \qquad\qquad
 \begin{tikzpicture}[scale=0.6]
\node[fill=white] at    (0.5,2.54) {$M_{31}$}  ;
\node[fill=white] at    (0.5,1.54) {$M_{21}$}  ;
\node[fill=white] at    (0.5,0.55) {$M_{11}$}  ;
\node[fill=white] at    (1.8,1.54) {$M_{22}$}  ;\node[fill=white] at    (1.8,0.55) {$M_{12}$}  ;
\node[fill=white] at    (3,0.55) {$M_{13}$}  ;
\draw[black] (-.2,0.1) -- (-.2,3.1) -- (1.15,3.1) -- (1.15,0.1) -- (3.6,0.1) -- (3.6,1.1)--(-.2,1.1)--(-.2,0.1)--(2.4,0.1)--(2.4,2.1)--(-.2,2.1);
 \end{tikzpicture}.
\end{center}
We then obtain 
\[
\mathcal L((M_{ij})_{i,j})=M_{31}\alpha_{34}+M_{21}\alpha_{24}+M_{11}\alpha_{14}+M_{22}\alpha_{23}+M_{12}\alpha_{13}+M_{13}\alpha_{12}.
\]
\end{example}

\begin{proof}[{Proof of \cref{l:dunilemma}}]
Recall that $\Delta_i(u)=\Minor^{[i]}_{[i]+n-i+1}(u)$. It suffices to consider $u=\phi^{\mathbf i_0}((m_{ij})_{ij})$, so that we have by \cref{l:mijFormula} that
\[
\frac{\Delta_{i}(u)}{\Delta_{i-1}(u)}=\frac{\prod_{k=1}^{i}\prod_{r=1}^{n-i+1}\ m_{kr}}{\prod_{k=1}^{i-1}\prod_{r=1}^{n-i+2} m_{kr}}=(\prod_{r=1}^{n-i+1}\ m_{{i},r})(\prod_{k=1}^{i-1}m_{k,n-i+2})\inv.
\]
Therefore
\[
\Val\left(\frac{\Delta_{i}(u)}{\Delta_{i-1}(u)}\right)=\sum_{r=1}^{n-i+1}\ M_{{i},r}-\sum_{k=1}^{i-1}M_{k,n-i+2}
\]
with $M_{ij}=\Val(m_{ij})$. On the other hand, if we express the positive roots as $\alpha_{i,k+1}=\ep_i-\ep_{k+1}$, then the Lusztig weight map in terms of the $M_{ij}$ coordinates is given by 
\[
\mathcal L((M_{ij})_{i,j})=\sum_{(i,j)\in \mathcal S_{\le n+1}} M_{ij}(\ep_{i}-\ep_{n-j+2})=\sum_{i=1}^{n+1}( \sum_{r=1}^{n-i+1}\ M_{{i},r}-\sum_{k=1}^{i-1}M_{k,n-i+2})\,\ep_i.
\]
This implies the lemma.
\end{proof}

\subsection{Two theorems about finite tropical Toeplitz matrices}
The remainder of Section~\ref{s:TropToeplitzFinite} will be concerned with the following two theorems that we can now state about finite tropical Toeplitz matrices. These results  are building on \cite{Ludenbach} and \cite{Judd:Flag}. We start with the following definition. Note that this is a negated version of what is called an `ideal filling' in \cite{Judd:Flag}, see~\cref{r:JuddLudenbach}.

\begin{defn}[Min-ideal fillings] \label{d:minideal}
Let a \textit{min-ideal filling} be a lower-triangular tableau of real numbers with the property that  any entry $C\in \R$ not along the top diagonal is the minimum of its upper and its right-hand neighbour. Pictorially, we have the following rule,
\begin{center}
\begin{tikzpicture}[scale=0.6]
\node[fill=white] at    (0.5,2.6) {$A$}  ;
\node[fill=white] at    (0.5,1.55) {$B$}  ;
\node[fill=white] at    (1.7,1.6) {$C$}  ;
\draw[black] (-.1,1.05)--(-.1,3.1) -- (1.15,3.1) -- (1.15,1.05)-- (2.3,1.05)--(-.1,1.05)--(-.1,2.1)--(2.3,2.1)--(2.3,1.05); 
\node[fill=white] at (5.5,2) {$\implies$};
\node[fill=white] at (10,2) {$\min(A,C)=B$,};
 \end{tikzpicture}
\end{center}
for all minimal triangular sub-arrays. 
 We write $\Imin_{n+1}$ for the set of min-ideal fillings with $n$ rows and columns. We interpret $\Imin_{n+1}$ as a subset $U_+(\R_{\min})$ for $PSL_{n+1}$ via  \cref{d:standardcoordFin}, compare \cref{ex:UtropFin}.
\end{defn}

\begin{remark}\label{r:meetsemilattice} The set of min-ideal fillings $\Imin_{n+1}$ forms a meet semilattice, with poset structure given by: $(M_{ij})_{i,j}\le (M'_{ij})_{i,j}$ if an only if $M_{ij}\le M'_{i,j}$ in $(\R,\le)$ for all $i,j\in \mathcal S_{\le n+1}$. The `meet' of two min-ideal fillings is the coordinate-wise minimum,
\[(M_{ij})_{i,j}\wedge (M'_{ij})_{i,j} =(\min(M_{ij},M'_{ij}))_{i,j},\]
which is straightforwardly again a min-ideal filling. 
\end{remark}

\begin{remark}\label{r:JuddLudenbach} In the work of Judd \textit{ideal fillings} (which are upper-triangular arrays obeying a $\max$-condition, instead of a $\min$-condition) arise in connection with tropical critical points of the superpotential $W$ written as the Laurent polynomial  first proposed by Givental~\cite{Givental:QToda}, see~\cite{rietsch}, and with specific values $q_i=t^{\lambda_{i}-\lambda_{i+1}}$ where $\lambda$ is a dominant weight. The interpretation is that for any dominant $\lambda$ there is a polytope whose lattice points parametrise a basis of the representation of highest weight $\lambda$, and there is a unique point in the interior of this polytope which is the `tropical critical point' of a  Laurent polynomial $W_{t^\lambda}$ which bijectively corresponds to an ideal filling. The highest weight $\lambda=\lambda(\mathrm n)$ can be read off the ideal filling $\mathrm n$, 
and  is a dominant weight interpreted as the highest weight of a representation, and the components $n_{ij}$ lie in $\mathbb Q_{\ge 0}$. Accordingly, the tropical critical points considered in \cite{Judd:Flag} correspond to elements of $U_+(\mathbb Q^{\le 0}_{\min})$, noting that ideal-fillings are related to min-ideal fillings by a change of sign. The canonical basis, on the other hand, is indexed by  $U_+(\Z^{\ge 0}_{\min})$. The extension of the map $\lambda$ to $U_+(\R_{\min})$ as Lusztig's weight map as in \cref{l:dunilemma}
is not visible in the original construction from \cite{Judd:Flag}. In  \cite{Ludenbach}, reciprocal inverses of the $\phi_{>0}^{\mathbf i_0}$ coordinates are introduced and called `ideal coordinates', because, as proved in \cite{Ludenbach}, at critical points they de-tropicalise the ideal fillings from \cite{Judd:Flag}. Here we extend these previous works, reversing the signs, and making the connection to the canonical basis.
\end{remark}

The theorem below describes the set $\ToeplitzU_{n+1}(\R_{\min})$. It is essentially a reformulation of a result stated in the introduction of \cite{Ludenbach}, namely Theorem~1.0.1. The proof outlined there covers the tropicalisation of Toeplitz matrices $u=\phi^{\mathbf i_0}((z_i)_i)$ for which the entries of $\mathbf d_{uni}(u)$ are of the form $t^{\lambda_i}$ as in \cref{r:JuddLudenbach}. We will prove this now using \cref{t:PuiseauxFinite}. 

\begin{theorem} \label{t:MinIdealIsToepl} 
The subset $\ToeplitzU_{n+1}(\R_{\min})$ 
 of  $U_+(\R_{\min})$ is precisely the set $\Imin_{n+1}$ of min-ideal fillings.  
\end{theorem}
\begin{remark}
This theorem says a priori that the description of $\ToeplitzU_{n+1}(\R_{\min})$ in terms of the special tropical chart associated to $\mathbf i_0$ is given by the min-ideal condition. But in fact, an ideal filling $M$ of $\Imin_{n+1}$ is naturally an element of $\R^{R_+}$, and it turns out that its image $\phi^{\mathbf i}_{trop}(M)$ from this perspective is independent of $\mathbf i$, see \cite[Proposition~5.5.6]{Ludenbach}. Therefore a min-ideal filling is in a sense a canonical description of an element of $\ToeplitzU_{n+1}(\Rmin)$. This is another feature that makes the subset  $\ToeplitzU_{n+1}(\R_{\min})$ of $U_+(\R_{\min})$ very special.
\end{remark}

The second theorem of this section is the tropical parametrisation theorem for tropical Toeplitz matrices. 
\begin{theorem}\label{t:TropFinite}
The restriction 
\[\mathcal L|_{\ToeplitzU_{n+1}}:\ToeplitzU_{n+1}(\R_{\min})\to \mathfrak h^{*}_{PSL_{n+1}}(\R)\]
of Lusztig's weight map $\mathcal L:U_+(\R_{\min})\to \mathfrak h^{*}_{PSL_{n+1}}(\R)$ is a bijection. 
\end{theorem}

We will prove  \cref{t:MinIdealIsToepl} and \cref{t:TropFinite}  in reverse order, making use of \cref{t:TropFinite} in the proof of \cref{t:MinIdealIsToepl}.
The proof of \cref{t:TropFinite} arises for us as a corollary of Theorem~\ref{t:PuiseauxFinite}, the parameterisation theorem for $\Toeplitz_{n+1}(\KK_{>0})$.
 The proof of  \cref{t:MinIdealIsToepl} relies also on a proof found in \cite{Judd:Flag} that shows that there is a \textit{unique} ideal filling for a given $\lambda\in\mathfrak h^*_{PSL}(\R)$. While Judd in fact assumes in his work that $\lambda$ is a dominant weight, this proof works in this greater generality.  Indeed, though neither \cref{t:MinIdealIsToepl} nor \cref{t:TropFinite} are proved in  \cite{Judd:Flag}, the proofs from \cite{Judd:Flag} appear to suffice to show the remarkable combinatorial corollary that $\Imin_{n+1}\to \mathfrak h^*_{PSL_{n+1}}(\R)$ is a bijection. 

\begin{proof}[Proof of Theorem~\ref{t:TropFinite}]

We have the following commutative diagram,
\[\begin{tikzcd}
	{U_+(\KK_{>0})} & {\ToeplitzU_{n+1}(\KK_{>0})} & {T^{PGL}(\KK_{>0})} \\
	{U_+(\R_{\min})} & {\ToeplitzU_{n+1}(\R_{\min})} & {\mathfrak h^*_{PSL_{n+1}}(\R) .}
	\arrow[from=1-1, to=2-1]
	\arrow[hook', from=1-2, to=1-1]
	\arrow["\Delta^{PGL}_{>0}", from=1-2, to=1-3]
	\arrow[from=1-2, to=2-2]
	\arrow[from=1-3, to=2-3]
	\arrow[hook', from=2-2, to=2-1]
	\arrow["\mathcal L|_{\ToeplitzU}",from=2-2, to=2-3]
\end{tikzcd}\]
The top right-hand map $\Delta^{PGL}_{>0}$ is the map from \eqref{e:PGLFiniteParam}, which is a homeomorphism by \cref{t:PuiseauxFinite} and \cref{r:SL-PGL-FiniteParam}, with inverse given in \cref{c:PGLcritptmap}. The bottom right-hand map is, by \cref{l:dunilemma}, the restriction of the Lusztig weight map. The two outer vertical maps are taking the quotient using the equivalence relation from \cref{d:tropgen} and the middle one is a restriction of the left-hand map. All of the vertical maps are surjective. It follows that $\mathcal L|_{\ToeplitzU}$ is surjective. Now we consider the inverse of the map $\Delta_{>0}^{PGL}$. By \cref{c:PGLcritptmap} this inverse is given by \[
q_\bullet\mapsto\phi_{>0}^{PGL}(q_\bullet,\bar p_{crit}(q_\bullet))
\]
where $\bar p_{crit}(q_\bullet)\in \KK_{>0}^N$ is the positive critical point of $\mathcal W_{q_\bullet}$. By the construction of the positive critical point, the valuation of $\bar p_{crit}(q_\bullet)$, namely the associated tropical point in $\R^N$, depends only on  $Q_\bullet:=(\Val(q_i))\in\R^n$, by \cite{JuddRietsch:24}. It follows that the associated tropical Toeplitz matrix, the equivalence class of $\phi_{>0}^{PGL}(q_\bullet,\bar p_{crit}(q_\bullet))$, also depends  only on $Q_\bullet$. This therefore constructs an inverse to $\mathcal L|_{\ToeplitzU}$, completing the proof.      
\end{proof}

\subsection{Tropical Toeplitz matrices}
In this section we prove Theorem~\ref{t:MinIdealIsToepl}.
We begin by giving an explicit description of $\Toeplitz_{n+1}(\RR_{>0})$ in terms of the standard coordinates $m_{ij}$ using a quiver that we now introduce. 

\begin{defn}[The quiver $Q_{n+1}$] \label{d:Qn+1} Consider a set of vertices $\mathcal V_{n+1}=\mathcal V_\bullet\cup\mathcal V_\circ$ and a set of arrows $\mathcal A$ defined as follows. 
\[
\begin{array}{ccl}
 \mathcal V_{\bullet}&=&\{v_{ij}\mid i,j\in \Z_{>0},  i+j\le n\},\\
 \mathcal V_{\circ}&=&\{v_{ij}\mid i,j\in \Z_{>0}, i+j= n+1\text{ or $i=j=0$}\}. 
 \end{array}
 \] 
We may also write $v_{ij}=\bullet_{(i,j)}$ if $v_{ij}\in\mathcal V_\bullet$, and $v_{ij}=\circ_{(i,j)}$ if $v_{ij}\in\mathcal V_\circ$. We picture the vertex $v_{ij}$ as located at the point with coordinates $(j,i)$ in the plane $\R^2$, and place an arrow from any $\bullet_{(i,j)}$ to its immediate neighbors above and to the right. Additionally, we let two arrows go from $\circ_{(0,0)}$ to $\bullet_{(1,1)}$ and one arrow from $\circ_{(0,0)}$ to each of the vertices $v_{1,j}$ and $v_{i,1}$, where $i,j>1$.  See \cref{f:QuiverExample} for an example.   
  We have a bijection between  $\mathcal S_{\le n+1}\cup\{(0,0)\}$, and $\mathcal V_{n+1}$ sending $v_{ij}$ to $(i,j)$, and $\mathcal S_{\le n+1}$ will be used as a convenient indexing set. 
\end{defn}

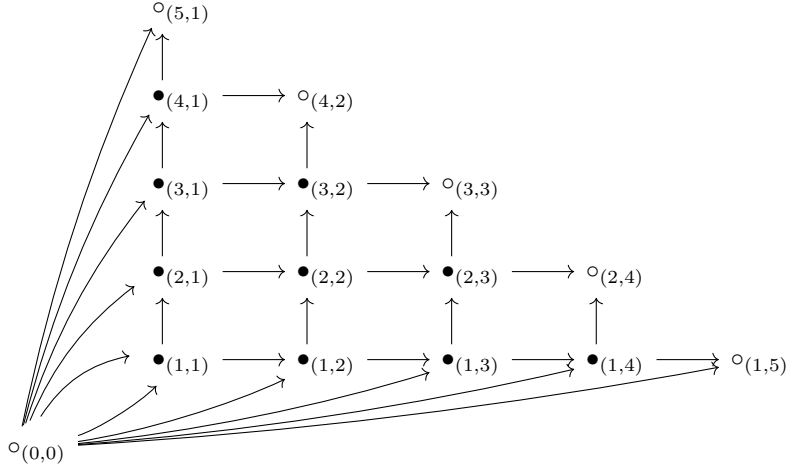
\begin{figure}[h!]
\[\begin{tikzcd}
	& { \circ_{(5,1)}} \\
	& { \bullet_{(4,1)}} & { \circ_{(4,2)}} \\
	& { \bullet_{(3,1)}} & { \bullet_{(3,2)}} & { \circ_{(3,3)}} \\
	& { \bullet_{(2,1)}} & { \bullet_{(2,2)}} & { \bullet_{(2,3)}} & { \circ_{(2,4)}} \\
	& { \bullet_{(1,1)}} & { \bullet_{(1,2)}} & { \bullet_{(1,3)}} & { \bullet_{(1,4)}} & { \circ_{(1,5)}} \\
	\circ_{(0,0)}
	\arrow[shift left=3,from=2-2, to=1-2]
	\arrow[from=2-2, to=2-3]
	\arrow[shift left=3,from=3-2, to=2-2]
	\arrow[from=3-2, to=3-3]
	\arrow[shift left=3,from=3-3, to=2-3]
	\arrow[from=3-3, to=3-4]
	\arrow[shift left=3,from=4-2, to=3-2]
	\arrow[from=4-2, to=4-3]
	\arrow[shift left=3,from=4-3, to=3-3]
	\arrow[from=4-3, to=4-4]
	\arrow[shift left=3,from=4-4, to=3-4]
	\arrow[from=4-4, to=4-5]
	\arrow[shift left=3,from=5-2, to=4-2]
	\arrow[from=5-2, to=5-3]
	\arrow[shift left=3,from=5-3, to=4-3]
	\arrow[from=5-3, to=5-4]
	\arrow[shift left=3,from=5-4, to=4-4]
	\arrow[from=5-4, to=5-5]
	\arrow[shift left=3,from=5-5, to=4-5]
	\arrow[from=5-5, to=5-6]
	\arrow[shift left=3,curve={height=-6pt}, from=6-1, to=1-2]
	\arrow[shift left=3,curve={height=-6pt}, from=6-1, to=2-2]
	\arrow[shift left=3,curve={height=-6pt}, from=6-1, to=3-2]
	\arrow[shift left=3,curve={height=-6pt}, from=6-1, to=4-2]
	\arrow[shift left=3,curve={height=-6pt}, from=6-1, to=5-2]
	\arrow[curve={height=6pt}, from=6-1, to=5-2]
	\arrow[curve={height=6pt}, from=6-1, to=5-3]
	\arrow[curve={height=6pt}, from=6-1, to=5-4]
	\arrow[curve={height=6pt}, from=6-1, to=5-5]
	\arrow[curve={height=6pt}, from=6-1, to=5-6]
\end{tikzcd}\]
\caption{The quiver $Q_6=(\mathcal V_{\bullet}\cup\mathcal V_{\circ},\mathcal A)$\label{f:QuiverExample}}
\end{figure}

We now introduce labels for the vertices and edges of the quiver $Q_{n+1}$. These labels will be Laurent polynomials in the standard coordinates $(m_{ij})_{(i,j)\in\mathcal S_{\le n+1}}$ of $U_+$.
For any vertex $v_{ij}$ let 
\[SW(v_{ij})=\{(i-r,j-r)\mid 1\le r< \min(i,j)\}\]
be the indexing set for the $\bullet$-vertices directly southwest of $b$ 
(where this set may be empty). We associate to the vertex $v_{ij}$ a label $\ell_{ij}$ given by
\[
\begin{array}{lc}
\ell_{ij}=\begin{cases}m_{ij}\inv+\sum_{b\in SW(v_{ij})}m\inv_{b} & (i,j)\ne (0,0),\\
0 & (i,j)=(0,0).
\end{cases}
& \qquad \text{(vertex label)} \end{array}
\]
For an arrow $a$ with head $h(a)=v_{ij}$ we write $\ell_{h(a)}$ for $\ell_{ij}$ and similarly with head $h(a)$ replaced by the tail $t(a)$ of the arrow $a$. We then label the arrow $a$ itself by 
\[
\begin{array}{lr}
\ 
\kappa_a=\ell_{h(a)}-\ell_{t(a)} &\qquad \qquad \qquad \qquad \qquad\qquad\qquad\quad  \text{(arrow label).}
\end{array}
\]
Finally, for every $\bullet$-vertex $v_{ij}=\bullet_{(i,j)}$ we consider the sets of arrows
\[
\begin{array}{ccc}
\operatorname{in}(v_{ij})= \{a\in\mathcal A\mid h(a)=v_{ij}\}&\text{and}&
\operatorname{out}(v_{ij})=\{a\in\mathcal A\mid t(a)=v_{ij}\}.
\end{array}
\]

\begin{prop} \label{p:ToeplitzviaQ}
$\ToeplitzU_{n+1}(\RR_{>0})$ has the following description,
\begin{equation}\label{e:ToeplitzviaQfin}
\ToeplitzU_{n+1}(\RR_{>0})=\bigg{\{}\left.\phi^{\mathbf i_0}_{>0}((m_{ij})_{i,j})\,\right| \, m_{ij}\in \RR_{>0}, \prod_{a\in\operatorname{in}(v)}\kappa_a=
\prod_{a\in\operatorname{out}(v)}\kappa_a\,\text{ for all $v\in\mathcal V_\bullet$}\bigg{\}},
\end{equation}
in terms of the quiver $Q_{n+1}=(\mathcal V,\mathcal A)$ from \cref{d:Qn+1}, with the arrow labels $\kappa_a$ defined as  above.
Moreover, if $\phi^{\mathbf i_0}_{>0}((m_{ij})_{i,j})$ lies in $\ToeplitzU_{n+1}(\RR_{>0})$ then $\kappa_a\in\RR_{>0}$ for all arrows $a$.
\end{prop}
The proof of the above proposition will be given in \cref{a:ToeplitzviaQ}.
We can now prove the following result. 
\begin{prop}\label{p:TropToeplitzviaQ} Suppose $(\RR,\RR_{>0},\Val)$ is a ring with a valuative positive structure. Consider an element $u=\phi_{>0}^{\mathbf {i_0}}((m_{ij})_{i,j})$ of $U_+$ and let $M_{ij}=\Val(m_{ij})$.
If $u\in \ToeplitzU_{n+1}(\RR_{>0})$ then $(M_{ij})_{i,j}$ is a min-ideal filling. 
\end{prop}
\begin{remark} This proposition proves one containment direction of \cref{t:MinIdealIsToepl}.
\end{remark}
\begin{proof}  Consider a vertex $b\in \mathcal V_\bullet$ with its surrounding arrows. The quiver $Q_{n+1}$ locally around $b$ looks like one of the examples shown below. 
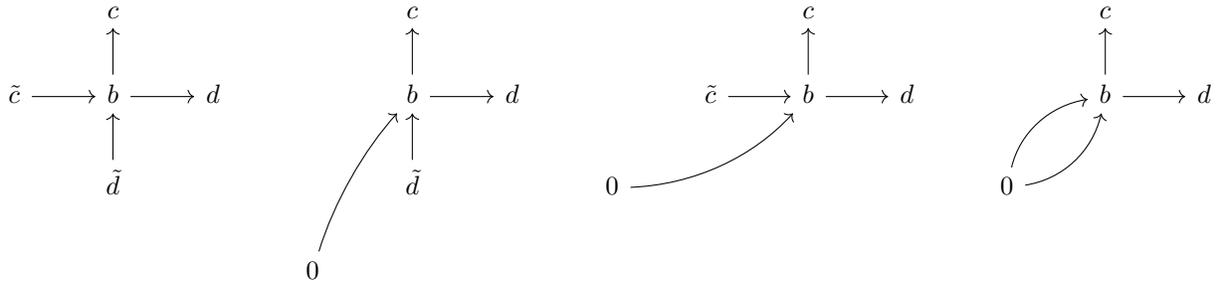
\[\begin{tikzcd}
	& c &&& c &&&& c &&& c \\
	\tilde c & b & d && b & d && \tilde c & b & d && b & d \\
	& \tilde d &&& \tilde d && 0 &&&& 0 \\
	&&& 0
	\arrow[from=2-1, to=2-2]
	\arrow[from=2-2, to=1-2]
	\arrow[from=2-2, to=2-3]
	\arrow[from=2-5, to=1-5]
	\arrow[from=2-5, to=2-6]
	\arrow[from=2-8, to=2-9]
	\arrow[from=2-9, to=1-9]
	\arrow[from=2-9, to=2-10]
	\arrow[from=2-12, to=1-12]
	\arrow[from=2-12, to=2-13]
	\arrow[from=3-2, to=2-2]
	\arrow[from=3-5, to=2-5]
	\arrow[curve={height=12pt}, from=3-7, to=2-9]
	\arrow[curve={height=12pt}, from=3-11, to=2-12]
	\arrow[curve={height=-12pt}, from=3-11, to=2-12]
	\arrow[curve={height=-6pt}, from=4-4, to=2-5]
\end{tikzcd}\] 
Consider $(m_{ij})_{i+j\le n+1}$ with $\phi^{\mathbf i_0}_{>0}((m_{ij}))\in \Toeplitz_{n+1}(\RR_{>0})$. Write $f_b((m_{ij})_{i,j}):= \prod_{a\in\operatorname{in}(b)}\kappa_a-
\prod_{a\in\operatorname{out}(b)}\kappa_a$, so that the Toeplitz condition implies the identity $f_{b}((m_{ij}))=0$, by \cref{p:ToeplitzviaQ}. If $b$ is as in the first configuration shown above, this identity takes the form 
\begin{multline*}
\left(m_c\inv-m_b\inv +\sum_{c'\in SW(c)} m_{c'}\inv  - \sum_{b'\in SW(b)} m_{b'}\inv\right) 
\left(m_d\inv - m_b\inv +\sum_{d'\in SW(d)} m_{d'}\inv - \sum_{b'\in SW(b)} m_{b'}\inv\right)\\
=\left( m_b\inv -m_{\tilde c}\inv + \sum_{b'\in SW(b)} m_{b'}\inv-\sum_{c'\in SW(\tilde c)} m_{c'}\inv\right) 
\left( m_b\inv -  m_{\tilde d}\inv +\sum_{b'\in SW(b)}  m_{b'}\inv -\sum_{d'\in SW(\tilde d)} m_{d'}\inv \right).
\end{multline*}
The identity for the other configurations is the same but with some sums absent. Namely, in the other configurations we have $SW(b)=\emptyset$, as well as either $SW(c)$ or $SW(d)=\emptyset$, or both. After cancellations and multiplying both sides of the equation by $m_b m_c m_d$ we get the simpler looking identity,
\begin{equation}\label{e:triangleatb}
\left(m_c +m_d\right)\left(1 + \sum_{b'\in SW(b)} \frac {m_b}{m_{b'}}\right) =
m_b \left( 1+\sum_{c'\in SW(c)}\frac {m_{c}}{m_{c'}}+\sum_{d'\in SW(d)} \frac{m_d}{m_{d'}}\right), 
\end{equation}
which is equivalent to the one above. 

We now prove the proposition using induction. The starting case is the first min-ideal condition, namely the identity $M_{1,1}=\min(M_{1,2},M_{2,1})$, which holds since \eqref{e:triangleatb} gives $m_{12}+m_{21}=m_{11}$. 
Now we fix some $k\ge 3$ and assume that  $M_{i,j}=\min(M_{i+1,j},M_{i,j+1})$ whenever $i+j< k$, so that the truncation $(M_{i,j})_{i+j\le k}$ is a min-ideal filling, and then proceed by induction on $k$.  We start by  showing the following claim. 
\vskip .2cm
\noindent{{\it Claim:}}
Consider any vertex $v=(r,s)$ with $r+s\le k+1$ and let $\tilde v=(r-1,s-1)$. Then $ M_{\tilde v}\le M_v$.   
\vskip .2cm
\noindent{{\it Proof of the Claim:}} If $r+s<k+1$ then the claim holds true by the min-ideal condition, which we can use thanks to the induction hypothesis. We therefore consider the case $r+s=k+1$. We may assume we have a subquiver of $Q_{n+1}$ as shown below, as any other case is similar and simpler. The dotted line indicates where the induction hypothesis applies. Namely the $M_c$ for vertices $c$ below the dotted line are already part of a min-ideal filling. 
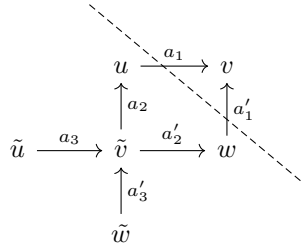
\[\begin{tikzcd}
	{} & {} \\
	& u & v \\
	{\tilde u} & {\tilde v} & w & {} \\
	& {\tilde w} && {}
	\arrow[""{name=0, anchor=center, inner sep=0}, draw=none, from=1-1, to=1-2]
	\arrow["{a_1}", from=2-2, to=2-3]
	\arrow["{a_3}", from=3-1, to=3-2]
	\arrow["{a_2}"', from=3-2, to=2-2]
	\arrow["{a_2'}", from=3-2, to=3-3]
	\arrow["{a_1'}"', from=3-3, to=2-3]
	\arrow["{a_3'}"', from=4-2, to=3-2]
	\arrow[""{name=1, anchor=center, inner sep=0}, draw=none, from=4-4, to=3-4]
	\arrow[shorten <=9pt, shorten >=9pt, dashed, no head, from=0, to=1]
\end{tikzcd}\]
From the construction of the vertex and arrow labellings 
we have that
\begin{equation}\label{e:kappasum}
\kappa_{a_1}+\kappa_{a_2}=\kappa_{a'_1}+\kappa_{a'_2}=m_v\inv.
\end{equation}
Let us set $K_a:=\Val(\kappa_a)$ and $\overline M_c:=\Val(m_c\inv)=-M_c$. Recall that by \cref{p:ToeplitzviaQ} we have $\kappa_a\in \RR_{>0}$. It  now follows from  \eqref{e:kappasum} that 
\begin{equation}\label{e:Kasatv}
 \min(K_{a_1},K_{a_2})=\min(K_{a'_1},K_{a'_2})=\overline M_v.
\end{equation}
Let us assume indirectly that $M_v<M_{\tilde v}$. Then  $\overline M_v>\overline M_{\tilde v}$ and therefore each of the $K_a$ involved in \eqref{e:Kasatv} is bigger than $\overline M_{\tilde v}$,
\begin{equation}\label{e:Ksvstildev}
 K_{a_1},K_{a_2},K_{a'_1},K_{a'_2}> \overline M_{\tilde v}.
\end{equation}
On the other hand, by the induction hypothesis, $M_{\tilde v}=\min(M_u,M_w)$, wherefore $\overline M_{\tilde v}=\max(\overline M_u, \overline M_w)$. So we can extend \eqref{e:Ksvstildev} to 
\begin{equation}\label{e:Ksvstildevetc}
 K_{a_1},K_{a_2},K_{a'_1},K_{a'_2}> \overline M_{\tilde v}=\max(\overline M_u,\overline M_w)\ge \overline M_u,\overline M_w.
\end{equation}
Analogously to \eqref{e:kappasum}, but at $u$ and $w$ we have $\kappa_{a_2}+\kappa_{a_3}=m_u\inv$ and $\kappa_{a_2'}+\kappa_{a_3'}=m_{w}\inv$, and therefore
\begin{equation}
\min(K_{a_2},K_{a_3})=\overline M_u\qquad \text{and}\qquad \min(K_{a'_2},K_{a'_3})=\overline M_{w}.
\end{equation}
Now since $K_{a_2}>\overline M_u$ and $K_{a'_2}>\overline M_w$ by \eqref{e:Ksvstildevetc}, we must have $K_{a_3}=\overline M_u$ and $K_{a'_3}=\overline M_w$. Now we can compare the sum of the valuations of the outgoing arrows of $\tilde v$ with those of the incoming  arrows, and we find that 
\begin{equation}
K_{a_2}+K_{a'_2}>2\overline M_{\tilde v}\ge \overline M_u+\overline M_w =K_{a_3}+K_{a'_3}
\end{equation}
where the two inequalities follow from \eqref{e:Ksvstildevetc}. This contradicts the  relation at the vertex $\tilde v$ from \eqref{e:ToeplitzviaQfin}, which reads $\kappa_{a_2}\kappa_{a'_2}=\kappa_{a_3}\kappa_{a'_3}$ and implies $K_{a_2}+K_{a'_2}=K_{a_3}+K_{a'_3}$. Thus the claim is proved. 

\vskip .2cm
We can now finish the inductive proof of the proposition. We pick a vertex $b=(i,j)$ with $i+j=k$ and we want to show that $M_b=\min(M_c,M_d)$ for $c=(i+1,j)$ and $d=(i,j+1)$. 

By repeated application of the claim just proved we see that 
\[
 M_b\ge M_{b'},\quad M_c\ge  M_{c'},\quad M_d\ge M_{d'}\quad \text{ whenever\quad $b'\in SW(b),\ c'\in SW(c)$ and $d'\in SW(d)$.}
\]
It follows that the factors found in  \eqref{e:triangleatb}, namely 
\[
\left(1 + \sum_{b'\in SW(b)} \frac {m_b}{m_{b'}}\right)\qquad\text{ and }\qquad \left( 1+\sum_{c'\in SW(c)}\frac {m_{c}}{m_{c'}}+\sum_{d'\in SW(d)} \frac{m_d}{m_{d'}}\right),\]
both have valuation $0$. 
Now applying $\Val$ to the two sides of \eqref{e:triangleatb} we are left with the identity $
\Val(m_c+m_d)=\Val(m_b)$, which is precisely the min-ideal condition $\min(M_c,M_d)=M_b$. 
\end{proof}

We now complete the theorem describing $\ToeplitzU_{n+1}(\R_{\min})$ by making use of \cref{t:TropFinite}, that we have proved already, together with an argument from \cite{Judd:Flag}. 

\begin{proof}[Proof of \cref{t:MinIdealIsToepl}]  We have the following commutative diagram, 
\[\begin{tikzcd}
	{\ToeplitzU_{n+1}({\Rmin})} && {\mathfrak h^*_{PSL}(\R)} \\
	\Imin_{n+1}
	\arrow["\mathcal L|_{\Toeplitz}",from=1-1, to=1-3]
	\arrow[hook, from=1-1, to=2-1]
	\arrow["{\mathcal L|_{\Imin}}"', from=2-1, to=1-3]
\end{tikzcd}\]
where the horizontal map $\mathcal L|_{\Toeplitz}$ is a bijection by \cref{t:TropFinite}, and the vertical map is the  inclusion that we have by \cref{p:ToeplitzviaQ}. It now suffices to prove that the piecewise linear map $\mathcal L|_{\Imin}:\Imin_{n+1}\to\mathfrak h^*_{PSL}(\R)$ is injective. This follows by the combinatorial proof given in \cite[Proof of Proposition~5.6]{Judd:Flag}, which applies without change to our  setting. 
\end{proof}

\section{Infinite min-ideal fillings and a tropical Edrei theorem}\label{s:TropEdrei}
By an infinite min-ideal filling we will simply mean an element $(M_{ij})_{i,j}\in\R^{\N\times\N}$ satisfying the condition
\[
M_{ij}=\min(M_{i+1,j},M_{i,j+1})
\]
for all $i,j\in\N$. Graphically, we may represent it as an infinite pattern of numbers in the positive quadrant $\R_{>0}^2$ with $M_{ij}$ placed at the lattice point $(j,i)\in \N^2$, obeying the condition shown in \cref{d:minideal}.
\begin{center}
 \begin{tikzpicture}[scale=0.6]
\node[fill=white] at    (0.5,2.6) {$\tiny{M_{31}}$}  ;
\node[fill=white] at    (0.5,1.6) {$M_{21}$}  ;
\node[fill=white] at    (0.5,0.6) {$M_{11}$}  ;
\node[fill=white] at    (1.8,1.6) {$M_{22}$}  ;\node[fill=white] at    (1.8,0.6) {$M_{12}$}  ;
\node[fill=white] at    (3,0.6) {$M_{13}$}  ;
\draw[black] (-.2,0.1) -- (-.2,3.55) --(-.2,3.1)--(1.35,3.1)-- (1.15,3.1) --(1.15,3.3) -- (1.15,0.1) -- (4,0.1)--(3.6,0.1) -- (3.6,1.3)--(3.6,1.1)--(3.8,1.1)--(3.6,1.1)--(-.2,1.1)--(-.2,0.1)--(2.4,0.1)--(2.4,2.3)--(2.4,2.1)--(2.6,2.1)--(2.4,2.1)--(-.2,2.1);
 \end{tikzpicture}
 \end{center} 
For the infinite case we may also consider various alternative versions in terms of the limiting behaviour.  We now give some more formal definitions, after which we will state and prove directly a parametrisation theorem for the infinite min-deal fillings. 

\begin{defn}[Infinite min-ideal fillings]
Let 
\[
\Imin_{\infty}(\Rmininf):=\{(M_{ij})_{i,j\in\Z_{>0}}\mid M_{ij}\in \Rmininf,\, M_{ij}=\min (M_{i+1,j},M_{i,j+1})\ \text{ for all $i,j\in\Z_{>0}$}\}
\]
and let $\Imin_{\infty}(\Rmin)$ be the subset of  
$\Imin_{\infty}(\Rmininf)$ where all coordinates $M_{ij}$ lie in $\Rmin$. 
We refer to the elements of $\Imin_{\infty}(\Rmin)$ as \textit{infinite min-ideal fillings}. The elements of $\Imin_{\infty}(\Rmininf)$ will be called \textit{generalised infinite min-ideal fillings}. The monoid structure given by the fomula
\[
(M_{ij})_{i,j=1}^\infty\wedge(M'_{ij})_{i,j=1}^\infty=(\min(M_{ij},M'_{ij}))_{i=1}^\infty.
\]
turns both $\Imin_{\infty}(\Rmininf)$ and  $\Imin_{\infty}(\Rmin)$ into meet semi-lattices.
\end{defn}

\begin{remark}[Projective limit description] We may describe the set $\Imin_{\infty}(\Rmin)$ of infinite min-ideal fillings as a projective limit. Consider the sequence of maps 
\[\begin{tikzcd}
	{\Imin_2} & {\Imin_3} & {} & {} & {\Imin_n} & {\Imin_{n+1}} & {}
	\arrow["\pi^{3}_2"',from=1-2, to=1-1]
	\arrow["\pi^{4}_3"',from=1-3, to=1-2]
	\arrow[dotted, no head, from=1-4, to=1-3]
	\arrow["{\pi^{\, n}_{n-1}}"',from=1-5, to=1-4]
	\arrow["{\pi^{n+1}_n}"', from=1-6, to=1-5]
	\arrow[from=1-7, to=1-6]
\end{tikzcd}\]
where $\pi^{k+1}_k$ sends $(M_{ij})_{(i,j)\in \mathcal S_{k+1}}$ to its truncation $(M_{ij})_{(i,j)\in \mathcal S_{k}}$. 
Then we have 
\begin{equation}\label{e:IMaslimit}\Imin_{\infty}(\Rmin)=\underset{\longleftarrow}{\lim}\ \Imin_{n+1}.
\end{equation}
This description endows $\Imin_{\infty}(\Rmin)$ with a natural map to $\Imin_{n+1}$ for every $n$, 
\begin{equation}
\begin{array}{rccc}
\pi_{n+1}:&\Imin_{\infty}(\Rmin)&\to&\Imin_{n+1}\\
&(M_{ij})_{i,j=1}^\infty&\mapsto &(M_{ij})_{i+j\le n+1}.
\end{array}
\end{equation}
Observe also that all of the maps above are compatible with the meet semilattice structures. 
\end{remark}
\begin{remark}\label{r:NonSurj}
It is important to note that the maps $\pi^{n+1}_n:\Imin_{n+1}\to\Imin_{n}$ are not generally surjective. For example, the min-ideal filling in $\Imin_4$ given by
\begin{center} 
 \begin{tikzpicture}[scale=0.6]
\node[fill=white] at    (0.6,2.6) {$2$}  ;
\node[fill=white] at    (0.6,1.6) {$1$}  ;
\node[fill=white] at    (0.6,0.6) {$1$}  ;
\node[fill=white] at    (1.7,1.6) {$1$}  ;\node[fill=white] at    (1.7,0.6) {$1$}  ;
\node[fill=white] at    (2.7,0.6) {$2$}  ;
\draw[black] (0,0.1) -- (0,3.1) -- (1.15,3.1) -- (1.15,0.1) -- (3.2,0.1) -- (3.2,1.1)--(0,1.1)--(0,0.1)--(2.2,0.1)--(2.2,2.1)--(0,2.1);
 \end{tikzpicture}\quad
 \end{center}
does not lie in the image of $\pi^5_4$. Otherwise the extended ideal filling $(M_{ij})_{i+j\le 5}$ would have to have $\min(M_{23},M_{32})=M_{2,2}=1$, which would imply that $M_{31}, M_{13}\le 1$. But we have $M_{31}=M_{13}=2$.  Similarly, the maps $\pi_{n+1}:\Imin_{\infty}(\Rmin)\to\Imin_{n+1}$ are not surjective once $n\ge 3$. 
\end{remark}

\begin{defn}[Variations on infinite min-ideal fillings] \label{d:MinIdealVar}

We call an infinite min-ideal  filling  \textit{divergent} if the supremum of its entries is $\infty$.  We write \[\Imin^{div}_{\infty}(\Rmin)=\{(M_{ij})\in\Imin_{\infty}(\Rmin)\mid  \sup(\{M_{ij}\mid i,j\in\N\})=\infty\}.\]
We call a min-ideal filling \textit{asymptotically real} if rows and columns individually have suprema in $\R$. Set
\[
\Imin^{\R}_{\infty}(\Rmin)=\{(M_{ij})\in\Imin_{\infty}(\Rmin)\mid \text{$\sup\{M_{ik}\mid k\in\N\},\sup\{M_{kj}\mid k\in\N\}\in\R$ for every $i,j\in\N$}\}.
\]
We call an infinite ideal filling \textit{stable} if the rows and columns stabilise, that is, if 
\begin{equation}\label{e:stability}
\begin{array}{ccl}
 \lim_{i\to\infty} M_{ij}&=&\max(\{M_{ij}\mid i\in\Z_{>0}\}),\text{ for all $j\in\Z_{>0}$,}\\   \lim_{j\to\infty}M_{ij}&=&\max(\{M_{ij}\mid j\in\Z_{>0}\}),\text{ for all $i\in\Z_{>0}$.}
\end{array}
\end{equation}
We denote the set of stable min-ideal fillings by $\Imin^s_\infty(\Rmin)$.

The  divergent, the asymptotically real, and the stable infinite min-ideal fillings are all sub-meet-semilattices of $\Imin_\infty(\Rmin)$. 
\end{defn}
We now state a parametrisation theorem for the  infinite ideal fillings. 
We require the following definitions.

\begin{defn}[Interlacing and weakly interlacing sequences]\label{d:interlacing}
Let
\[
\overline{\operatorname{MonSeq}}:=\{\bigA=(A_i)_{i=1}^{\infty}\mid A_i\in\R_{\min,\infty},\ A_1\le A_2\le\dotsc \},
\]
denote the set of weakly increasing sequences in $\Rmininf$.  
Two elements $\bigA=(A_i)_{i\in \Z_{\ge 0}}$ and $\bigB=(B_j)_{j\in \Z_{\ge 0}}$ of $\overline{\operatorname{MonSeq}}$ are called \textit{weakly interlacing} if 
\[
\sup\{A_i\mid i\in\mathbb Z_{>0}\}=
\sup\{B_i\mid i\in\mathbb Z_{>0}\},
\]
where the supremum is allowed to be $\infty$. The two sequences  $\bigA,\bigB$ will be called \textit{interlacing} if for every $A_i$ from $\bigA$ there is a $B_j$ from $\bigB$ such that $B_j\ge A_i$, and vice versa, for every $B_i$ there is an $A_j$ such that $A_j\ge B_i$. Note that interlacing implies weakly interlacing. 
The weakly interlacing sequences that are not interlacing are precisely those for which the joint supremum is a maximum for one of the sequences but not for the other.  
\end{defn}

\begin{defn}[Tropical parameter spaces]\label{d:Omegatrop} 
We make the following key definitions. 
\begin{eqnarray*}
\Omega(\Rmin)\ &:=&\{(\bigA,\bigB)\mid \bigA,\bigB \in\overline{\operatorname{MonSeq}},\, \min(A_i,B_i)<\infty \text{ for all $i\in\Z_{>0}$}\},\\
\Omega_\star(\Rmin)&:=&\{(\bigA,\bigB)\in \Omega(\Rmin)\mid \text{with $\bigA,\bigB$ weakly interlacing}\}.
\end{eqnarray*}
We also have the following alternative versions (subsets) of $\Omega_\star(\Rmin)$,
\begin{eqnarray*}
\Omega_{div}(\Rmin)&:=&\{(\bigA,\bigB)\in \Omega_\star(\Rmin)\mid \text{with $\sup(\{A_i\})=\sup(\{B_i\})=\infty$}\},\\
\Omega_\star^{il}(\Rmin)&:=&\{(\bigA,\bigB)\in \Omega_\star(\Rmin)\mid \text{with $\bigA,\bigB$ interlacing}\}.
\end{eqnarray*}
We may replace $\Rmin$ by $\Rmininf$ above, which will refer to the analogous set with the condition that $\min(A_i,B_i)<\infty$ removed.  So, for example, we write $\Omega(\Rmininf)$ for $\overline{\operatorname{MonSeq}}\times \overline{\operatorname{MonSeq}}$, and we set $\Omega_\star(\Rmininf)$ to be the set of weakly interlacing sequences in $\Omega(\Rmininf)$, 
\[
\Omega_\star(\Rmininf):=\{(\bigA,\bigB)\mid  \text{ $\bigA,\bigB \in\overline{\operatorname{MonSeq}}$, weakly interlacing}\}.
\]
We may think of $\Omega_\star(\Rmin)$ as the subset of $\Omega_\star(\Rmininf)$ with the additional property that at least one of the sequences lies entirely in $\Rmin$. 

We finally consider one further subset of $\Omega_\star(\Rmin)$, 
\[
\Omega^\R_\star(\Rmin):=\{(\bigA,\bigB)\in \Omega_\star(\Rmin)\mid \text{with $A_i,B_j\in\R$ for all $i,j\in\N$}\}.
\]
which is made up  of real weakly interlacing sequences $\bigA,\bigB$. 
\end{defn}
\begin{remark}\label{r:interlacingvsR} Note that
\[\Omega^{il}_\star(\Rmin) \subset \Omega^\R_\star(\Rmin)\subset \Omega_\star(\Rmin)\quad \text{ and }\quad  \Omega_\star(\Rmin)=\Omega^{\R}_\star(\Rmin)\cup\Omega^{div}_{\star}(\Rmin),\]
with a nontrivial intersection $\Omega^{\R}_\star(\Rmin)\cap\Omega^{div}_{\star}(\Rmin)=\Omega^{il}_\star(\Rmin)\cap\Omega^{div}_{\star}(\Rmin)$ made up of real, divergent parameter sequence pairs.  
\end{remark}
We now state the following parametrization theorem covering infinite min-ideal fillings of various types. We will go on to show how this result relates to the Edrei theorem in \cref{s:KdelTropToep} and \cref{s:restrictedT}.  
\begin{theorem}\label{t:TropEdrei}
If $(\bigA,\bigB)\in\Omega(\Rmininf)$ then $M_{ij}:=\min(A_i,B_j)$ determines a  generalised infinite min-ideal filling $(M_{ij})_{i,j\in\mathbb Z_{>0}}$. Moreover the map
\[
\begin{array}{rccl}
\bar{\mathbb E}:&\Omega_\star(\Rmininf)&\longrightarrow&\Imin_{\infty}(\Rmininf)
\\
&{(\bigA,\bigB)}&\mapsto& (\min(A_i,B_j)_{ij})_{i,j\in\mathbb Z_{>0}}
\end{array}
\]
is a  bijection, as is the restriction $\bar{\mathbb E}^s:\Omega^{il}_\star(\Rmininf)\overset\sim\longrightarrow \Imin^s_{\infty}(\Rmininf)$ to interlacing sequences. 
Furthermore we obtain by restriction
\begin{equation*}
\begin{array}{llccl}
\mathbb E:&\Omega_\star(\Rmin)&\longrightarrow &\Imin_{\infty}(\Rmin),\\
\mathbb E^s:&\Omega^{il}_\star(\Rmin)&\longrightarrow &\Imin^s_{\infty}(\R_{\min}),\\
\mathbb E^\R:&\Omega^{\R}_\star(\Rmin) &\longrightarrow &\Imin^{\R}_{\infty}(\R_{\min}),\\
\mathbb E^{div}:&
\Omega^{div}_\star(\Rmin)&\longrightarrow &\Imin^{div}_{\infty}(\R_{\min}),
\end{array}
\end{equation*} 
and these maps are again well-defined bijections. 
\end{theorem}

We begin by proving some key properties of (generalised) infinite min-ideal fillings. 

\begin{lemma}\label{l:InfiniteProperty} Suppose $(M_{ij})_{i,j}\in \Imin_{\infty}(\Rmininf)$ is a generalised infinite min-ideal filling. Then we have 
\begin{equation}\label{e:InfiniteProperty}
M_{ij}=\min(M_{i+k,j},M_{i,j+k})
\end{equation}
for any $k\in\Z_{\ge 0}$ .
\end{lemma}

\begin{remark}
It is necessary in this lemma that the min-ideal filling be  infinite. For example, Remark~\ref{r:NonSurj} shows a finite min-ideal filling for which $M_{1,1}=1$ but $M_{3,1}=M_{1,3}=2$. So for this finite min-ideal filling, the above formula with $i=j=1$ and $k=2$ fails. 
\end{remark}

\begin{proof}[Proof of Lemma~\ref{l:InfiniteProperty}]
The three entries of the min-ideal filling $(M_{ij})$ that are of interest are configured
\begin{equation}
\begin{matrix}
M_{i+k,j} &&\\
\vdots &\ddots &\\
M_{i,j} & \dots        &M_{i,j+k}.
\end{matrix}
\end{equation}
If $k=1$ then the equality \eqref{e:InfiniteProperty} becomes $M_{ij}= \min(M_{i+1,j},M_{i,j+1})$, which is automatically true as it is the defining property of a min-ideal filling. Next, it follows recursively that $M_{ij}$ is the minimum of the entries along the diagonal of the above diagram,
\begin{equation}\label{e:DiagonalMax}
M_{ij}=\min(M_{i+k,j},M_{i+k-1,j+1},\dotsc, M_{i+1,j+k-1}, M_{i,j+k}).
\end{equation}
In particular, $M_{ij}\le \min(M_{i+k,j},M_{i,j+k})$. This is true both in the finite and the infinite cases.

Let us now explicitly consider an infinite generalised min-ideal filling, and suppose indirectly that \eqref{e:InfiniteProperty} is false. Then we must have $M_{ij}< \min(M_{i,j+k},M_{i+k,j})$. It follows from \eqref{e:DiagonalMax} that $M_{ij}=M_{i+k-\ell,j+\ell}$ for some $0<\ell<k$.  Moreover we can choose $\ell$ minimal. Then there exists an index $\ell'$ with $\ell\le \ell'<k$ such that 
\[
M_{i+k-\ell+1,j+\ell-1}>M_{i+k-\ell,j+\ell}=\dotsc=M_{i+k-\ell',j+\ell'}<M_{i+k-\ell'-1,j+\ell'+1}.
\]
Thus a portion of the diagonal of the min-ideal filling takes the form $a,c,c,\dotsc,c,b$ with $a,b>c$. 

We now consider the entries of the infinite min-ideal filling that are above  and to the right of this portion of the diagonal. We denote the entries to the right of $a$ by $a_1,a_2,\dotsc $, and note that by the min-ideal condition the sequence $(a_i)_i$ is weakly increasing. Therefore $a_i>c$ for all $i$. It follows recursively using the min-ideal property that the entry below $a_i$ in the filling is equal to $c$ whenever $i\ge 1$.

 Similarly going up the column from $b$ we have a weakly increasing sequence of entries that we denote by $b_1\le b_2\le b_3\dotsc $, and the entry to the left of $b_i$ is equal to $c$ for all $i\ge 1$. Thus a part of the min-ideal filling looks as follows. 
\[
\begin{matrix}
 a &a_1&a_2&\dotsc& a_r&&\\
 *   &c  &c   &\dotsc& c  &c&&\\
  *  &*&c&&&c&b_r\\
 \vdots&&\ddots&\ddots&&\vdots &\vdots\\
 \vdots&  &&\ddots&c &c&b_2\\
 * &\ddots&&&*&c&b_1\\
 * &*&\dots&\dots&*&*&b\\
  \end{matrix}
\]
But then the entries $a_{r+1}, b_{r+1}$ just to the right of $a_r$ and above $b_r$, respectively, are both strictly greater than $c$.  Therefore $\min(a_{r+1},b_{r+1})>c$, in contradiction to the min-ideal condition. 
\end{proof}

\begin{remark}\label{r:convexity} In the proof of Lemma~\ref{l:InfiniteProperty} we have shown in particular that any diagonal segment in a min-ideal filling, that is a segment of the form  $M_{i_0+k,j_0}, M_{i_0+k-1,j_0+1},\dotsc, M_{i_0,j_0+k}$, satisfies 
\[
M_{i_0+\ell,j_0+k-\ell}\ge\min (M_{i_0+k,j_0}, M_{i_0,j_0+k})
\] 
for all $\ell\in[0,k]$.
\end{remark}
\begin{lemma}\label{l:OneLegConstant}
Suppose $M=(M_{ij})_{i,j\in\Z_{>0}}$ is an element of $ \Imin_{\infty}(\Rmininf)$. For any $i_0,j_0\in\Z_{> 0}$ we have either
$M_{i_0,j_0+\ell}=M_{i_0,j_0}$ for all $\ell\ge 0$ or $M_{i_0+\ell,j_0}=M_{i_0,j_0}$ for all $\ell\ge 0$. 
\end{lemma}
\begin{proof}
Recall that $M_{ij}$ must be weakly increasing in both $i$ and $j$. Suppose indirectly that the statement of the lemma is false. Then we can choose $\ell$ large enough so that both $M_{i_0,j_0+\ell}>M_{i_0,j_0}$  and $M_{i_0+\ell,j_0}>M_{i_0,j_0}$. On the other hand by Lemma~\ref{l:InfiniteProperty} 
\[
M_{i_0,j_0}=\min(M_{i_0,j_0+\ell},M_{i_0+\ell,j_0})
\]
and we have arrived at a contradiction.
\end{proof}

Thanks to Lemma~\ref{l:OneLegConstant} we can strengthen the previous Lemma~\ref{l:InfiniteProperty} as follows.

\begin{lemma}\label{l:InfinitePropertyStrong} If $(M_{ij})_{i,j}\in\Imin_{\infty}(\Rmininf)$ is a generalised infinite min-ideal filling, then   we have 
\begin{equation}\label{e:genZmin}
M_{i_0,j_0}=\min(M_{i_0,j_0+\ell},M_{i_0+k,j_0})\quad \text{for all $\ell,k\in\Z_{\ge 0}$,}
\end{equation} 
for any $i_0,j_0\in\Z_{>0}$.
\end{lemma}
\begin{defn}\label{d:suprema}
Let $(M_{ij})_{i,j}\in\Imin_{\infty}(\Rmininf)$ be a generalised infinite min-ideal filling. We define $M_{i,\infty},M_{\infty,j}\in\R\cup\{\infty\}$ by \[
M_{i,\infty}=\sup\{M_{i,k}\mid k\in\Z_{>0}\} \text{ and } M_{\infty, j}=\sup\{M_{k,j}\mid k\in\Z_{>0}\},
\] respectively. 
\end{defn}
We prove Lemma~\ref{l:InfinitePropertyStrong} together with the following corollary.

\begin{cor}\label{c:InfinitePropertyStrongII} If $M=(M_{ij})_{i,j}\in\Imin_{\infty}(\Rmininf)$ is a generalised infinite min-ideal filling, then for any $i_0,j_0\in\Z_{>0}$ we have 
\begin{equation}\label{e:genZmininf}
M_{i_0,j_0}=\min(M_{i_0,\infty},M_{\infty,j_0}).
\end{equation}
Additionally, if  $M\in\Imin_\infty(\Rmin)$ then $\min(M_{i_0,\infty},M_{\infty,j_0})\in \R$. 
\end{cor}
\begin{proof} [Proof of Lemma~\ref{l:InfinitePropertyStrong} and Corollary~\ref{c:InfinitePropertyStrongII}] Suppose $M$ is a generalised infinite min-ideal filling. Since the  $M_{ij}$ are weakly increasing in $i$ and in $j$ we have $M_{i_0,\infty}\ge M_{i_0,j_0+\ell}\ge M_{i_0,j_0}$ and $M_{\infty,j_0}\ge M_{i_0+k,j_0}\ge M_{i_0,j_0}$.
Now by Lemma~\ref{l:OneLegConstant} there are two cases. In the first case, $M_{i_0,j_0+\ell}=M_{i_0,j_0}$ for all $\ell\ge 0$, and so we have
$M_{i_0,\infty}=M_{i_0,j_0+\ell}=M_{i_0,j_0}$. Therefore we have the following two identities,
\[
\begin{array}{l}
\min(M_{i_0,\infty},M_{\infty,j_0})=M_{i_0,\infty}=M_{i_0,j_0},\\
\min(M_{i_0,j_0+\ell},M_{i_0+k,j_0})=M_{i_0,j_0+\ell}=M_{i_0,j_0}.
\end{array}
\]
The second case is the same but with the role of $i$ and $j$ swapped. In this case we have  $
M_{\infty,j_0}=M_{i_0+k,j_0}=M_{i_0,j_0}$
and
\[
\begin{array}{l}
\min(M_{i_0,\infty},M_{\infty,j_0})=M_{\infty,j_0}=M_{i_0,j_0}, \\
\min(M_{i_0,j_0+\ell},M_{i_0+k,j_0})=M_{i_0+\ell,j_0}=M_{i_0,j_0}.
\end{array}
\]
Either way, \eqref{e:genZmin} and \eqref{e:genZmininf} hold. 
\end{proof}

The next lemma explains the relevance of the `weakly interlacing' property used in \cref{d:Omegatrop}. 

\begin{lemma} \label{l:injectiveInfMin} Consider the equivalence relation $\sim$ on $\Omega(\Rmininf)=\overline{\MonSeq}\times\overline{\MonSeq}$  defined by 
\[
(\bigA,\bigB)\sim (\bigA',\bigB') \iff \min(A_i,B_j)=\min(A'_i,B'_j)\text{ for all $i,j\in\mathbb Z_{>0}$.}
\]
We have a bijective map $
\Omega_\star(\Rmininf)\to\Omega(\Rmininf)/\sim$ defined by sending $(\bigA,\bigB)$ to its equivalence class $[\bigA,\bigB]$. 
This map restricts to a bijection $\Omega_\star(\Rmin)\to \Omega(\Rmin)/\sim$.
\end{lemma}

\begin{proof}
We first show that any element $[\bigA',\bigB']$ of $\Omega(\Rmininf)$ has a representative $(\bigA,\bigB)$ that is weakly interlacing. Namely, we claim that we can construct $(\bigA,\bigB)$ out of the values $\min(A'_i,B'_j)$ by setting
\[
\begin{array}{ccl}
A_i&=&\sup(\{\min(A'_i,B'_k)\mid {k\in\mathbb Z_{>0}}\}),\\
B_j&=&\sup(\{\min(A'_k,B'_j)\mid {k\in\mathbb Z_{>0}}\}).
\end{array}
\]
It follows from the fact that $\bigA'$ and $\bigB'$ are weakly increasing, that $\bigA=(A_i)_{i\in\mathbb Z_{>0}}$ and $\bigB=(B_j)_{j\in\mathbb Z_{>0}}$ defined in this way are in $\overline{\MonSeq}$.

Next, since $\min(A'_i,B'_k)\le A'_i$ we have that $A_i\le A'_i$. Similarly $B_j\le B'_j$. Therefore $\min(A_i,B_j)\le\min(A'_i,B'_j)$.  On the other hand, $A_{i}\ge \min(A'_i,B'_j)$ by choosing $k=j$ inside the supremum, and similarly $B_{j}\ge \min(A'_i,B'_j)$ by choosing $k=i$ in the supremum. Therefore we must have that $\min(A_i,B_j)=\min(A'_i,B'_j)$. It follows that $(\bigA,\bigB)\sim(\bigA',\bigB')$ and $(\bigA,\bigB)$ represents the same element in $\Omega(\Rmininf)$. 

Next notice that from the definition of $A_i$ and $B_j$ we have
\[
\sup(\{A_i\mid i\in\Z_{>0}\})=\sup(\{\min(A'_i,B'_j)\mid {i,j\in\mathbb Z_{>0}}\})=\sup(\{B_j\mid j\in\Z_{>0}\}).
\]
Therefore $\bigA$ and $\bigB$ are weakly interlacing. 

This way we have constructed a right inverse to the map $\Omega_\star(\Rmininf)\to \Omega(\Rmininf)/\sim$ and shown that the map is surjective. To see that this right-inverse is also a left inverse we just need to check that 
\[
\begin{array}{ccl}
A_i&=&\sup(\{\min(A_i,B_k)\mid {k\in\mathbb Z_{>0}}\})\\
B_j&=&\sup(\{\min(A_k,B_j)\mid {k\in\mathbb Z_{>0}}\})
\end{array}
\]
whenever $(\bigA,\bigB)\in\Omega_\star(\Rmininf)$. But this is clear. If $\bigA$ and $\bigB$ have the same supremum, then $\sup(\{B_k\mid k\in\mathbb Z_{>0}\})\ge A_i$. Therefore 
$\sup(\{\min(A_i,B_k)\mid {k\in\mathbb Z_{>0}}\})=A_i$. Similarly $\sup(\{\min(A_k,B_j)\mid {k\in\mathbb Z_{>0}}\})=B_j$. 

It is straightforward that the restriction $\Omega_\star(\Rmin)\to \Omega(\Rmin)/\sim$ is well-defined and also a bijection.
\end{proof}

We can now prove the `tropical Edrei theorem'. 
\begin{proof}[Proof of Theorem~\ref{t:TropEdrei}]
 First let us check that $\bar{\mathbb E}(\bigA,\bigB)$ is an infinite min-ideal filling, so that the map $\mathbb E$ is well-defined. 
Let $\bigA_{\sup}$ be the constant sequence given by $\sup\{A_i\mid i\in\Z_{>0}\}$, and let ${\bigB}_{\sup}$ be the constant sequence given by $\sup\{ B_i\mid i\in\Z_{>0}\}$. Then $({\bigA},{\bigA}_{\sup})$ and $({\bigB}_{\sup},{\bigB})$ lie in $\Omega_\star(\Rmininf)$. 

Now
 $\bar{\mathbb E}({\bigA},{\bigA}_{\sup})=( A_{i})_{i,j}$, and this is clearly a (constant along rows) infinite min-ideal filling. Similarly,  $\bar{\mathbb E}({\bigB}_{\sup},{\bigB})=( B_j)_{i,j}$ is a (constant along columns) infinite min-ideal filling. But now in the meet-semilattice structure of $\Imin_{\infty}(\Rmininf)$ we have
\[
\bar{\mathbb E}({\bigA},{\bigA}_{\sup})\ \wedge\  \bar{\mathbb E}({\bigB}_{\sup},{\bigB})=(\min( A_i, B_j))_{i,j},
\]
which is precisely $\bar{\mathbb E}({\bigA},{\bigB})$. Therefore it follows that $\bar{\mathbb E}({\bigA},{\bigB})\in\Imin_{\infty}(\Rmininf)$. 

We have thus seen that  $\bar{\mathbb E}$ is a well-defined map $\Omega_\star(\Rmininf)\to \Imin_{\infty}(\Rmininf)$. The image $\bar{\mathbb E}({\bigA},{\bigB})$ lies in $\Imin_{\infty}(\Rmin)$  if and only if each entry $\min( A_i, B_j)$ lies in $\R$, that is, if and only if $({\bigA},{\bigB})\in \Omega_\star(\Rmin)$. Therefore the restriction $\mathbb E$ of $\bar{\mathbb E}$ to $\Omega_\star(\Rmin)$ gives a well-defined map 
$\Omega_\star(\Rmin)\to \Imin_{\infty}(\Rmin)$.

We now show bijectivity. It follows immediately from Lemma~\ref{l:injectiveInfMin} that  $\bar{\mathbb E}$ (and hence  also $\mathbb E$) is injective. For surjectivity, suppose $(M_{ij})_{i,j}\in \Imin_{\infty}(\Rmininf)$  is given. Consider the sequences $(M_{\infty,j})_{j\in\Z_{>0}}$ and $(M_{i,\infty})_{i\in\Z_{>0}}$ in $\R_{\min,\infty}$ constructed as in Definition~\ref{d:suprema}. These are both weakly increasing sequences. Their suprema equal to the overall supremum of the set $\{M_{ij}\mid {i,j}\in\Z_{>0}\}$, and thus agree (we may denote this common supremum by $M_{\infty,\infty}$). We now set 
\[
\bigA=(M_{i,\infty})_{i\in\Z_{>0}},\qquad \bigB=(M_{\infty,j})_{j\in\Z_{>0}},
\]
and observe that $\bigA$ and $\bigB$ lie in $\overline{\MonSeq}$ and have the same supremum. Therefore $(\bigA,\bigB)\in\Omega_{\star}(\Rmin)$. Now by Corollary~\ref{c:InfinitePropertyStrongII} we have that  $\bar{\mathbb E}(\bigA,\bigB)=(M_{ij})_{i,j}$. If our original  $(M_{ij})_{i,j}$ is an infinite ideal filling, $ (M_{ij})_{i,j}\in \Imin_{\infty}(\Rmin)$ then the the pair $(\bigA,\bigB)$ constructed above lies in $\Omega_{\star}(\Rmin)$ and $\mathbb E(\bigA,\bigB)=(M_{ij})_{i,j}$.

We now consider the restriction of $\bar{\mathbb E}$ to  $\Omega^{il}_\star(\Rmininf)$. If $(\bigA,\bigB)$ is interlacing then $\min(A_i,B_j)$ will equal to $A_i$ for $j$ large enough, since eventually $B_j\ge A_i$, and similarly $\min(A_i,B_j)=B_j$ for $i$ large enough. Thus the generalised min-ideal filling with $M_{ij}=\min(A_i,B_j)$ is stable and $\bar{\mathbb E}^s:\Omega^{il}_\star(\Rmininf)\longrightarrow \Imin^s_{\infty}(\Rmininf)$ is well-defined. If 
$(\bigA,\bigB)$ are \textit{not} interlacing, then some row or column of $(M_{ij})_{i,j}$ will not stabilise. Thus the stable generalised min-ideal fillings must all come from interlacing sequences. It follows that $\bar{\mathbb E}^s$ is surjective. Since  $\Omega^{il}_\star(\Rmin)=\Omega^{il}_\star(\Rmininf)\cap \Omega_\star(\Rmin)$  and $\Imin^s_{\infty}(\Rmin)=\Imin^s_{\infty}(\Rmininf)\cap \Imin_\infty(\Rmin)$, it follows that ${\mathbb E}^s$ is a bijection  as well.

Next we consider the restriction of $\mathbb E$ to $\Omega_\star^\R(\Rmin)$. For $(\bigA,\bigB)\in \Omega_\star^\R(\Rmin)$ the associated min-ideal filling satisfies 
\begin{eqnarray*}
\lim_{i\to\infty}M_{ij} &=&\lim_{i\to\infty}\min(A_i,B_j)\ \le\ B_j\ <\ \infty\\
\lim_{j\to\infty}M_{ij}& =&\lim_{j\to\infty}\min(A_i,B_j)\ \le\ A_i\ <\ \infty,
\end{eqnarray*}
which implies that $(M_{ij})_{i,j}$ is asymptotically real. Therefore $\mathbb E^\R:\Omega_\star^\R(\Rmin)\to\Imin^\R_{\infty}(\Rmin)$ is well-defined as restriction of $\mathbb E$. The construction of the inverse of $\mathbb E$ applied to an asymptotically real min-ideal filling straightforwardly produces real parameters, so that $\mathbb E^\R$ is seen to be a bijection. 

Finally, for the `divergent' case note that the condition $\lim_{i\to\infty} A_i=\lim_{j\to\infty} B_j=\infty$ is equivalent to  $\lim_{i\to\infty}\min(A_i,B_i)=\infty$. Therefore we clearly have that  $(\bigA,\bigB)\in \Omega_\star^{div}(\Rmin)$ if and only if $\mathbb E(\bigA,\bigB)\in\Imin_\infty^{div}(\Rmin)$.
\end{proof}

\section{Lusztig weight asymptotics}\label{s:VK}

In this section we show that the parameters $(\bigA, \bigB)$ of asymptotically real infinite min-ideal fillings from \cref{t:TropEdrei} may be obtained as limits from the parameters of the finite min-ideal fillings given by Lusztig's weight map. Let us denote by $\mathcal L^{(n+1)}$ the restriction of Lusztig's type $A_n$ weight map to tropical Toeplitz matrices,  
\begin{equation}\label{e:ParamByLn}
\begin{array}{cccc}
\mathcal L^{(n+1)}:&\ToeplitzU_{n+1}(\Rmin)&\longrightarrow &\mathfrak{h}^*_{PSL_{n+1}}(\R).\end{array}
\end{equation}
This map is a bijection by \cref{t:TropFinite}. Moreover, $\ToeplitzU_{n+1}(\Rmin)$ is just $\Imin_{n+1}$ in terms of standard coordinates by  \cref{t:MinIdealIsToepl}.

\begin{theorem}\label{t:VKtypethm}
Let $M^{(\infty)}=(M_{ij})_{i,j\in\N}\in\Imin^\R_\infty(\Rmin)$ be an asymptotically real min-ideal filling, and $M^{(n+1)}=\pi_{n+1}(M)$ be its projection to $\Imin_{n+1}$. Suppose
\begin{equation}\label{e:Ln+1}
\mathcal L^{(n+1)}(M^{(n+1)})=\begin{pmatrix}
\lambda^{(n+1)}_1 & &&&\\
&\lambda^{(n+1)}_2 & &&\\
&&\ddots &&\\
&&&\lambda^{(n+1)}_n & \\
&&&&\lambda^{(n+1)}_{n+1}
\end{pmatrix}.
\end{equation}
Then the limits
\begin{equation}\label{e:limits}
\lim_{n\to\infty}\frac{\lambda^{(n+1)}_i}n=A_i\qquad \text{and}\qquad \lim_{n\to\infty}\frac{-\lambda^{(n+1)}_{n+2-j}}{n}=B_j,
\end{equation}
exist in $\R$ and the sequences $\bigA=(A_i)_i$ and $\bigB=(B_j)_j$  define an element $(\bigA,\bigB)\in\Omega^{\R}_\star(\Rmin)$ which precisely recovers  $M^{(\infty)}=\mathbb E^{\R}(\bigA,\bigB)$.
\end{theorem}

\begin{remark}
Here the relation between the $A_i$ and the $B_i$ gains an interpretation as the limit of the involution $-w_0$ on $\mathfrak{h}_{SL_{n+1}}^*$, which applied to $\mathcal L^{(n+1)}(M^{(n+1)})$ from \eqref{e:Ln+1} gives the diagonal matrix with $j$-th diagonal entry $-\lambda^{(n+1)}_{n+2-j}$.
\end{remark}
\begin{proof}
By \cref{t:TropEdrei} there is a unique element $(\bigA,\bigB)\in\Omega^{\R}_\star(\Rmin)$ such that $M^{(\infty)}=(M_{ij})=\mathbb E^{\R}(\bigA,\bigB)$. We just  need to prove the identities \eqref{e:limits} for this $(\bigA,\bigB)=((A_i)_i,(B_j)_j)$.
By the definition of $\mathcal L^{(n+1)}$, see  \cref{Lweightmap} and also \cref{ex:UtropFin}, we have that 
\[
\lambda^{(n+1)}_i=\left(M_{i1}+\dotsc+M_{i,n-i+1}\right)-\left(M_{1,n-i+2}+\dotsc + M_{i-1,n-i+2}\right),
\]
with $(n-i+1)$ positive sign summands, and $(i-1)$ negative sign summands. 
Now recall that for each index $h$ we have $A_h=\lim_{k\to\infty} M_{h,k}$, by the proof of \cref{t:TropEdrei}. Let $\ep>0$ and pick $N_\ep$ such that $|M_{hk}-A_h|<\ep$ for all $h\in [1,i]$ whenever $k> N_\ep$. Next suppose $n$ is large enough so that $n-i+1>N_\ep$. Then we have
\begin{equation*}
\begin{array}{ccl}
|\lambda^{(n+1)}_i-nA_i|&\le &\left(|M_{i1}-A_i|+\dotsc+|M_{i,n-i+1}-A_i|\right)+\left(|M_{1,n-i+2}|+\dotsc + |M_{i-1,n-i+2}|\right)+(i-1)|A_i|\\
&\le& \sum_{k=1}^{N_\ep}|M_{ik}-A_i|+ \ep (n-i+1-N_\ep) +\left(\sum _{h=1}^{i-1}(|M_{h,n-i+2}-A_{h}|+|A_h|)\right) + (i-1)|A_i|\\
&\le& \sum_{k=1}^{N_\ep}|M_{ik}-A_i|+ \ep (n-i+1-N_\ep) +\left((i-1)\ep +\sum _{h=1}^{i-1}|A_h|\right) + (i-1)|A_i|.
\end{array}
\end{equation*}
All of the summands on the right-hand side apart from $\ep(n-i+1-N_\ep)$ are independent of $n$. Thus dividing by $n$ and then taking the limit $n\to\infty$ we find that
\[
\lim_{n\to\infty}\left|\frac{\lambda^{(n+1)}_i}n-A_i\right|\le \ep.
\]
This holds for any $\ep>0$ and so the first part of \eqref{e:limits} is proved in this case. The formula for $B_j$ follows by symmetry.  
\end{proof}
\begin{remark}
The `asymptotically real' condition in \cref{t:VKtypethm} cannot be dropped. Let us construct an example of an $(M_{ij})\in\Imin_{\infty}(\Rmin)$ that is not asymptotically real, where  $A_i:=\lim_{k\to\infty}M_{ik}=\infty$ but
\[
\lim_{n\to\infty}\frac{\lambda_i^{(n+1)}}{n}\ne \infty,
\]
for some $i$. Suppose $(\bigA,\bigB)\in \Omega_\star(\Rmin)$ with  $A_i=\infty$ and for all $i$. Then $\bigB=(B_j)_{j\in\N}$ is a divergent sequence that is weakly monotonely increasing. We let $(M_{ij})=\mathbb E(\bigA,\bigB)$, thus $M_{ij}=\min(A_i,B_j)=B_j$ for all $i,j$. In this case we have
\begin{equation*}
\lambda^{(n+1)}_i=B_1+B_2+\dotsc +B_{n-i+1}-(i-1)B_{n-i+2}.
\end{equation*}
Now fix $i\ge 2$, for example $i=2$. This gives  
\[
\lambda^{(n+1)}_2= B_1+B_2+\dotsc +B_{n-1}-B_{n}.
\]
Then setting $B_1=1$ and $B_n=\sum_{i=1}^{n-1} B_i$ recursively defines a divergent sequence, namely $\bigB=(B_n)_n$ is given by $B_1=1$ and $B_n=2^{n-2}$ for $n\ge 2$, such that we have 
\[
\lambda^{(n+1)}_2=0\qquad {\text {for all $n\in\N$.}}
\]
Thus $\lim_{n\to\infty}\frac{\lambda^{(n+1)}_2}{n}=0$, while $A_2=\infty$. We see that the conclusion of \cref{t:VKtypethm} does not hold in this  example. 
\end{remark}
\begin{remark}\label{r:rhoDominant} Note that the tropical parameter sequences $\bigA=(A_i)_i$ and $\bigB=(B_j)_j$ are always weakly increasing, which is not the case for the finite parameter sequences $(\lambda^{(n+1)}_k)_k$, and is a phenomenon that only emerges in the (normalised) limit as $n\to\infty$. For example, suppose we consider the constant ideal filling equal to $1$ in  $\Imin^{(n+1)}$. Its Lusztig weight is the sum of the positive roots for $PSL_{n+1}$, let us denote it $2\rho^{(n+1)}$, and is regular dominant. So we have $\lambda^{(n+1)}_1> \lambda^{(n+1)}_2> \dots>\lambda^{(n+1)}_{n+1} $ --  decreasing instead of increasing. But looking in more detail we find that $\frac{\lambda_i^{(n+1)}}{n}=1-\frac{2i-2}{n}$, so that as $n\to\infty$ the limit becomes constant equal to $1$ for all $i$, which is indeed the associated  sequence $\bigA=(A_i)_i$. In fact, $-w_0(2\rho^{(n+1)})=2\rho^{(n+1)}$ in this case and $\bigB=\bigA$ is also the constant sequence $(1)$. 
\end{remark}
\section{General interpretation of the tropical Schoenberg parameters}\label{s:PuiseauxEdrei}

\subsection{Infinite totally positive Toeplitz matrices} We now turn our attention to infinite upper-triangular totally positive Toeplitz matrices over $(\RR,\RR_{>0},\Val)$, compare \cref{s:Semifields}. The goal is to relate our results about infinite min-ideal fillings to the classical Edrei theorem. 


\begin{defn}\label{d:infToeplitz}
Let $\pr^{n+1}_n:U_+^{(n+1)}(\RR_{>0})\to U_+^{(n)}(\RR_{>0})$ be the map defined by deletion of the last row and column, and denote also by $\pr^{n+1}_n$ its restriction to $\ToeplitzU_{n+1}(\RR_{>0})$. 
We set
\begin{equation*}
\begin{array} {lcc}
U_+^{(\infty)}(\RR_{>0})&:=&\underset{\longleftarrow}\lim  \ U_+^{(n+1)}(\RR_{>0}),\\ 
\ToeplitzU_{\infty}(\RR_{>0})&:=&\underset{\longleftarrow}\lim\ToeplitzU_{n+1}(\RR_{>0}).
\end{array}
\end{equation*}
Let $\pr_{n+1}$ denote the projection map $U_+^{(\infty)}(\RR_{>0})\to U_+^{(n+1)}(\RR_{>0})$ as well as its restriction to $\ToeplitzU_{\infty}(\RR_{>0})$.
\end{defn}

\begin{remark} The above definition is  analogous to \eqref{e:IMaslimit}. Note that, as in the tropical setting, the maps $\pr^{n+1}_n$ and $\pr_{n+1}$ are not surjective. This is because of the additional minors occurring in the extended matrix that make the total positivity condition increasingly restrictive as $n$ grows. 
\end{remark}
We have the following compatibility lemma. Recall the indexing set $\mathcal S_{\le n+1}$ from \eqref{e:Sn+1}. We focus on the standard positive coordinate chart $\phi^{\mathbf i_0}_{>0}=\phi^{\mathbf i_0,(n+1)}_{>0}$ for $U_+^{(n+1)}$ with coordinates $m_{ij}$ indexed by $\mathcal S_{\le n+1}$, see  \cref{d:standardcoordFin}. The image of the inclusion \begin{equation*}\label{e:inclusion-ideal-coordinates}
\begin{array}{cccc}
\iota_{n+1}:&\ToeplitzU_{n+1}(\RR_{>0})&\hookrightarrow &U^{(n+1)}_+(\RR_{>0}) 
\end{array}
\end{equation*}
in terms of  $(m_{ij})_{(i,j)\in\mathcal S_{\le n+1}}$ was described in \cref{p:ToeplitzviaQ}.  It turns out that the maps $\pr^{n+1}_n$, the coordinate charts $\phi^{\mathbf i_0,(n+1)}_{>0}$ and the above inclusions are all straightforwardly compatible.

\begin{lemma} \label{l:idealrestriction}
We have the following commutative diagram  
\[\begin{tikzcd}
	{\ToeplitzU_{n+1}(\RR_{>0})} && {U_{+}^{(n+1)}(\RR_{>0})} && {\ (\RR_{>0})^{\mathcal S_{\le n+1}}} \\
	{\ToeplitzU_{n}(\RR_{>0})} && {U_{+}^{(n)}(\RR_{>0})} && {\ (\RR_{>0})^{\mathcal S_{\le n}}}
	\arrow["{\iota_{n+1}}", hook, from=1-1, to=1-3]
	\arrow["{\pr^{n+1}_n}", from=1-1, to=2-1]
	\arrow["{\sim}"', from=1-5, to=1-3]
	\arrow["{\pr^{n+1}_n}", from=1-3, to=2-3]
	\arrow["{\qquad}", from=1-5, to=2-5]
	\arrow["{\iota_n}", hook, from=2-1, to=2-3]
	\arrow["{\sim}"', from=2-5, to=2-3]
\end{tikzcd}\]
where horizontal maps on the right are the appropriate standard positive charts $\phi^{\mathbf i_0}_{>0}$, and the vertical map on the right-hand side is the coordinate projection.  
\end{lemma}
\begin{proof}
The commutativity of the left-hand side square is obvious. Suppose now $u=\phi_{>0}^{\mathbf i_0,(n+1)}(m_{ij})$. \cref{l:mijFormula} gives a formula expressing the $m_{ij}$ in terms of minors of $u$ that doesn't depend on $n$. This implies that the horizontal maps on the right-hand side are compatible and that square also commutes. 
\end{proof}

\begin{defn}  Suppose $\pr^{n+1}_{n}:X^{(n+1)}\to X^{(n)}$ is a  projective system of varieties over $\RR$ with (individual) positive structures. By a positive structure on the projective limit, we mean of a  set of positive charts $\phi^{(n)}:\RR_{>0}^{I^{(n)}}\to X^{(n)}(\RR_{>0})$ and inclusions $I^{(n)}\hookrightarrow I^{(n+1)}$, that are \textit{compatible} meaning we have commutative diagrams 
\[\begin{tikzcd}
	{\RR_{>0}^{I^{(n+1)}}} && {X^{(n+1)}(\RR_{>0})} \\
	{\RR_{>0}^{I^{(n)}}} && {X^{(n)}(\RR_{>0}).}
	\arrow["{\phi^{(n+1)}}", from=1-1, to=1-3]
	\arrow["{\qquad} "', from=1-1, to=2-1]
	\arrow["{\pr^{n+1}_n}", from=1-3, to=2-3]
	\arrow["{\phi^{(n)}}", from=2-1, to=2-3]
\end{tikzcd}\]
Here the left-hand vertical map is the coordinate projection  associated to the inclusions $I^{(n)}\hookrightarrow I^{(n+1)}$. The positive part of $ X^{(\infty)
}=\underset{\longleftarrow}\lim\, X^{(n)}$   is then defined as 
\[ 
X^{(\infty)}(\RR_{>0}):=\underset{\longleftarrow}\lim\,
 X^{(n)}(\RR_{>0}),
\]
and setting $I^{(\infty)}:=\underset{\longrightarrow}\lim \, I^{(n)}$, we have the projective limit $\phi^{(\infty)}:=\underset{\longleftarrow}{\lim}\ \phi^{(n)}$ giving us a positive chart
\[
\phi^{(\infty)}:\RR_{>0}^{I^{(\infty)}}\longrightarrow X^{(\infty)}(\RR_{>0}).
\]
\end{defn}

By \cref{l:idealrestriction}, the positive charts $\phi^{\mathbf {i_0},(n+1)}_{>0}$  for $U_+^{(n+1)}$ with positive coordinates $(m_{ij})_{(i,j)\in\mathcal S_{n+1}}$ give a positive structure on $U^{(\infty)}_+(\RR)$, and determine a positive chart 
\[
\phi^{(\infty)}_{>0}:\RR_{>0}^{\N\times \N}\longrightarrow U_+^{(\infty)}(\RR_{>0}).
\]  
Let us call the associated coordinates $(m_{ij})_{i,j\in\N}$ the \textit{standard coordinates} of $U_+^{(\infty)}$.

We can now obtain a description of $\ToeplitzU_\infty(\RR_{>0})$ inside $U_+^{(\infty)}(\RR_{>0})$ in terms of the standard  coordinates 
using an infinite analogue of the quiver $Q_{n+1}$ from \cref{f:QuiverExample}.  
\begin{defn}\label{d:Qinf}
Consider the (infinite) quiver $Q_\infty=(\mathcal V,\mathcal A)$ with vertex set $\mathcal V=\N^2\cup\{(0,0)\}$ and arrows as follows, 
\begin{itemize}
\item $(0,0)\overset{c_k}\To (1,k)$ for $k\in\N$,
\item $(i,j)\overset{c_{ij}}\To(i+1,j)$ (vertical arrows in column $j$),
\item $(0,0)\overset{d_k}\To(k,1)$ for $k\in\N$,
\item $(i,j)\overset{d_{ij}}\To (i,j+1)$ (horizontal arrows in row $i$).
\end{itemize}
Note that there are two distinct arrows, $c_1$ and $d_1$, pointing from $(0,0)$ to $(1,1)$.
\end{defn}
\begin{figure}
\[\begin{tikzcd}
	{\quad}& {\quad} 	& {\quad}	& {\quad}	& {\quad}& {\quad}	\\
	& { \bullet_{(4,1)}} & {\bullet_{(4,2)}} &\bullet_{(4,3)} &\bullet_{(4,4)} & {\quad} \\
	& { \bullet_{(3,1)}} & { \bullet_{(3,2)}} & {\bullet_{(3,3)}} &{\bullet_{(3,4)}} &{\quad}\\
	& { \bullet_{(2,1)}} & { \bullet_{(2,2)}} & { \bullet_{(2,3)}} & {\bullet_{(2,4)} }&{\quad} \\
	& { \bullet_{(1,1)}} & { \bullet_{(1,2)}} & { \bullet_{(1,3)}} & { \bullet_{(1,4)}} &{\quad} \\
	\circ_{(0,0)}
	\arrow[shift left=3,from=2-2, to=1-2]
	\arrow[from=2-2, to=2-3]
	\arrow[from=2-3, to=2-4]
	\arrow[from=2-4, to=2-5]
	\arrow[from=2-5, to=2-6]
	\arrow[shift left=3,from=3-2, to=2-2]
	\arrow[from=3-2, to=3-3]
	\arrow[shift left=3,from=3-3, to=2-3]
	\arrow[from=3-3, to=3-4]
	\arrow[from=3-4, to=3-5]
	\arrow[from=3-5, to=3-6]
	\arrow[shift left=3,from=4-2, to=3-2]
	\arrow[from=4-2, to=4-3]
	\arrow[shift left=3,from=4-3, to=3-3]
	\arrow[from=4-3, to=4-4]
	\arrow[shift left=3,from=4-4, to=3-4]
	\arrow[from=4-4, to=4-5]
	\arrow[shift left=3,from=5-2, to=4-2]
	\arrow[from=4-5, to=4-6]
	\arrow[from=5-2, to=5-3]
	\arrow[shift left=3,from=5-3, to=4-3]
	\arrow[from=5-3, to=5-4]
	\arrow[shift left=3,from=5-4, to=4-4]
	\arrow[from=5-4, to=5-5]
	\arrow[shift left=3,from=5-5, to=4-5]
	\arrow[from=5-5, to=5-6]
	\arrow[shift left=3,curve={height=-6pt}, from=6-1, to=1-2]
	\arrow[shift left=3,curve={height=-6pt}, from=6-1, to=2-2]
	\arrow[shift left=3,curve={height=-6pt}, from=6-1, to=3-2]
	\arrow[shift left=3,curve={height=-6pt}, from=6-1, to=4-2]
	\arrow[shift left=3,curve={height=-6pt}, from=6-1, to=5-2]
	\arrow[shift left=3,from=2-3, to=1-3]
	\arrow[shift left=3,from=2-4, to=1-4]
	\arrow[shift left=3,from=2-5, to=1-5]
	\arrow[shift left=3, from=3-4, to=2-4]
	\arrow[shift left=3, from=3-5, to=2-5]	
	\arrow[shift left=3, from=4-5, to=3-5]	
	\arrow[curve={height=6pt}, from=6-1, to=5-2]
	\arrow[curve={height=6pt}, from=6-1, to=5-3]
	\arrow[curve={height=6pt}, from=6-1, to=5-4]
	\arrow[curve={height=6pt}, from=6-1, to=5-5]
	\arrow[curve={height=6pt}, from=6-1, to=5-6]
\end{tikzcd}\]
\caption{The infinite quiver $Q_\infty$. \label{f:InfQuiverExample}}
\end{figure}
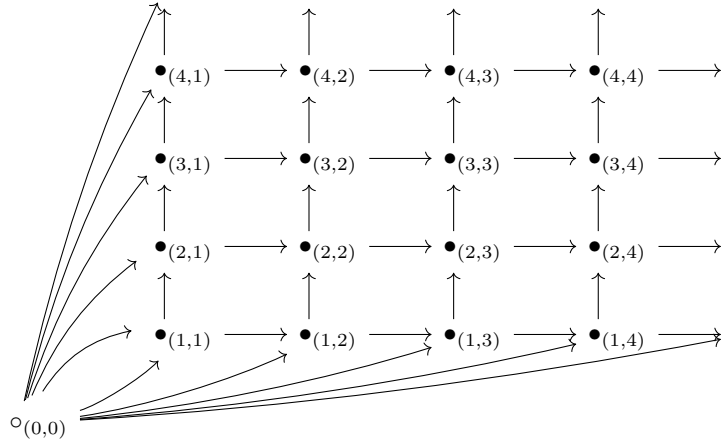
Recall the notations from \cref{d:Qn+1}.
\begin{cor} \label{c:infQcoords}
Associate to $u\in U^{(\infty)}_{+}(\RR_{>0})$ with standard coordinates $(m_{ij})_{i,j\in\N}$ the vertex labeling $(v_{(i,j)})_{(i,j)\in\mathcal V}$ of $Q_{\infty}$ given by
\[
v_{(i,j)}=\sum_{k=1}^{\min(i,j)-1} \frac 1{m_{i-k+1,j-k+1}},
\]   
where $v_{(0,0)}=0$, for the case where the sum is empty. Consider the arrow labelling $(\kappa_a)_{a\in\mathcal A}$ given by
\[\kappa_a=v_{h(a)}-v_{t(a)}.
\] 
Then $u\in \ToeplitzU_{\infty}(\RR_{>0})$ if and only if at every vertex $v\ne(0,0)$ we have
\[
\prod_{a\in\operatorname{in}(v)}\kappa_a=\prod_{a\in\operatorname{out}(v)}\kappa_a.
\]
In this case also $\kappa_a\in\RR_{>0}$ for all arrows $a\in\mathcal A$. 
\end{cor}
\begin{proof} This is a direct consequence of \cref{l:idealrestriction} and \cref{p:ToeplitzviaQ}.
\end{proof}

The standard coordinates $m_{ij}$ for a given element $u\in U_+^{(\infty)}$ can also be computed directly in terms of minors of the infinite matrix $u$ via \begin{equation}\label{e:mviaminorsInf}
m_{ij}=\frac{\Minor_{[i]+j}^{[i]}(u)\Minor_{[i-1]+(j-1)}^{[i-1]}(u)}{\Minor^{[i]}_{[i]+(j-1)}(u)\Minor_{[i-1]+j}^{[i-1]}(u)},
\end{equation}
as in \eqref{e:mviaminors}. We recall the following result proved by Edrei about these precise minor expressions when $u$ is a totally positive Toeplitz matrix over $\R$ (where all $m_{ij}$ are defined and  strictly positive). Consider the factorisation of the generating function associated to $u=(c_{i-j})$ provided by the classical Edrei theorem (\cref{t:EdreiIntro}),  
\[
1+c_1x+c_2x^2+c_3x^3+\dotsc =e^{\gamma x}\,\frac{\prod_{j=1}^\infty (1+\beta_jx)}{\prod_{i=1}^\infty(1-\alpha_ix)}.
\]
We call the $\alpha_i,\beta_j\in \R_{>0}$ the Schoenberg parameters of $u$. 
\begin{prop}[{\cite[Assertion iii]{Edrei53}}]\label{p:Edreilimits} For $u\in\ToeplitzU_\infty(\R_{>0})$, the minor expressions $m_{ij}$ from \eqref{e:mviaminorsInf} converge separately in~$i$ and in~$j$. Moreover,
$\lim_{k\to\infty}(m_{ik})=\alpha_i$ and $\lim_{k\to\infty}(m_{kj})=\beta_j$ for the Schoenberg parameters $\alpha_i$ and $\beta_j$ of $u$. 
\end{prop}
This proposition is part of a complete proof that Edrei provided of \cref{t:EdreiIntro} as an alternative to  \cite{ASW,Edrei52}.

\subsection{Tropicalisations for infinite Toeplitz matrices}
Let  $(\RR,\RR_{>0},\Val)$ be a ring with valuative positive structure as in \cref{d:semifieldvaluation}. We first construct an  associated tropicalisation $U_{+,\RR_{>0}}^{(\infty)}(\Rmin)$ of $U_+^{(\infty)}$. 
\begin{defn}\label{d:Uinftytrop}
Let us define an equivalence relation $\sim$ on $U_+^{(\infty)}(\RR_{>0})$ by declaring $\phi^{(\infty)}_{>0}((m_{ij}))$ and $\phi^{(\infty)}_{>0}((m'_{ij}))$ equivalent if and only if $\Val(m_{ij})=\Val(m'_{ij})$ for all $i,j\in\N$. Then we consider as our tropical analogue of $U_+^{(\infty)}$ the set $ 
U_+^{(\infty)}(\RR_{>0})/\sim$,
where for each standard coordinate $m_{ij}$ we have a well-defined coordinate function on  that sends $[\phi^{(\infty)}_{>0}(m_{ij})]$ to $M_{ij}=\Val(m_{ij})$. This tropicalisation depends on $\RR_{>0}$ in that the standard coordinates run through the image of the valuation map. We assume this image to be all of $\R_{\min}$, unless stated otherwise, and then we denote this tropicalisation 
\[
U_+^{(\infty)}(\Rmin):=U_+^{(\infty)}(\RR_{>0})/\sim.
\]
We have that the coordinates $(M_{ij})_{i,j\in\N}$ all together give us a bijection
\[
U_+^{(\infty)}(\Rmin)\overset\sim\longrightarrow\R^{\N\times\N},
\]
that we call the \textit{standard tropical chart} for $U_+^{(\infty)}(\Rmin)$. We also refer to the  $M_{ij}$ as the standard (tropical) coordinates on $U_+^{(\infty)}(\Rmin)$.
\end{defn} 

We now define the (standard) tropicalisation  
for $\ToeplitzU_{\infty}$ in an analogous way to \cref{d:tropToeplitz}.  
\begin{defn}[Tropical infinite Toeplitz matrices]\label{d:TropInf}
For $(\RR,\RR_{>0},\Val)$ a ring with valuative positive structure, let $\ToeplitzU_{\infty,{\RR_{>0}}}(\Rmin)$ be the image of the composition 
\[
\ToeplitzU_{\infty}(\RR_{>0})\hookrightarrow
U_+^{(\infty)}(\RR_{>0})\longrightarrow U_+^{(\infty)}(\Rmin).
\]
This subset of $U_+^{(\infty)}(\Rmin)$ depends more strongly on $\RR_{>0}$, hence the $\RR_{>0}$ forms part of our notation. We think of its elements as described in terms of the standard coordinates $(M_{ij})_{i,j\in\N}$ of $U_+^{(\infty)}(\Rmin)$.
\end{defn}
The following lemma is a straightforward consequence of \cref{p:TropToeplitzviaQ}.
\begin{lemma}\label{l:infminideal}
If $(M_{ij})_{i,j\in\N}\in\ToeplitzU_{\infty,\RR_{>0}}(\Rmin)$, then   $(M_{ij})_{i,j\in\N}$ is an infinite min-ideal filling. \qed    
\end{lemma}

We also introduce the following more subtle generalisation of the real infinite totally positive Toeplitz matrices and its tropicalisation. For this we require a topology on the semiring $\RR_{\ge 0}=\RR_{>0}\cup\{0\}$.
\begin{defn}[Restricted totally positive Toeplitz matrices]\label{d:ResToepl} Suppose $(\mathcal R,\mathcal R_{>0})$ is a ring with a positive structure and along with a topology on the semiring $\mathcal R_{\ge 0}=\mathcal R_{>0}\cup\{0\}$. We set 
\[
\ToeplitzU^{res}_{\infty}(\mathcal R_{>0}):=\{(m_{ij})_{i,j}\in\ToeplitzU_{\infty}(\mathcal R_{>0})\mid \exists\ \lim_{k\to\infty}m_{ik} \text{ and } \lim_{k\to\infty}m_{kj} \text{ in $\mathcal R_{\ge 0}$}\}.
\]
\end{defn}
\begin{remark}
For $\mathcal R_{>0}=\R_{>0}$ with the usual Euclidean topology on $\R_{\ge 0}$, the limit conditions in the definition above are automatically satisfied, by \cref{p:Edreilimits}. Therefore this definition does not give anything new over $\R$. As a result, $\ToeplitzU_{\infty}(\mathcal R_{>0})$ and $\ToeplitzU^{res}_{\infty}(\mathcal R_{>0})$ are equally valid generalisations of the set of infinite real totally positive Toeplitz matrices. 
\end{remark}

We now require $\RR$ to be a ring with a topological valuative positive structure as in \cref{d:semifieldvaluation}. 

\begin{defn}[Restricted tropical infinite Toeplitz matrices]\label{d:resTropToepl} For $(\RR,\RR_{>0},\Val)$ a ring  with a topological valuative positive structure, let $\ToeplitzU^{res}_{\infty,{\RR_{>0}}}(\Rmin)$ be the image of the composition 
\[
\ToeplitzU^{res}_{\infty}(\RR_{>0})\hookrightarrow
U_+^{(\infty)}(\RR_{>0})\longrightarrow U_+^{(\infty)}(\Rmin).
\]
\end{defn}
Given these definitions and \cref{l:infminideal} we now have inclusions
\[
\ToeplitzU^{res}_{\infty,\RR_{>0}}(\Rmin)\subseteq\ToeplitzU_{\infty,\RR_{>0}}(\Rmin)\subseteq \Imin_{\infty}(\Rmin).\]
A key question in the infinite case is what choice of $\RR_{>0}$ and version infinite Toeplitz matrix (restricted or not) gives the best extension of $\ToeplitzU_{n+1}(\Rmin)$ to the infinite setting.

\section{The positive semifield $\Kdelpos$ of continuous functions}
\label{s:Kdel}

Because topology plays a greater role in the infinite setting, we will replace the field of generalised Puiseaux series whose $t$-adic topology has very restrictive convergence. Thus in this section we introduce a new ring with topological valuative positive structure that will turn out to be more convenient.  

\begin{defn}\label{d:Kdel}
Let $0<\delta<1$ be fixed and consider $\Kdel=C^0((0,\delta])$, the ring of continuous $\R$-valued functions on $(0,\delta]$.
 We define the following `positive part' for $C^0((0,\delta])$,
 \[\begin{array}{lcl}
\Kdelpos&:=&\left\{k:(0,\delta]\to \R_{> 0}\mid \text{ $k$ is continuous and there exists $K\in\R$ with $\lim_{t\to 0} t^{-K}k(t)\in\R_{>0}$}\right\}.
\end{array}
\]
\end{defn}
We will construct two topologies on $\Kdelpos$  in \cref{d:Kdeltop}. First, though, let us check that $(\Kdel,\Kdelpos)$ has all the desired algebraic properties. 
\begin{lemma}\label{l:Kdelstructure}
The pair $(\Kdel,\Kdelpos)$ has the following  properties. 
\begin{enumerate}
\item $\Kdelpos$ is a subsemifield of $\Kdel$.\item 
For any $k\in\Kdelpos$, the element $K\in\R$ for which $\lim_{t\to 0}t^{-K}k(t)\in\R_{>0}$ is \emph{unique}, and the resulting map $\Val:\Kdelpos\to\Rmin$ defined by $\Val(k):=K$
is a semifield homomorphism. 
\item 
We have a well-defined map $\lc:\Kdelpos\to\R_{>0}$ defined by $\lc(k)=\lim_{t\to 0}t^{-\Val(k)}k(t)$.  
\end{enumerate} 
\end{lemma}
\begin{remark} We refer to the map $\lc:\Kdelpos\to\R_{>0}$ as the `leading coefficient' map, for obvious reasons.  This lemma also implies that $\Kdelz$ is a semiring. The valuation may be extended to $\Val:\Kdelz\to\Rmininf$.
\end{remark}
\begin{proof}
Let $k\in\Kdelpos$.  We have that $\lim_{t\to 0}t^{-K}k(t)\in \R_{>0}$ for some $K\in\R$. Suppose $M\ne K$. Then it is straightforward to verify that 
\begin{equation}\label{e:MneK}
\lim_{t\to 0}t^{-M}k(t)=\begin{cases} 0, & \text{ if $M<K$,}\\
\infty, & \text{ if $M>K$. }
\end{cases}
\end{equation}
Therefore $K$ is uniquely determined by the condition $\lim_{t\to 0}t^{-M}k(t)\in\R_{>0}$, and $\Val(k)=K$ is well-defined. As a consequence, the map  $\lc$ from (3) is also well-defined. 

Consider now $k_1,k_2\in\Kdelpos$ with $\Val(k_i)=:K_i$. Then, $\lim_{t\to 0} (t^{-K_1-K_2}k_1k_2)=\lc(k_1)\lc(k_2)$ lies in $\R_{>0}$. In particular, if $k_1,k_2\in\Kdelpos$ then also $k_1k_2\in\Kdelpos$, and it follows that $\Val(k_1 k_2)=K_1+K_2$. Similarly, but using~\eqref{e:MneK},
\begin{equation}\label{e:lcadd}
\lim_{t\to 0} (t^{-\min(K_1,K_2)}(k_1+k_2))=\delta_{K_1,\min(K_1,K_2)}\lc(k_1)+ \delta_{K_2,\min(K_1,K_2)}\lc(k_2).
\end{equation}
Therefore, if $k_1,k_2\in\Kdelpos$, then $k_1+k_2\in\Kdelpos$ and we have $\Val(k_1 +k_2)=\min(K_1,K_2)$.  
For the inverse of $k\in\Kdelpos$ we have that  
\[
\lim_{t\to 0 }t^K\frac{1}{k(t)}=\lc(k)\inv,
\] 
so that $\Val(k\inv)=-K$. Thus $\Kdelpos$ is a semifield and $\Val$ is a semifield homomorphism. 
\end{proof}

\begin{defn}
Inside $\Kdelpos$ we may define 
\[\begin{array}{ccl}
\Odelpos&:=&\left\{f\in\Kdelpos\mid \Val(f)\ge 0\right\},\\
\Mdelpos&:=&\left\{g\in\Kdelpos\mid \Val(g)> 0\right\}.
\end{array}
\]
\end{defn}
\begin{lemma}
We have that $\Odelpos$ is a subsemiring (with identity) of $\Kdelpos$, and $\Mdelpos$ is a semiring ideal in $\Odelpos$. An element $k\in \Kdel$ extends to a continuous  function on $[0,\delta]$ if and only if $k\in\mathcal O_{>0}$. We denote the extension also by $k$, so that we have $k:[0,\delta]\to\R_{\ge 0}$ in this case. For $k\in\mathcal O_{>0}$ we have that $k(0)=0$ if and only if $f\in\Mdelpos$.  
\end{lemma}

\begin{proof}
The first statement follows directly from the fact that $\Val$ is a semiring homomorphism. The second statement, that $k$ extends to a continuous function on $[0,\delta]$ if $K=\Val(k)\ge 0$ with $k(0)=0$ if and only if $K>0$, follows using \eqref{e:MneK} with $M=0$. The remaining statements are straightforward. 
\end{proof}

\begin{defn}[Strong and weak topologies on $\Kdelpos$]\label{d:Kdeltop} Let $\|\ \|_{\sup}$ denote the supremum norm on $C^0([0,\delta])$. Now supposing $f\in\Odelpos$, then $f$ extends to a continuous function on $[0,\delta]$, that we also denote $f$ as above. For $\ep>0$ and $f\in\Odelpos$ we define 
\[
B_\ep(f)=\{g\in\Odelpos\mid \|g-f\|_{\sup}<\ep\}.
\]
We define the \textit{strong topology} on $\Kdelpos$ to be the topology generated by the open sets $t^AB_\ep(f)$, where $f\in\Odelpos$ and $A\in\R$. 

We define the \textit{weak topology} as the topology combining pointwise convergence on $(0,\delta]$ with convergence of the valuations. Namely, let the weak topology on $\Kdelpos$ be the topology generated by the open sets $B_{x,\ep}(f)=\{g\in\Kdelpos\mid |g(x)-f(x)|<\ep\}$ and $B^{{\Val}}_{\ep}(f)=\{g\in\Kdelpos\mid |\Val(g)-\Val(f)|<\ep\}$, where $f\in\Kdelpos,\ep>0$ and $x\in (0,\delta]$.  

It is straightforward to check that the semifield operations are continuous with respect to these topologies. Therefore $\Kdelpos$ is topological semifield (in two ways). We may write $\Kdelposst$ and $\Kdelposwk$ for $\Kdelpos$ endowed with the strong or the weak topology, respectively.  We also have induced weak and strong topologies on $\Mdelpos$ and~$\Odelpos$.
\end{defn}

\begin{remark}

Note that we may embed $\Odelpos$ into $C^0([0,\delta])$, and that $C^0([0,\delta])$ is naturally a Banach-algebra via its $\|\ \|_{\sup}$-topology. The induced topology on $\Odelpos$ differs from both the weak and strong topologies on $\Odelpos$. Namely the strong topology on $\Odelpos$ is finer than the $\|\ \|_{\sup}$-topology, and the weak topology is coarser than the $\|\ \|_{\sup}$-topology. For example, the sequence $f_n(t)=t^{\frac{n+1}{n}}$ in $\Odelpos$ converges to $f(t)=t$ in terms of $\|\  \|_{\sup}$, but it does not converge in the strong topology of $\Kdelpos$. Meanwhile, the sequence $f_n(t)=t^{\frac{1}{n}}$ converges to $f(t)=1$ in the weak topology of $\Odelpos$, but doesn't converge to $1$ in terms of $\|\ \|_{\sup}$.
\end{remark}

We extend these topologies to $\Kdelnn=\Kdelpos\cup\{0\}$ as follows. 
\begin{defn}[Strong and weak topologies on $\Kdelnn$]\label{d:Kdeltop}\label{d:topCnn}
To define $\Kdelnnst$, we add the set $\{0\}$ to the open sets of $\Kdelposst$, and together let them generate the \textit{strong topology} on $\Kdelnn$. Thus, strong convergence to $0$ requires stabilisation of a sequence to $0$. This is in keeping with stabilisation of $\Val$ on convergent sequences in the strong topology that we will observe in \cref{l:valcont}. For the \textit{weak topology}, $\Kdelnnwk$, we introduce open neighbourhoods of $0$ of the form $B_N(0)=\{f\in\Kdelnn\mid \Val(f)>N\}$. Thus for a sequence to converge to $0$ in the weak toplogy its valuations must diverge to $\infty$. 
\end{defn}

\begin{remark}
Note that the semifield $\Kdelpos$ cannot be embedded into any field. Indeed,  any ring containing $\Kdelpos$ will have zero-divisors.  For example, consider two distinct elements $f_1\le f_2$ in $\Kdelpos$ that agree outside $(\frac{\delta}4,\frac{\delta}3)$ and  two further distinct elements $g_1\le g_2$ that agree outside $(\frac{\delta}2,\delta]$. Then $(f_2-f_1)(g_2-g_1)=0$.     
\end{remark}

\subsection{The strong topology}\label{s:strongtop}
We collect some properties of the strong topology on $\Kdelpos$.
\begin{lemma}\label{l:valcont}
The valuation map $\Val:\Kdelposst\to\Rmin$ is continuous for the  discrete topology on $\Rmin$. 
In particular, for any convergent sequence $(k_n)_n$ in $\Kdelposst$, the sequence $\Val(k_n)$ stabilises, that is,
\[
\Val(k_n)=\Val(k)\quad \text{ if $n>N$, for some $N\in\N$,}
\]
where $k=\lim_{n\to\infty}(k_n)$. The subsets $\Mdelpos$ and $\Odelpos$ are open and closed inside $\Kdelposst $.
\end{lemma}
\begin{proof}
The open sets are translates (in the sense of multiplying by a $t^A$) and unions of translates, of open sets in $\Odelpos$ for the sup-norm.
 Therefore, for any $f\in\Kdelposst$, the open sets $\mathcal U_{\ep,A}(f)=t^{A} B_\ep(t^{-A}f)$ where $A\in\R$ with $A\le\Val(f)$, and $\ep>0$, form a neighborhood basis of $f$.

Pick any $K\in\Rmin$ and let $\mathcal S=\Val\inv(\{K\})$. It suffices to show that $\mathcal S$ is open. Let $k\in\mathcal S$ and pick $\ep$ with $0<\ep<\lc(k)$. We now show that the open neighborhood $\mathcal U_{\ep,K}(k)$ of $k$ lies in $\mathcal S$. Let $g\in \mathcal U_{\ep,K}(k)$. This means that $t^{-K} g\in B_\ep(t^{-K}k)$, so that 
\[
\sup_{t\in{(0,\delta]}}|t^{-K}g(t)-t^{-K}k(t)|<\ep.
\] 
It follows that for all $t\in(0,\delta]$ we have
\[
t^{-K}k(t)-\ep\ <\ t^{-K}g(t)\ <\ t^{-K}k(t)+\ep.
\]
Taking the limit $t\to 0$ gives 
\begin{equation}\label{e:lc-ineq}
\lc(k)-\ep\ \le \ \lim_{t\to 0} (t^{-K}g(t))\ \le\ \lc(k)+\ep,
\end{equation}
using that $\lim_{t\to 0}(t^{-K}k(t))=\lc(k)$ since $K=\Val(k)$.
However, we chose $\ep<\lc(k)$ and therefore the left-hand side of the above inequality is strictly positive. Now \eqref{e:lc-ineq} implies that $\lim_{t\to 0}(t^{-K}g(t)$ lies in $\R_{>0}$, which this implies that $\Val(g)=K$, and so $g\in \mathcal S$. 
\end{proof}
\begin{remark}
It follows immediately from the combination of \cref{l:valcont} and \cref{d:Kdeltop} that the extended valuation map $\Val:\mathcal C^{\mathrm st}_{\ge 0}\longrightarrow \Rmininf$ is continuous for the discrete topology on $\Rmininf$. Moreover, it is then continuous also for the extended Euclidean topology on $\Rmininf$ given in  \cref{d:semifieldvaluation}. 
\end{remark}
We now turn to the  leading coefficient map $\lc:\Kdelpos\to\R_{>0}$ defined by $\lc(k):=\lim_{t\to 0}(t^{-\Val(k)}k(t))$. 
\begin{lemma}\label{l:lccont}
We have that $\lc:\Kdelposst\to\R_{>0}$ is continuous for $\R_{>0}$ with the usual Euclidean topology. 
\end{lemma} 
\begin{proof}
Let $C=(c_1,c_2)$ be an open interval in $\R_{>0}$. For any $k\in\lc\inv(C)$ we want to find an open neighborhood of $k$ that also lies in $\lc\inv(C)$. By \cref{l:valcont} we have $\mathcal U_K:=\Val\inv(K)$ is an open neighborhood of $k$.  
Now suppose $\lc(k)=c$ and pick $\ep>0$ small enough so that $c_1<c-\ep<c+\ep<c_2$. We may consider the open neighborhood of $k$ given by $\mathcal U:=t^KB_{\ep}(t^{-K} k)\cap \mathcal U_K$.  Then for any $f\in\mathcal U$ we have $\Val(f)=K$ and $\sup_{t\in (0,\delta]}|t^{-K}f(t)-t^{-K}k(t)|<\ep$. Therefore, letting $t\to 0$ we find that $|\lc(f)-\lc(k)|=|\lc(f)-c|\le\ep$. This implies that $\lc(f)\in C$.
\end{proof}

Note that $\lc$ is not a map of semifields. It is compatible with multiplication, $\lc(k_1 k_2)=\lc(k_1)\lc(k_2)$, but not with addition, compare \eqref{e:lcadd}. However, we do have a semifield homomorphism in the other direction as we will see in \cref{l:Rincl}.

\begin{lemma}\label{l:Kcontcrit}
Let $(k_n)_{n\in\N}$ be a sequence in $\Kdelposst$. Suppose that $\Val(k_n)$ stabilises to $K\in\R$ for $n> N$ and that the sequence $(t^{-K}k_n)_{n>N}$ in $ C^0([0,\delta])$ 
converges uniformly to a strictly positive function. 
Then the sequence $k_n$ is convergent in the strong topology. 
\end{lemma}

\begin{proof}
  Since $\Val(k_n)=K$ for $n>N$, we have indeed that $t^{-K}k_n(t)$ extends to  a continuous function on $[0,\delta]$. Let $\ell(t)$ be the limit of $t^{-K}k_n(t)$ for $t\in [0,\delta]$. By assumption $\ell(t)>0$ for all $t$, and since $t^{-K}k_n$ converges to $\ell$ uniformly, $\ell$ is continuous. It follows that $\ell\in \Kdelpos$, and since $\ell(0)>0$ that $\Val(\ell)=0$. Therefore $k:=t^K\ell\in \Kdelpos$ with $\Val(k)=K$. 
  Now we check that $k$ is the limit of $k_n$ for the strong topology on $\Kdelpos$. Namely consider any open set $\mathcal U_{\ep,A}(k)=t^{A} B_\ep(t^{-A}k)$ from the standard neighborhood basis of $k$ (where $A\in\R$ with $A\le \Val(k)$). We have that $t^{-K}k_n\in B_\ep(t^{-K})$ for $n> N(\ep)$ by construction, and therefore 
  \[
 |t^{-A}k_n(t)-t^{-A}k(t)|= t^{K-A}|t^{-K}k_n(t)-t^{-K}k(t)|<\ep\qquad\forall t\in [0,\delta].
  \] 
for all $n>N(\ep)$, using 
that $t^{K-A}<1$. Thus $t^{-A}k_n\in B_\ep(t^{-A}k)$ and $k_n\in  \mathcal U_{\ep,A}(k)$ for $n>N(\ep)$. 
\end{proof} 

\begin{remark}
\label{r:subtleties} We  clearly obtain elements of $\Kdelpos$ from Laurent or (generalised) Puiseaux series that converge on $(0,\delta]$ with values in $\R_{>0}$. But we now also get elements from series such as  
\[
s(t)=\sum_{n=1}^\infty\frac{1}{2^n}t^{-\frac{1}{n}},
\]
where the exponents do not diverge to $\infty$. 
Here $s(t)$ is the (strongly convergent) limit of its partial sums $s_N(t)=\sum_{n=1}^N\frac{1}{2^n}t^{-\frac{1}{n}}$. The element $s(t)$ has $\Val(s)=-1$ and the sequence $s_N$ in fact converges to $s$ in $\Kdelposst$ because all $s_N$ have valuation $-1$, and $t s_N(t)=\sum_{n=1}^N\frac{1}{2^n}t^{1-\frac{1}{n}}$ is uniformly convergent on $[0,\delta]$ (for example because its terms are bounded by~$\frac{1}{2^n}$) to a strictly positive function.
\end{remark}
\begin{lemma}\label{l:Rincl} The inclusion map $\iota:\R_{>0}\to\Kdelposst$ of $\R_{>0}$ into $\Kdelpos$ as constant functions is an embedding of topological semifields, where $\R_{>0}$ is endowed with its usual Euclidean topology. 
\end{lemma}
 
\begin{proof}
Consider an open set $t^{-A} B_\ep(f)$ in $\Kdelposst$; here $f\in\Odelpos$ and $A\in\R$. Suppose it contains a constant function $r$. Then we have $\ep':=\|t^Ar-f\|_{\sup}<\ep$.  This approximation implies that $A\ge 0$ and $t^A$ is bounded on $(0,\delta]$. Therefore we can find a $\nu>0$ such that $\|t^Ar-t^Ar'\|_{\sup}<\ep-\ep'$ whenever $|r-r'|<\nu$. Finally, 
then 
\[
\|t^Ar'-f\|_{\sup}=\|t^Ar'-t^Ar+t^{A}r-f\|_{\sup}\le \|t^Ar'-t^Ar\|_{\sup} +\|t^{A}r-f\|_{\sup}< \ep-\ep'+\ep'=\ep.
\]
This implies that the neighbourhood $(r-\nu,r+\nu)$ of $r$ also lies in 
$t^{-A} B_\ep(f)$, so that the intersection of this open set of $\Kdelposst$ with $\R_{>0}$ is open.  Thus any open set of $\Kdelposst$  intersects $\R_{>0}$ in an open set for the Euclidean topology.
 In other words, the inclusion map $\iota$ is continuous for $\Kdelposst$ (and therefore also for  $\Kdelposwk$). The inverse map is the restriction of $\lc$ to $\iota(\R_{>0})$. Since $\lc$ is continuous for the strong topology, by \cref{l:lccont}, so is its restriction, and  we have shown that $\Kdelposst$ induces the usual Euclidean topology on $\R_{>0}$. 
\end{proof}

\subsection{The weak topology}\label{s:weaktop}
Convergence of a sequence $(f_n)_n$ in $\Kdelposwk$, that is, with respect to the weak topology, is equivalent to the combination of  pointwise convergence of $f_n$ on $(0,\delta]$ together with convergence of the valuations $\Val(f_n)$ in $\R$. We record the following basic properties. 
\begin{lemma}\label{l:valcontwk}
The valuation map $\Val:\Kdelposwk\to\Rmin$ is continuous for the Euclidean topology on $\Rmin$. 
We have that $\Mdelpos$ is open in $\Kdelposwk$, and $\Odelpos$ is its closure. 
\end{lemma}
\begin{proof}
That the valuation map is continuous is immediate from \cref{d:Kdeltop}. It follows from this that $\Mdelpos=\Val\inv(\R_{>0})$ is open, and $\Odelpos=\Val\inv(\R_{\ge 0})$ is closed. If $\Val(f)=0$ then $t^{\frac{1}n}f\in\Mdelpos$, and weakly converges to $f$ (i.e. converges in the weak topology) as $n\to\infty$. Thus $\Odelpos$ is the closure of $\Mdelpos$. 
\end{proof}

\begin{remark}
It now follows from the combination of \cref{l:valcontwk} and \cref{d:Kdeltop} that the extended valuation map $\Val:\mathcal C^{\mathrm wk}_{\ge 0}\longrightarrow \Rmininf$ is continuous for the extended Euclidean topology on $\Rmininf$ given in  \cref{d:semifieldvaluation}. 
\end{remark}

\begin{remark} The  leading coefficient map $\lc:\Kdelpos\to\R_{>0}$ defined by $\lc(k):=\lim_{t\to 0}(t^{-\Val(k)}k(t))$ is not continuous for the weak topology. For example the sequence $f_n(t)=\frac {1}{n}t^{\frac{1}{n+1}}+
t^{\frac{1}{n}} $ in $\Kdelpos$ converges pointwise to $f(t)=1$ and the valuations converge (correctly) to $0$. Therefore it converges to $f=1$ in the weak topology. But $\lc(f_n)=\frac{1}{n}$ converges to $0$, while $\lc(f)=1$. 
\end{remark}

\begin{lemma}\label{l:Rinclwk} The inclusion of constant functions map $\iota:\R_{>0}\to\Kdelposwk$ is an embedding of topological semifields, where $\R_{>0}$ is endowed with the standard Euclidean topology. 
\end{lemma}
 
 \begin{proof} We  identify elements of $\R_{>0}$ with their corresponding constant functions, and need to consider the subspace topology of $\R_{>0}\subset\Kdelposwk$, showing it agrees with the Euclidean topology. For $f\in\Kdelpos$ and $\ep>0$ we have
 \[
 B_{x,\ep}(f)\cap\R_{>0}=(f(x)-\ep, f(x)+\ep)\,\cap\, \R_{>0}.
\]   
On the other hand, 
\[
B^{\val}_{\ep}(f)\cap \R_{>0}=\begin{cases} \R_{>0}, & \quad\text{if $|\Val(f)|<\ep$,}\\
\emptyset, &\quad\text{if $|\Val(f)|\ge\ep$.} 
\end{cases}
\] 
These open sets generate precisely the Euclidean topology of $\R_{>0}$ and the lemma is proved. 
\end{proof}

\begin{remark}\label{r:Kdeltopposstr}
We have now shown that $(\mathcal C,\Kdelposst,\Val)$ and $(\mathcal C,\Kdelposwk,\Val)$ are  both rings with topological valuative positive structures in the sense of \cref{d:semifieldvaluation}, and that in  both cases $\R_{>0}$ with its usual Euclidean topology embeds as a topological subsemifield. Using $\Kdelpos$ will allow us to better connect the classical Edrei theorem (over $\R_{>0}$) to the tropical Edrei theorem (over $\Rmin$).
\end{remark}
\section{The $\Kdelpos$-tropicalisation of $\ToeplitzU_\infty$}\label{s:KdelTropToep}

In this section we consider totally positive Toeplitz matrices over the valued semifield $\Kdelpos$ from \cref{s:Kdel}. The basic set of `tropical infinite Toeplitz matrices' is then as  defined in \cref{d:TropInf}, or equivalently
\[
\ToeplitzU_{\infty,\Kdelpos}(\Rmin)=\ToeplitzU_{\infty}(\Kdelpos)/\sim,
\]
with $u\sim u'$ if the valuations of the standard coordinates agree. 
We view $\ToeplitzU_{\infty,\Kdelpos}(\Rmin)$ naturally as a subset of $U_+^{(\infty)}(\Rmin)$, with its tropical standard coordinates $M_{ij}$, and recall that $\ToeplitzU_{\infty,\Kdelpos}(\Rmin)\subseteq \Imin_\infty(\Rmin)$ by \cref{l:infminideal}. 
We prove two main results. 

The first main result, which is stated later on as \cref{p:KdelEdrei}, says in essence that the parameters from our tropical Edrei theorem,~\cref{t:TropEdrei}, can be interpreted as valuations of actual Schoenberg-type parameters in $\Kdelpos$. The second main result, which will rely on the first but can be stated independently, is the following. 

\begin{theorem}\label{t:KdelTrop} For tropicalisation using the valuative semifield $\Kdelpos$, the inclusion from \cref{l:infminideal} is an equality, 
\[\ToeplitzU_{\infty,\Kdelpos}(\Rmin)=\Imin_\infty(\Rmin).
\]
\end{theorem}

The constructions in this section involve power series of the type 
\begin{equation*}
\frac{\prod_{i=1}^{\infty}(1+b_j x)}{\prod_{j=1}^{\infty}(1-a_i x)}=\sum_{k=0}^\infty c_k x^k.
\end{equation*} 
These also arise in the theory of supersymmetric functions, about which we now recall some background.

\subsection{Supersymmetric Functions} 
 Supersymmetric functions arise as limits of supersymmetric polynomials, which are characters in the representation theory of supergroups \cite{BalantekinBars,DondiJarvis,BereleRegev}. They also arise somewhat differently in \cite{Macdonald:variations,GouldenGreene94}. We refer to \cite{Macdonald:Book1} for background on symmetric functions and some standard notations including concerning partitions, Young diagrams and tableaux. 

Let $\Lambda_{\mathbf y}\subset\Z[y_1,y_2,\dotsc]$ denote the ring of symmetric functions  in the infinite set of formal variables $\mathbf y=(y_1,y_2\dotsc)$.  We now consider  functions in two sets of variables that are separately symmetric,
\[
\Lambda_{(\mathbf a;\mathbf b)}:=\Lambda_{\mathbf a}\otimes \Lambda_{\mathbf b}.
\]
For $f\in\Lambda_{\mathbf a}$ and $g\in\Lambda_{\mathbf b}$, the element $f\otimes g\in\Lambda_{(\mathbf a;\mathbf b)}$ may also be written as $f(\mathbf a)g(\mathbf b)$ or $g(\mathbf b)f(\mathbf a)$. Define the \textit{super-symmetric complete homogeneous functions} $H_m(\mathbf a||\mathbf b)$ by
\begin{equation}\label{e:generatingH}
\frac{\prod_{i=1}^{\infty}(1+b_j x)}{\prod_{j=1}^{\infty}(1-a_i x)}=1+H_1(\mathbf a||\mathbf b )x+ H_2(\mathbf a||\mathbf b)x^2+\dotsc+H_m(\mathbf a||\mathbf b)x^m
+\dotsc.
\end{equation}
Note that $H_m(\mathbf a||\mathbf b )=\sum_{k=0}^m h_k(\mathbf a)e_{m-k}(\mathbf b)$ in terms of the standard complete and elementary symmetric functions. 
\begin{defn}[Supersymmetric functions] Let $\Lambda(\mathbf a||\mathbf b)$ denote the subring of $\Lambda_{(\mathbf a;\mathbf b)}$ generated by the $H_m(\mathbf a||\mathbf b)$, and call elements of $\Lambda(\mathbf a||\mathbf b)$ `supersymmetric functions'. 
\end{defn}
This subring of $\Lambda_{(\mathbf a;\mathbf b)}$  can be  characterised in terms of a `cancellation property' following \cite{Stembridge85}, thus relating it to the ring of so-called bisymmetric functions introduced in  \cite{MetropolisNicolettiRota} (with an alternative sign convention). The `cancellation property' is that $f((c,a_2,\dotsc),(-c,b_2,\dotsc))$ is independent of $c$, which simply expresses the cancellation of the numerator and denominator in \eqref{e:generatingH} if $a_1=-b_1=c$.

We now define the supersymmetric analogue of the Schur functions. Let $s_{\eta\setminus\mu}$ denote the standard (skew) Schur function, which may be defined positively as a sum of monomials indexed by \textit{semistandard Young tableaux} on $\eta\setminus\mu$. Let $(\ )'$ denote the transpose operation. 
Then the \textit{super-symmetric Schur function} associated to the partition $\lambda$ may be defined by the formula
\begin{equation}\label{e:Sdef}
S_\lambda(\mathbf a||\mathbf b)=\sum_{\nu\subseteq\lambda}s_{\nu}(\mathbf a) s_{(\lambda\setminus\nu)'}(\mathbf b).
\end{equation} 
If $\lambda$ is a partition with $r$ parts then we also have
\begin{equation}\label{e:SJTformula}
S_\lambda(\mathbf a||\mathbf b):=\det(H_{\lambda_i+j-i}(\mathbf a||\mathbf b))_{i,j=1}^{r},
\end{equation}
by a generalised  Jacobi-Trudi formula. Moreover, the $S_\lambda(\mathbf a||\mathbf b)$ are known to be a basis of $\Lambda(\mathbf a||\mathbf b)$, see \cite{BereleRegev,BergeronGarsia,PraTho}. 
There are two combinatorial `tableaux' type formulas that follow from \cite{BereleRegev} and \cite{GouldenGreene94,Macdonald:variations}, respectively. The first one describes $S_\lambda(\mathbf a||\mathbf b)$ as a sum of  infinitely many monomials. We will use second formula, in which the summands are not monomials but are indexed by semistandard Young tableaux. 

Let $\lambda$ be a Young diagram and $T$ be a  tableau, that is, a filling of $\lambda$ with positive integers.  We write $SSYT(\lambda)$ for the set of semistandard Young tableau of shape $\lambda$, recalling that these are fillings of the Young diagram $\lambda$ by positive integers, that are strictly increasing along columns and weakly increasing along rows. Write $T(i,j)$ for the entry of $T$ in row $i$ and column $j$ (where rows and columns are enumerated from the top and left, respectively, as for a matrix). 

Associated to $T\in SSYT(\lambda)$ we define a new tableaux $T^*$ by adding to each entry $T(i,j)$ the `content' $c(i,j)=j-i$ of the $(i,j)$-box,
\begin{equation}\label{e:Tstar}
T^*(i,j):=T(i,j)+j-i.
\end{equation}
Note that $T^*$ is now strictly increasing along rows and weakly increasing along columns. In other words, the transpose of $T^*$ is an SSYT for $\lambda'$. 
Let
 \begin{equation}\label{e:supersymmetricmonomial}
(\mathbf a||\mathbf b)_T:=\prod_{(i,j)\in\lambda}(a_{T(i,j)}+b_{T^*(i,j)}).
\end{equation}
The formula of  Goulden, Greene and Macdonald \cite{GouldenGreene94,Macdonald:variations} reads
\begin{equation}\label{e:GGMformula}
S_\lambda(\mathbf a||\mathbf b)=\sum_{T\in SSYT(\lambda)} (\mathbf a||\mathbf b)_T.
\end{equation}
Here setting $\mathbf b$ to $0$ the above formula recovers the classical combinatorial formula for $s_\lambda(\mathbf a)$  as sum of monomials associated to SSYT (and recovers the formula for  $s_{\lambda'}(\mathbf b)$ if $\mathbf a$ is set to $0$).

\subsection{Construction of infinite totally positive Toeplitz matrices}\label{s:KdelEdrei}

Let $(\RR,\RR_{>0},\Val)$ be a ring with a topological valuative positive structure as in \cref{d:semifieldvaluation}. Recall that in this case $\RR_{\ge 0}$ is a topological semiring and $\Val:\RR_{\ge 0}\to\Rmininf$ is continuous. 
For a Young diagram $\lambda$ let $SSYT_{[n]}(\lambda)$ denote the set of semistandard Young tableaux of shape $\lambda$ with entries in $[1,n]$. If $T$ is a tableau of shape $\lambda$, write $T(b)$ for the entry of $T$ associated to the box $b$ of $\lambda$. 

\begin{lemma}\label{l:SchurVal}
Let $\boldalpha=(\boldalpha_1,\boldalpha_2,\dots )$ be a weakly monotonely decreasing sequence in $\RR_{\ge 0}$ and let $\lambda$ be a partition. 
Suppose that the Schur function associated to $\lambda$ can be evaluated at $\boldalpha$, meaning that the Schur polynomial values
\[ 
s_\lambda(\boldalpha_{[n]}):=s_\lambda(\boldalpha_1,\dotsc,\boldalpha_n)=\sum_{T\in SSYT_{[n]}(\lambda)}\ \prod_{b\in\lambda}\boldalpha_{T(b)},
\] 
converge in $\RR_{\ge 0}$ as $n\to\infty$.  Then setting $A_i=\Val(\boldalpha_i)$ we obtain that
\begin{equation}\label{e:ValSchur}
\Val(s_\lambda(\boldalpha))=A_\lambda:=\sum_{i=1}^r \lambda_i A_i
\end{equation}
holds in $\Rmininf$. If all of the $\boldalpha_i$ lie in $\RR_{>0}$ then so does $s_\lambda(\boldalpha)$ for any $\lambda$.
\end{lemma}
\begin{proof} Note first that by \cref{l:Kdelposet} we have that the $A_i=\Val(\boldalpha_i)$ form a weakly increasing sequence.
Suppose $\lambda=(\lambda_1,\dotsc,\lambda_r)$ and $n\ge r$. Amongst the summands of  $s_\lambda(\boldalpha_1,\dotsc,\boldalpha_n)$  indexed by SSYT of shape $\lambda$, the one with minimal valuation corresponds to the  tableau  $T_0(i,j)=i$, whose $i$-th row boxes are labelled by $i$.  We therefore have $\Val(s_\lambda(\boldalpha_1,\dotsc,\boldalpha_n))=\sum_{i=1}^r \lambda_i A_i=A_\lambda$ for the Schur polynomial values. Note that the right-hand side is independent of $n$. Now since $s_\lambda(\boldalpha)$ is the limit of the $s_\lambda(\boldalpha_1,\dotsc,\boldalpha_n)$ as $n\to\infty$, and $\Val$ is continuous, the identity \eqref{e:ValSchur} follows.
The final statement is clear.
\end{proof}
\begin{lemma}\label{l:EvalSSSchur} 
Suppose $\boldalpha=(\boldalpha_1,\boldalpha_2,\dots )$ and $\boldbeta=(\boldbeta_1,\boldbeta_2,\dots)$  are weakly monotonely decreasing sequences in $\RR_{\ge 0}$ such that $\boldalpha_i+\boldbeta_i\in\RR_{>0}$ for all $i$. 
Suppose moreover that for any partition $\lambda$ the Schur function $s_\lambda$ can be evaluated on $\boldalpha$ and $\boldbeta$, meaning that $s_\lambda(\boldalpha_{[n]})$ and $s_\lambda(\boldbeta_{[n]})$ converge in $\RR_{\ge 0}$ as $n\to\infty$, as in \cref{l:SchurVal}. Then the supersymmetric Schur polynomials $S_\lambda(\boldalpha_{[n]}||\boldbeta_{[n]})$ converge  to a well-defined element  $S_\lambda(\boldalpha||\boldbeta)$ in $\RR_{>0}$ and we have
\begin{equation}\label{e:SkewSchurVal}
\Val(S_\lambda(\boldalpha||\boldbeta))=\sum_{(i,j)\in\lambda}\min(A_i,B_j),
\end{equation}
where $A_i=\Val(\boldalpha_i)$ and $ B_j=\Val(\boldbeta_j)$.
\end{lemma}

\begin{proof}
It is well-known that skew Schur functions $s_{\mu\setminus\eta}$ are positive (integral) linear combinations of Schur functions, see \cite[A1.3.3]{Stanley:BookII}. This implies that we can  evaluate skew Schur functions $s_{\mu\setminus\eta}$ on $\boldalpha$, respectively $\boldbeta$, with values in $\RR_{\ge 0}$.  By the definition of the supersymmetric Schur function in terms of skew Schur functions from \eqref{e:Sdef} we deduce that $S_\lambda(\boldalpha_{[n]}||\boldbeta_{[n]})$ converges to a well-defined element $S_\lambda(\boldalpha||\boldbeta)\in\RR_{\ge 0}$.
Now one (at least) of the sequences $\boldalpha,\boldbeta$ must lie entirely in $\RR_{>0}$, since otherwise we would have $\boldalpha_i+\boldbeta_i=0$ for some large enough $i$. This implies that either the summand $s_{\lambda}(\boldalpha)$ or the summand $s_{\lambda'}(\boldbeta)$ of $S_\lambda(\boldalpha||\boldbeta)$ must lie in $\RR_{>0}$. Therefore $S_\lambda(\boldalpha||\boldbeta)\in \RR_{>0}$ as claimed.

It remains to compute the valuation. We use  the Goulden-Greene-Macdonald tableau formula \eqref{e:GGMformula}. Consider first the SSYT $T_0$ of shape $\lambda$ defined by $T_0((i,j))=i$. We have $T^*_0((i,j))=j$ and $S_\lambda(\boldalpha||\boldbeta)$ has the summand
\begin{equation*}
(\boldalpha||\boldbeta)_{T_0}=\prod_{(i,j)\in\lambda}(\boldalpha_i+\boldbeta_j).
\end{equation*}
The valuation of this summand is given by $\Val\left((\boldalpha||\boldbeta)_{T_0}\right)=\sum_{(i,j)\in\lambda}\min(A_i,B_j)$. 

Now consider the $n$-variable approximation of $S_\lambda(\boldalpha||\boldbeta)$, e.g. setting $\boldalpha_p=\boldbeta_p=0$ for $p>n$. Then
\begin{equation}\label{e:finiteSuperSymmetric}
S_\lambda(\boldalpha_{[n]}||\boldbeta_{[n]}) =\sum_{T} (\boldalpha_{[n]}||\boldbeta_{[n]})_T,
\end{equation}
where the sum can be taken over the finite set of tableaux $T$ for which $\min(T(i,j),T^*(i,j))\le n$, as the summand $(\boldalpha_{[n]}||\boldbeta_{[n]})_T$ otherwise automatically vanishes. Let us choose $n$ to be large enough so that $T_0$ appears in the sum and $(\boldalpha_{[n]}||\boldbeta_{[n]})_{T_0}=(\boldalpha||\boldbeta)_{T_0}$. 
 
If $T$ is an SSYT appearing in \eqref{e:finiteSuperSymmetric} different from $T_0$. Then $T(i,j)\ge i$ by the column-strict property and with that, $T^*(i,j)\ge j$. Therefore, $\boldalpha_{T(i,j)}\le \boldalpha_{i}$ and $\boldbeta_{T^*(i,j)}\le\boldbeta_j$, which implies that $A_{T(i,j)}\ge A_i$ and $B_{T^*(i,j)}\ge B_j$. Moreover, replacing $\boldalpha$ by $\boldalpha_{[n]}$ and $\boldbeta$ by $\boldbeta_{[n]}$ can only increase the valuations further.  As a consequence
\begin{equation*}
\Val\left((\boldalpha_{[n]}||\boldbeta_{[n]})_{T}\right)\ge \sum_{(i,j)\in\lambda}\min(A_{T(i,j)},B_{T^*(i,j)})\ge \sum_{(i,j)\in\lambda}\min(A_i,B_j)=\Val\left((\boldalpha||\boldbeta)_{T_0}\right).
\end{equation*} 
 We therefore find that 
\[\Val(S_\lambda(\boldalpha_{[n]}||\boldbeta_{[n]}))=\Val((\boldalpha||\boldbeta)_{T_0})=\sum_{(i,j)\in\lambda}\min(A_i,B_j).
\]
Finally, letting $n\to\infty$ we find  
$\Val(S_\lambda(\boldalpha||\boldbeta))=\sum_{(i,j)\in\lambda}\min(A_i,B_j)$, by continuity of the valuation.
\end{proof}

Now choose $(\mathcal C,\Kdelposst,\Val)$ for  $(\RR,\RR_{>0},\Val)$, see \cref{s:Kdel}, in particular \cref{r:Kdeltopposstr}. 
 We introduce the following practically convenient parameter space.

\begin{defn}\label{d:Omegacirc} Recall the strong topology defined  in \cref{d:Kdeltop}.
\[
\Omega^\circ(\Kdelposst):=\left\{(\boldalpha,\boldbeta)\in\Kdelz^\N\times \Kdelz^\N\left |\begin{array}{l} \boldalpha_1\ge\boldalpha_2\ge\boldalpha_3\ge\dotsc, \ \text{ and }\ 
\boldbeta_1\ge\boldbeta_2\ge\boldbeta_3\ge \dotsc, \\ \boldalpha_i+\boldbeta_i\ne 0 \text{ for all $i\in\N$}\\
\text{for every $k$, the series $\sum_{i=k}^\infty \boldalpha_i$ is either $=0$ or  converges in $\Kdelposst$, }\\
\text{and $\sum_{i=k}^\infty \boldbeta_i$, is either $=0$ or converges in $\Kdelposst$}.
\end{array}\right.\right\}.
\]
\end{defn}

\begin{lemma}\label{l:slambdaconv}
Suppose $\boldalpha=(\boldalpha_1,\boldalpha_2,\dots )$ is a weakly monotonely decreasing sequence in $\Kdelposst$ such that $\sum_{i=k}^n\boldalpha_i$ converges in $\Kdelposst$ for all $k\in\N$ (as $n\to\infty$). Let $\lambda=(\lambda_1,\dotsc,\lambda_r)$ be a partition. Then $s_\lambda(\boldalpha_1,\dotsc,\boldalpha_n)=\sum_{T\in SSYT_{[n]}(\lambda)}\ \prod_{b\in\lambda}\boldalpha_{T(b)}$ converges 
in $\Kdelposst$ as $n\to\infty$.
\end{lemma}

\begin{proof} 
 By our assumption $(\sum_{i=k}^{n}\boldalpha_i)_n$ converges in $\Kdelposst$ (for any fixed $k$). Let us write $\summ_{\ge k}(\boldalpha)$ for the limit. It follows that any power $((\sum_{i=k}^{n}\boldalpha_i)^\ell)_n$ also converges, namely to $\summ_{\ge k}(\boldalpha)^\ell$. Let $YT_{[\text{row},n]}(\lambda)$ denote all of the tableaux of shape $\lambda$ with entries in row $k$ lying in $[k,n]$, but otherwise no constraint. Then
\[
\sum_{T\in YT_{[\text{row},n]}(\lambda)}\left(\prod_{b\in\lambda}\boldalpha_{T(b)}\right)
=(\boldalpha_1+\dotsc+\boldalpha_n)^{\lambda_1}(\boldalpha_2+\dotsc+\boldalpha_n)^{\lambda_2}\dotsc (\boldalpha_k+\dotsc+\boldalpha_n)^{\lambda_r}.
\]
We therefore have that 
\[
s_{\lambda}(\boldalpha_1,\dotsc,\boldalpha_n)\le (\boldalpha_1+\dotsc+\boldalpha_n)^{\lambda_1}(\boldalpha_2+\dotsc+\boldalpha_n)^{\lambda_2}\dotsc (\boldalpha_r+\dotsc+\boldalpha_n)^{\lambda_r}.
\]
We note that both left and right-hand sides of this inequality have the same valuation, namely $A_\lambda=\sum \lambda_i A_i$ if $A_i=\Val(\boldalpha_i)$. The inequality also implies \textit{pointwise} convergence of $s_{\lambda}(\boldalpha_1,\dotsc,\boldalpha_n)$ on $(0,\delta]$ to a limit function that we will call $s_\lambda(\boldalpha)$. Let us now prove that $s_\lambda(\boldalpha)$ lies in $\Kdelpos$ and the convergence is in the strong topology.

We multiply our inequality from above by $t^{-A_\lambda}$, giving
\begin{equation}\label{e:multipliedineq}
t^{-A_\lambda}s_{\lambda}(\boldalpha_1,\dotsc,\boldalpha_n)\le t^{-A_\lambda}(\boldalpha_1+\dotsc+\boldalpha_n)^{\lambda_1}(\boldalpha_2+\dotsc+\boldalpha_n)^{\lambda_2}\dotsc (\boldalpha_r+\dotsc+\boldalpha_n)^{\lambda_r}.
\end{equation}
The right-hand side can be written as a product of factors that have valuation $0$ and individually converge uniformly,
\[
t^{-\lambda_kA_k}(\boldalpha_k+\dotsc+\boldalpha_n)^{\lambda_k}\, \longrightarrow\, t^{-\lambda_k A_k}\summ_{\ge k}(\boldalpha)^{\lambda_k} \quad (n\to\infty),
\]
as follows from the convergence of $(\boldalpha_k+\dotsc+\boldalpha_n)^{\lambda_k}$ to $\summ_{\ge k}(\boldalpha)^{\lambda_k}$ and the definition of  $\Kdelposst$. Note that since both sides of \eqref{e:multipliedineq} have valuation $0$, they extend to continuous functions on $[0,\delta]$ with positive value at $0$. Let us write $f_n(t)\in C^0([0,\delta])$ for the function extending $t^{-A_\lambda}s_{\lambda}(\boldalpha_1,\dotsc,\boldalpha_n)$.
 We can now estimate $\|f_n-f_m\|_{\sup}$ for $m<n$ using \eqref{e:multipliedineq},
\begin{multline*}
\|t^{-A_\lambda}s_\lambda(\boldalpha_1,\dotsc,\boldalpha_n)-t^{-A_\lambda}s_{\lambda}(\boldalpha_1,\dotsc,\boldalpha_m)\|_{\sup}=\left\| t^{-A_\lambda}\sum_{SSYT_{[n]}\setminus SSYT_{[m]}}\left( \prod_{b\in\lambda}\boldalpha_{T(b)}\right)\right\|_{\sup}\le
\\
\left \|t^{-A_\lambda}\sum_{YT_{[\text{row},n]}\setminus YT_{[\text{row},m]}}\left( \prod_{b\in\lambda}\boldalpha_{T(b)}\right)\right\|_{\sup}= 
 \left\|t^{-A_\lambda}\prod_{k=1}^r(\boldalpha_k+\dotsc+\boldalpha_n)^{\lambda_k}-t^{-A_\lambda}\prod_{k=1}^r(\boldalpha_k+\dotsc+\boldalpha_m)^{\lambda_k}\right\|_{\sup}.
\end{multline*}
From the uniform convergence of the right-hand side of \eqref{e:multipliedineq} and the above inequality it follows that $\|f_n-f_m\|_{\sup}<\ep$ for all $n,m>N$, for some  $N=N(\varepsilon)$. Thus $f_n$ is a Cauchy-sequence in $C^0([0,\delta])$ (considered as a Banach algebra via the $\sup$-norm), and so  $f_n$ converges uniformly to a continuous function $f(t)$ on $[0,\delta]$. Moreover, since $f_n$ is positive and the $f_n$ are weakly monotonely increasing in $n$ (for every $t>0$ and thus also for $t=0$), we have that $f$ is positive, and in particular $f(0)\in\R_{>0}$. It follows that $f$ determines an element of $\Kdelpos$ with valuation equal to $0$. 

Finally, pointwise for $t>0$, we have that $t^{A_\lambda}f(t)$ is the limit of the $s_\lambda(\boldalpha_1,\dotsc,\boldalpha_n)$. Therefore $s_\lambda(\boldalpha)=t^{A_\lambda}f(t)$ and lies in $\Kdelpos$ with valuation $A_\lambda$. It follows from \cref{l:Kcontcrit} that $s_\lambda(\boldalpha)$ is the limit of the $s_\lambda(\boldalpha_1,\dotsc,\boldalpha_n)$ in terms of the strong topology of $\Kdelpos$.
 \end{proof}

\begin{cor}\label{c:EvalSSSchur} Let $(\boldalpha,\boldbeta)\in\Omega^\circ(\Kdelposst)$ and $\lambda$ any partition. Then the supersymmetric Schur function $S_\lambda$ converges to a well-defined element  $S_\lambda(\boldalpha||\boldbeta)$ in $\Kdelposst$ and we have
\begin{equation}\Val(S_\lambda(\boldalpha||\boldbeta))=\sum_{(i,j)\in\lambda}\min(A_i,B_j).\end{equation}
\end{cor}
\begin{proof}
This corollary follows from \cref{l:slambdaconv} using the general result  \cref{l:EvalSSSchur}.
\end{proof}

\begin{defn}\label{d:SlAB} Recall the set $\Omega(\Rmin)$ introduced in \cref{d:Omegatrop}. For $(\mathbf A,\mathbf B)\in\Omega(\Rmin)$ with $\mathbf A=(A_1,A_2,\dotsc)$ and $\mathbf B=(B_1,B_2,\dotsc)$ we define a function
\[
\mathbf S_\lambda:\Omega(\Rmin)\to\R
\]
by setting  $\mathbf S_\lambda(\mathbf A,\mathbf B)=\sum_{(i,j)\in\lambda}\min(A_i,B_j)$.
\end{defn}
\begin{remark}
\cref{c:EvalSSSchur} now says that  $\mathbf S_\lambda(\mathbf A,\mathbf B)=\Val(S_\lambda(\boldalpha||\boldbeta))$ for any $(\boldalpha,\boldbeta)\in\Omega^\circ(\Kdelposst)$ with $A_i=\Val(\boldalpha_i)$ and $B_i=\Val(\boldbeta_i)$.
\end{remark}

We are now ready to prove the key proposition that relates the tropical Schoenberg parameters from \cref{t:TropEdrei} to the Schoenberg parameters in the classical Edrei theorem. 
\begin{prop}\label{p:KdelEdrei}
We have a well-defined map $
\mathcal T^\circ:\Omega^\circ(\Kdelposst)\to\ToeplitzU_{\infty}(\Kdelpos)$ sending $(\boldalpha,\boldbeta)$ to the infinite Toeplitz matrix 
 $\mathcal T^\circ(\boldalpha,\boldbeta)$ with entry $\mathbf c_k$ along the $k$-th diagonal, where  
\begin{equation}\label{e:Tpower}
\sum_{k=0}^\infty \mathbf c_k x^k=\frac{\prod_{i=1}^\infty (1+\boldbeta_j x)}{\prod_{j=1}^\infty (1-\boldalpha_i x)}.
\end{equation}
The standard coordinates $(m_{ij})$ of the Toeplitz matrix $\mathcal T^\circ(\boldalpha,\boldbeta)$ are computed by
\begin{equation}\label{e:mijalphabeta1}
m_{ij}=\frac{S_{i\times j}(\boldalpha||\boldbeta) S_{(i-1)\times (j-1)}(\boldalpha||\boldbeta)}{S_{(i-1)\times j}(\boldalpha||\boldbeta) S_{i\times{(j-1)}}(\boldalpha||\boldbeta)},
\end{equation}
where $s\times r$ denotes the partition $r^s=(r,\dots, r)$ with $s$ parts. Moreover, the valuations $M_{ij}=\Val(m_{ij})$ are 
\begin{equation}\label{e:MijAB1}
M_{ij}=\min(A_i,B_j),
\end{equation}
for $A_i=\Val(\boldalpha_i)$ and $B_j=\Val(\boldbeta_j)$.
\end{prop}
\begin{proof}

We start by showing that the map $\mathcal T^\circ:\Omega^\circ(\Kdelposst)\to \ToeplitzU_\infty(\Kdelpos)$ is well-defined. The first observation is that the coefficients $\mathbf c_k$ are well-defined elements of $\Kdelpos$ by the application of \cref{c:EvalSSSchur}. Namely $\bold c_k=H_k(\boldalpha||\boldbeta)$, compare in \eqref{e:generatingH}, which is the evaluation of the supersymmetric Schur function associated to the partition $\lambda=(k)$. To show that the associated infinite upper-triangular Toeplitz matrix is totally positive over $\Kdelpos$ it suffices to show that it has standard coordinates $m_{ij}$ in $\Kdelpos$. Recall that the standard coordinates are computed in terms of minors  in \cref{l:mijFormula}.  For $\bold c_k=H_k(\boldalpha||\boldbeta)$ these minors agree with supersymmetric Schur functions evaluated on $(\boldalpha,\boldbeta)$. For example, we have 
\[
\Minor^{[i]}_{[i]+j}(\bold c_{r-s})=S_{i\times j}(\boldalpha||\boldbeta).
\]
The formula from \cref{l:mijFormula} straightforwardly translates to \eqref{e:mijalphabeta1}. The standard coordinates now lie in $\mathcal C_{>0}$ by \cref{c:EvalSSSchur}. Moreover tropicalising \eqref{e:mijalphabeta1}  using the second half of \cref{c:EvalSSSchur}  gives 
\begin{equation}\label{e:valmij1}
\Val(m_{ij})=\mathbf S_{i\times j}(\mathbf A,\mathbf B)+\mathbf S_{{(i-1)}\times(j-1)}(\mathbf A,\mathbf B)-\mathbf S_{i\times(j-1)}(\mathbf A,\mathbf B)-\mathbf S_{{(i-1)}\times j}(\mathbf A,\mathbf B),
\end{equation}
where $\mathbf S_{\lambda}(\bigA,\bigB)$ is as in \cref{d:SlAB}, and explicitly $\mathbf S_{i\times j}(\mathbf A,\mathbf B)=\sum_{r\le i}\sum_{s\le j} \min(A_r,B_s)$. Now all of the terms $\min(A_r,B_s)$ in \eqref{e:valmij1} cancel apart from one: $\min(A_i,B_j)$. This proves the formula~\eqref{e:MijAB1}.
\end{proof}

\subsection{The proof of \cref{t:KdelTrop}}
We  now have some tools for constructing elements in $\ToeplitzU_\infty(\Kdelpos)$ and are able to compute the valuations of their standard coordinates. We use this together with the `tropical Edrei theorem', \cref{t:TropEdrei}, to prove \cref{t:KdelTrop}. 
 
\begin{lemma}\label{l:detropK} Let $(A_i)_{i=1}^\infty$ be a weakly increasing sequence of real numbers, $A_1\le A_2\le\dotsc $ with $\sup(\{A_i\})$ either equal to $A\in\R$ or $\infty$. Let $\boldalpha_i=\frac{1}{2^i}t^{A_i}\in\Kdelpos$. Then $\Val(\boldalpha_i)=A_i$ and for every $k\in\Z_{>0}$ the infinite series $S_{\ge k}$ defined by
\[
S_{\ge k}(t):=\sum_{i=k}^{\infty}\boldalpha_i(t)
\]
converges in $\Kdelposst$. 
\end{lemma}

\begin{proof}
It is clear that $\Val(\boldalpha_i)=A_i$. It remains to prove the series $S_{\ge k}(t)$ converges in $\Kdelposst$ for every $k$.  We claim that the related series $t^{-A_k}S_{\ge k}(t)$ converges in $\Odelpos$. Let us consider $k$ fixed and drop it from our notation, and consider the partial sums of this series
\[
f_n(t):=t^{-A_k}\sum_{i=k}^{n}\frac{1}{2^{i}}t^{A_i}=\sum_{i=k}^{n}\frac{1}{2^{i}}t^{A_i-A_k}.
\]
We may consider $f_n$ as an element of $C^0([0,\delta])$ since the valuation is now $0$. Moreover, the sequence $(f_n)$ converges for the $\sup$-norm, for example since for $i> k$ we have $\frac{1}{2^{i}}t^{A_i-A_k}<\frac{1}{2^i}$ for all $t\in [0,\delta]$ (since $\delta<1$). Therefore we have a continuous function $f:[0,\delta]\to\R$ in the limit, 
\[
f(t)=t^{-A_k}\sum_{i=k}^{\infty}\frac{1}{2^{i}}t^{A_i}.
\]
It also has positive values everywhere, including at $t=0$; namely $f(0)=\frac 1{2^k}$. Thus $f(t)\in\Odelpos$ with valuation $0$. It follows that $S_{\ge k}=t^{A_k} f$ lies in $\Kdelpos$ with valuation $A_k$ and that the series $S_{\ge k}$ converges in~$\Kdelposst$ using  \cref{l:Kcontcrit}. 
\end{proof}
Recall the parameter spaces from \cref{d:Omegacirc} and \cref{d:Omegatrop}.
\begin{cor}\label{c:detropOmega} The map $\Val:\Omega^\circ(\Kdelposst)\to\Omega(\Rmin)$ given by coordinate-wise valuation is surjective. \end{cor}

\begin{proof}
Let $(\mathbf A,\mathbf B)\in\Omega(\Rmin)$. We define
\begin{equation}\label{e:paramlifting}
\begin{array}{lcr}
\boldalpha_i=\begin{cases} \frac{1}{2^i}t^{A_i} &\text{ if $A_i\ne \infty$,}\\
0 & \text{ if $A_i=\infty$,}
\end{cases}&\quad\text{and}\quad&
\boldbeta_i=\begin{cases} \frac{1}{2^i}t^{B_i} &\text{ if $B_i\ne \infty$,}\\
0 & \text{ if $B_i=\infty$.}
\end{cases}
\end{array}
\end{equation}
By the definition of $\Omega(\Rmin)$ we have that $\min(A_i,B_j)\in\R$ for all $i,j\in\N$. This implies that $\boldalpha_i+\boldbeta_j\in\Kdelpos$. The inequalities $\boldalpha_i\ge\boldalpha_{i+1}$ and $\boldbeta_i\ge\boldbeta_{i+1}$ are straightforward. Moreover the sums $\sum_{i=k}^\infty\boldalpha_i$ and $\sum_{i=k}^{\infty}\boldbeta_i$ converge in $\Kdelpos$ whenever there are infinitely many nonzero terms by \cref{l:detropK}. Otherwise they may be $0$, if all summands are zero. In any case we have that $\sum_{i=k}^\infty\boldalpha_i\in\Kdelz$, and  $\sum_{i=k}^\infty\boldbeta_i\in\Kdelz$ and convergence is in the strong topology. All of the conditions for $(\boldalpha,\boldbeta)$ to lie in $\Omega^\circ(\Kdelposst)$ are met. By construction, $\Val(\boldalpha,\boldbeta)=(\bigA,\bigB)$. Thus $(\bigA,\bigB)$ lies in the image of $\Val:\Omega^\circ(\Kdelposst)\to\Omega(\Rmin)$.     
\end{proof}

We can now prove that $\ToeplitzU_{\infty,\Kdelpos}(\Rmin)$ agrees with $\Imin_{\infty}(\Rmin)$ in terms of its tropical standard coordinates. 

\begin{proof}[Proof of \cref{t:KdelTrop}]
The standard coordinates $(M_{ij})_{i,j\in\N}$ of any element of $\ToeplitzU_{\infty,\Kdelpos}(\Rmin)$ lie in $\Imin_{\infty}(\Rmin)$ by \cref{l:infminideal}. Recall that the elements of $\ToeplitzU_{\infty,\Kdelpos}(\Rmin)$ are equivalence classes in $\ToeplitzU_{\infty}(\Kdelpos)$, by definition. It remains to prove that every infinite min-ideal filling $(M_{ij})_{i,j\in\N}$ can be represented by an element of $\ToeplitzU_{\infty}(\Kdelpos)$. Now \cref{t:TropEdrei} implies that $(M_{ij})_{i,j}=(\min(A_i,B_j))_{i,j}$ for some element $(\bigA,\bigB)\in\Omega(\Rmin)$ (which can be chosen weakly interlacing, but we don't need this property for now). 
By \cref{c:detropOmega} we can find $(\boldalpha,\boldbeta)\in\Omega^\circ(\Kdelposst)$ such that $\Val(\boldalpha_i)=A_i$ and $\Val(\boldbeta_i)=B_i$ for all $i\in\N$. Applying \cref{p:KdelEdrei} we obtain an infinite Toeplitz matrix $\mathcal T^\circ(\boldalpha,\boldbeta)\in\ToeplitzU_\infty(\Kdelpos)$ whose associated tropical point has standard coordinates $(\min(A_i,B_j))_{i,j}$. Thus   $\mathcal T^\circ(\boldalpha,\boldbeta)$ represents our original min-ideal filling $(M_{ij})_{i,j\in\N}$. This proves \cref{t:KdelTrop}.
\end{proof}

\section{The tropical Edrei theorem detropicalised}\label{s:restrictedT}
Let us consider a ring with a positive structure $(\RR,\RR_{>0})$, and let us assume we have a topology on $\RR_{>0}$ that extends to $\RR_{\ge 0}$. In this case we may alongside $\Toeplitz_{\infty}(\RR_{>0})$ consider the restricted version $\Toeplitz^{res}_{\infty}(\RR_{>0})$ of infinite totally positive Toeplitz matrices, compare \cref{d:ResToepl}. We now define more generally some associated parameter spaces. 

\begin{defn}\label{d:Omegas}
Write $\boldalpha=(\boldalpha_1,\boldalpha_2,\dotsc)$ and $\boldbeta=(\boldbeta_1,\boldbeta_2,\dotsc)$ for two infinite sequences in $\RR_{\ge 0}=\RR_{>0}\cup\{0\}$. 
Our general parameter space is
\begin{equation}
\Omega(\RR_{>0}):=\left\{(\boldalpha,\boldbeta)\in\RR_{\ge 0}^\N\times \RR_{\ge 0}^\N\left |\begin{array}{l} \boldalpha_1\ge\boldalpha_2\ge\boldalpha_3\ge\dotsc, \\
\boldbeta_1\ge\boldbeta_2\ge\boldbeta_3\ge \dotsc, \\
\boldbeta_i+\boldalpha_i\ne 0 \text{ for all $i\in\N,$}\\
  \text{$S_{i\times j}(\boldalpha||\boldbeta)$ is well-defined (converges in $\RR_{>0}$) for all $i,j\in\N$}
\end{array}
\right.\right\}.
\end{equation}
 Let us set 
\begin{equation}\label{e:mij(a,b)}
\bold m_{ij}(\boldalpha,\boldbeta):=\frac{S_{i\times j}(\boldalpha||\boldbeta)S_{{(i-1)}\times{(j-1)}}(\boldalpha||\boldbeta)}{S_{i\times{(j-1)}}(\boldalpha||\boldbeta)S_{{(i-1)}\times j}(\boldalpha||\boldbeta)}
\end{equation}
for any $(\boldalpha,\boldbeta)\in\Omega(\RR_{>0})$. We may then define a restricted version of $\Omega(\mathcal R_{>0})$, 
\begin{equation*}
\begin{array}{lll}
\Omega^{res}(\RR_{>0})&:=&\left\{(\boldalpha,\boldbeta)\in\Omega(\RR_{>0}) \mid \exists
   \text{ $\lim_{k\to\infty}\mathbf m_{ik}(\boldalpha,\boldbeta)$ and $\lim_{k\to\infty}\mathbf m_{kj}(\boldalpha,\boldbeta)$  in $\RR_{\ge 0}$,  for all $k\in\N$}\right\}.
      \end{array}
\end{equation*}
\end{defn}
Recall now the tropical parameter spaces from \cref{d:Omegatrop} and their related versions of infinite min-ideal fillings, see \cref{t:TropEdrei}. In this general situation we have the following result. 
\begin{theorem}\label{t:restrtropparam}
Let $(\RR,\RR_{\ge 0},\Val)$ be a ring with a topological valuative positive structure.
Then we have the following commutative diagram.
\begin{equation}\label{e:OmValGenR}
\begin{tikzcd}
	{(\boldalpha,\boldbeta)} & {\Omega(\RR_{>0})} && {\ToeplitzU_\infty(\RR_{>0})} & {(m_{ij})_{i,j}} \\
	{(\bigA,\bigB)} & {\Omega(\Rmin)} && {\Imin_\infty}(\Rmin) & {(M_{ij})_{i,j}}.
	\arrow[maps to, from=1-1, to=2-1]
	\arrow["\mathcal T", from=1-2, to=1-4]
	\arrow["\Val"', from=1-2, to=2-2]
	\arrow["\Val"',from=1-4, to=2-4]
	\arrow[maps to, from=1-5, to=2-5]
	\arrow["\widetilde{\mathbb E}", from=2-2, to=2-4]
\end{tikzcd}
\end{equation}
Here the map $\mathcal T$ sends $(\boldalpha,\boldbeta)$ to the totally positive Toeplitz matrix with standard coordinates given by \eqref{e:mij(a,b)},  and the map $\widetilde{\mathbb E}$ takes $(\bigA,\bigB)$ to $(\min(A_i,B_j))_{i,j\in\N}$. Moreover, $\mathcal T$ restricts to a map $\mathcal T^{res}:{\Omega^{res}(\mathcal R_{>0})}\longrightarrow {\ToeplitzU^{res}_\infty(\mathcal R_{>0})}$.  
\end{theorem}
\begin{remark}
The image $\mathcal T(\boldalpha,\boldbeta)$ in the above theorem is equivalently equal to the infinite Toeplitz matrix with entries $\mathbf c_k=H_{k}(\boldalpha||\boldbeta)$, that is, with generating function 
\begin{equation*}
\sum_{k=0}^\infty \mathbf c_k x^k=\frac{\prod_{i=1}^\infty (1+\boldbeta_j x)}{\prod_{j=1}^\infty (1-\boldalpha_i x)}.
\end{equation*}
\end{remark}
\begin{proof}[Proof of \cref{t:restrtropparam}]
Let $(\boldalpha,\boldbeta)\in\Omega(\RR_{> 0})$.  We have that the coordinate-wise valuations $\bigA$ and $\bigB$ of $\boldalpha$ and $\boldbeta$ are weakly increasing sequences in $\Rmininf$, by \cref{l:Kdelposet}. Moreover, since $\boldalpha_i+\boldbeta_i\ne 0$ we also have $\min(A_i,B_i)\in\R$.  Thus $(\bigA,\bigB)\in\Omega(\Rmin)$. We know that the valuations of the standard coordinates determine an infinite min-ideal filling in this generality by \cref{l:infminideal}. Now suppose $(m_{ij})_{i,j}=\mathcal T(\boldalpha,\boldbeta)$, that is,  $m_{ij}=\mathbf m_{ij}(\boldalpha,\boldbeta)$ as in  \eqref{e:mij(a,b)}. Then, indeed, $M_{ij}=\min(A_i,B_j)$ and the diagram commutes, as follows from \cref{l:EvalSSSchur}. The fact that $
\mathcal T^{res}:{\Omega^{res}(\mathcal R_{>0})} \longrightarrow {\ToeplitzU^{res}_\infty(\mathcal R_{>0})}$ is well-defined now follows from the definition of $\Omega^{res}(\mathcal R_{>0})$.
\end{proof}

Note that we do not necessarily have surjectivity of the vertical maps, this would depend on the choice of $\mathcal R_{>0}$. Also, $\widetilde{\mathbb E}$ is not one of the parametrisation maps from \cref{t:TropEdrei}; it is surjective but not injective.

We now recall our standard examples: the semiring $\Kdelnn$ with its two topologies from \cref{d:topCnn} and the semiring $\KK_{\ge 0}=\KK_{>0}\cup\{0\}$ of generalised Puiseaux series with its $t$-adic topology. 

\begin{remark}\label{r:Valcont}
All of these three topological semirings have the property that $\Val:\RR_{\ge 0}\to\Rmininf$ is continuous for the  topology on $\Rmininf$ generated by the  Euclidean topology on $\R$ with added neighbourhoods $(M,\infty]$ at $\infty$. In the case of $\Kdelnnst$ and $\KK_{\ge 0}$, the valuation map is even continuous for  $\Rmininf$ with the discrete topology. However, the price for the continuity of the valuation of $\Kdelnn$ at $0$ is that the inclusion of $\R_{\ge 0}$ into $\Kdelnn$ as constant functions cannot be continuous, even if the inclusion of $\R_{>0}$ is, see \cref{l:Rincl} and \cref{l:Rinclwk}. The constant functions $f_k= \frac{1}k$ can never converge to the $0$-function if the valuation map is continuous. Note that for $\KK_{\ge 0}$ even $\R_{>0}$ does not embed continuously. Rather, the induced topology on $\R_{>0}$ is the discrete topology. 
\end{remark}

\subsection{Generalised Puiseaux series}\label{s:GenPuiseauxInf} For $\mathcal R_{>0}=\KK_{>0}$, let 
\begin{equation}\label{e:Omegaconv}
\Omega^{(0)}(\KK_{>0})=\left\{
(\boldalpha,\boldbeta)\in\KK_{\ge 0}^\N\times \KK_{\ge 0}^\N\left |\begin{array}{l} \boldalpha_1\ge\boldalpha_2\ge\boldalpha_3\ge\dotsc, \\
\boldbeta_1\ge\boldbeta_2\ge\boldbeta_3\ge \dotsc, \\
\boldbeta_i+\boldalpha_i\ne 0 \text{ for all $i\in\N,$}\\
 \lim_{i\to\infty}(\boldalpha_i)=\lim_{i\to\infty}(\boldbeta_i)=0
 \end{array}
 \right.\right\}.
\end{equation}
With this parameter space we have the following proposition. 
\begin{prop}\label{p:OmegaDivInterp}
We have that $\Omega^{(0)}(\KK_{>0})$ from \eqref{e:Omegaconv} is equal to $\Omega(\KK_{>0})$ from  \cref{d:Omegas}. Define  $\ToeplitzU^{(0)}_\infty(\KK_{>0}):=\mathcal T(\Omega(\KK_{>0}))$ where $\mathcal T$ is the map from \cref{t:restrtropparam}. We  then have the following commutative diagram,
\begin{equation}\label{e:restrOmValPuis}
\begin{tikzcd}
	{\Omega(\KK_{>0})} && {\ToeplitzU^{(0)}_\infty(\KK_{>0})} &&  {\ToeplitzU_\infty(\KK_{>0})} \\
	{\Omega_\star^{div}(\Rmin)} && {\Imin^{div}_\infty}(\Rmin) && \Val(\ToeplitzU_\infty(\KK_{>0})).
	\arrow[two heads,from=1-1, to=1-3]
		\arrow[hook,from=1-3, to=1-5]
		\arrow[hook,from=2-3, to=2-5]
	\arrow["\Val"', two heads, from=1-1, to=2-1]
	\arrow["\Val"', two heads, from=1-3, to=2-3]
		\arrow["\Val"', two heads, from=1-5, to=2-5]
	\arrow["{\mathbb E^{div}}", from=2-1, to=2-3]
\end{tikzcd}
\end{equation}
where the map $\mathbb E^{div}$ is the bijection from \cref{t:TropEdrei}, and the inclusions on the right-hand side are both proper inclusions (not surjective).
\end{prop}
\begin{proof}
In the $t$-adic topology, convergence to $0$ is equivalent to divergence to $\infty$ of the valuations.  Therefore the vertical map $\Val$ on the left-hand side is well-defined and surjective. It is clear that if a sequence $\boldalpha=(\boldalpha_i)_i$ converges to $0$ in this topology, then the Schur polynomial values $s_\lambda(\boldalpha_1,\dotsc,\boldalpha_n)$ also converge as $n\to\infty$, since any power of $t$ will occur in at most finitely many summands (using the expansion indexed by SSYT). It follows that $\Omega^{(0)}(\KK_{>0})$ lies in $\Omega(\KK_{>0})$. 
Conversely, if $(\boldalpha,\boldbeta)\in\Omega(\KK_{>0})$ then we have that $S_{(1)}(\boldalpha ||\boldbeta)=\sum_{i=1}^\infty(\boldalpha_i+\boldbeta_i)$ converges in $\KK_{>0}$. Therefore the partial sums are a Cauchy sequence for the $t$-adic norm. Explicitly, we have that for every $\ep>0$ there is an $N_\ep\in\N$ such that 
\begin{equation}\label{e:KCauchy} \left\|\sum_{i=m}^{n}(\boldalpha_i+\boldbeta_i)\right\|_t<\ep,\qquad \forall n,m>N_\ep.
\end{equation} 
Let $A_i:=\Val(\boldalpha_i)$ and $B_i=\Val(\boldbeta_i)$.  Then the $A_i$ and the $B_i$ form weakly increasing sequences, and we denote by $A$ and $B$ their respective suprema. 
Using positivity of the $\boldalpha_i$ and $\boldbeta_i$ we have  
\begin{equation}\label{e:ValBound}
\Val\left(\sum_{i=m}^{n}(\boldalpha_i+\boldbeta_i)\right)=\min\left(\{\min(A_i,B_i)\mid m\le i\le n\}\right)\le \min(A,B).
\end{equation}
But letting $\ep$ go to $0$ in \eqref{e:KCauchy} we see that the valuations in \eqref{e:ValBound} cannot be bounded. Therefore $A=B=\infty$ and we have that 
$\lim_{i\to\infty}(\boldalpha_i)=\lim_{i\to\infty}(\boldbeta_i)=0$. This concludes the proof that $\Omega^{(0)}(\KK_{>0})=\Omega(\KK_{>0})$.

We can now consider the commutative diagram \eqref{e:OmValGenR} for $\RR_{>0}=\KK_{>0}$. We  factorise the map $\mathcal T$ to  $\ToeplitzU_\infty(\KK_{>0})$ through its image, $\ToeplitzU^{(0)}_\infty(\KK_{>0})$ to obtain the top line of \eqref{e:restrOmValPuis}. The image   $\Val(\Omega^{(0)}(\KK_{>0}))$ of the left-hand $\Val$ is clearly $\Omega_\star^{div}(\Rmin)$, and the left-hand side commutative diagram follows from \cref{t:TropEdrei}. The rest of the diagram is also clear. It remains to show that the horizontal maps on the right-hand side are not surjective. 

Any real totally positive Toeplitz matrix $u$ lies in $\ToeplitzU_\infty(\KK_{>0})$, but its valuation will be the $0$ min-ideal filling $\Val(u)=(M_{ij})_{i,j}=(0)_{i,j}$, which does not lie in $\Imin^{div}_{\infty}(\Rmin)$. Thus also $u$ cannot lie in $\ToeplitzU^{(0)}_\infty(\KK_{>0})$. It follows that the two horizontal inclusions are indeed proper.  
\end{proof}

\subsection{The semifield of continuous functions $\Kdelpos$}
In this section we show that $\mathcal T^{res}$ over $\Kdelposwk$ and $\Kdelposst$ detropicalises the parametrisation maps $\mathbb E$ and $\mathbb E^s$ from \cref{t:TropEdrei}. Namely, we have the following theorem. 
\begin{theorem}
For $\RR_{\ge 0}=\Kdelnnwk$ or $\mathcal R_{\ge 0} =\Kdelnnst$ the diagram \eqref{e:OmValGenR} restricts in the following ways,
\begin{equation}\label{e:restrOmVal}
\begin{tikzcd}
	{\Omega^{res}(\Kdelposwk)} && {\ToeplitzU^{res}_\infty(\Kdelposwk)} \\
	{\Omega_\star(\Rmin)} && {\Imin_\infty}(\Rmin),
	\arrow["{\mathcal T^{res}}", from=1-1, to=1-3]
	\arrow["\Val",two heads, from=1-1, to=2-1]
	\arrow["\Val", two heads, from=1-3, to=2-3]
	\arrow["{\mathbb E}", from=2-1, to=2-3]
\end{tikzcd} \qquad \qquad
\begin{tikzcd}
	{\Omega^{res}(\Kdelposst)} && {\ToeplitzU^{res}_\infty(\Kdelposst)} \\
	{\Omega_\star^{il}(\Rmin)} && {\Imin^{s}_\infty}(\Rmin),
	\arrow["{\mathcal T^{res}}", from=1-1, to=1-3]
	\arrow["\Val",two heads, from=1-1, to=2-1]
	\arrow["\Val",two heads, from=1-3, to=2-3]
	\arrow["{\mathbb E^s}", from=2-1, to=2-3]
\end{tikzcd}
\end{equation}
where the restrictions $\mathbb E$ and $\mathbb E^s$ of $\widetilde{\mathbb E}$ are bijections by  \cref{t:TropEdrei} and the vertical maps are all surjective.
\end{theorem}
\begin{proof}
Suppose that $(\boldalpha,\boldbeta)\in\Omega^{res}(\Kdelposwk)$. Let $\bigA=(A_i)_i$ and $\bigB=(B_j)_j$ where $A_i=\Val(\boldalpha_i)$ and $B_j=\Val(\boldbeta_j)$. Note that $(\bigA,\bigB)\in\Omega(\Rmin)$ by \cref{t:restrtropparam}. We need to show this pair is weakly interlacing. By definition of $\Omega^{res}(\Kdelposwk)$, the limits $\lim_{k\to\infty}\mathbf m_{ik}(\boldalpha,\boldbeta)$ and $\lim_{k\to\infty}\mathbf m_{kj}(\boldalpha,\boldbeta)$ exist in $\Kdelnnwk$. Now \cref{p:Edreilimits} implies that these limits equal to $\boldalpha_i$ and $\boldbeta_j$, respectively.  
Since the valuation map is continuous we therefore also have that 
$\lim_{k\to\infty}M_{ik}=A_i$ and $\lim_{k\to\infty}M_{kj}=B_j$, where $M_{ik}=\Val(\mathbf m_{ij}(\boldalpha,\boldbeta))$. This now implies that $\bigA$ and $\bigB$ have the same supremum, namely equal to $\sup(\{M_{ij}\mid i,j\in\N\})$, see also the proof of \cref{t:TropEdrei}. 
Therefore $\bigA$ and $\bigB$ are indeed weakly interlacing, and the left-hand vertical map is well-defined. The entire left-hand diagram now makes sense, and it commutes by \cref{t:restrtropparam}. 

Next we consider the diagram involving the strong topology. Here we first show that the right-hand vertical map is well-defined. Suppose that $(m_{ij})_{i,j\in\N}$ are the standard coordinates of an element in $\ToeplitzU^{res}_\infty(\Kdelposst)$. By the `restricted' property we have that  $\lim_{k\to\infty}m_{ik}$ and $\lim_{k\to\infty}m_{kj}$ exist in the strong topology. This  means that the associated valuations $M_{ik}$ and the $M_{kj}$ stabilize as $k\to\infty$, thanks to \cref{l:valcont}. Thus $(M_{ij})_{i,j\in\N}$ does indeed lie in the set $\Imin^s_{\infty}(\Rmin)$ of stable min-ideal fillings. For the left-hand vertical map, note that $\Omega_\star(\Kdelposst)$ is a subset of $\Omega_\star(\Kdelposwk)$. Thus $\Val$ takes  $(\boldalpha,\boldbeta)$ to a weakly interlacing pair $(\bigA,\bigB)\in\Omega_\star(\Rmin)$ and we have that $\mathbb E(\bigA,\bigB)=(M_{ij})_{i,j}$. But since  $(M_{ij})_{i,j}\in\Imin^s_\infty(\Rmin)$ this implies that $\bigA$ and $\bigB$ are actually interlacing, as follows from \cref{t:TropEdrei}.

Next we show that the map 
$\Val:\Omega^{res}(\Kdelposwk)\to \Omega_\star(\Rmin)$
from the left-hand side of \eqref{e:restrOmVal} is surjective. Suppose we have $(\bigA,\bigB)\in\Omega_\star(\Rmin)$. Consider any lift to $\Omega(\Kdelposwk)$. For example, as in \eqref{e:paramlifting}, we may set $\boldalpha_i:=\frac{1}{2^i}t^{A_i}$ and $\boldbeta_j=\frac{1}{2^j}t^{B_j}$  for $A_i,B_j\in\R$ (and $0$ whenever $A_i$ or $B_j$ is $\infty$). We want to show that the this lifting has the `restricted' property, that is, the sequences $(\mathbf m_{ik}(\boldalpha,\boldbeta))_k$ and $(\mathbf m_{kj}(\boldalpha,\boldbeta))_k$ converge as $k\to\infty$ for the weak topology of $\Kdelnn$. Convergence in the weak topology is equivalent to pointwise convergence together with convergence of valuations. Pointwise convergence holds by \cref{p:Edreilimits}. Convergence of the valuations follows from the assumption that $\bigA$ and $\bigB$ were weakly interlacing. Namely, the valuations $M_{ij}=\min(A_i,B_j)$ of the $\mathbf m_{ik}(\boldalpha,\boldbeta)$ recover in the limit $i\to\infty$ precisely $B_j$ and for $j\to\infty$ the $A_i$, as in the proof of \cref{t:TropEdrei}.

Let us now prove that the map 
$\Val: \Omega^{res}(\Kdelposst)\to \Omega^{il}_\star(\Rmin)$
from the  right-hand diagram in \eqref{e:restrOmVal} is surjective. Suppose $(\bigA,\bigB)\in\Omega^{il}_\star(\Rmin)$. Then in particular $(\bigA,\bigB)\in\Omega^{\R}_\star(\Rmin)$ and the components of the sequences are all in $\R$, compare \cref{r:interlacingvsR}. Setting $\boldalpha_i:=\frac{1}{2^i}t^{A_i}$ and $\boldbeta_j=\frac{1}{2^j}t^{B_j}$ as in \cref{l:detropK} we obtain a lift of $(\bigA,\bigB)$ to $\Omega^\circ(\Kdelposst)$, see \cref{c:detropOmega}, and thus a lift to $\Omega(\Kdelposst)$ by \cref{c:EvalSSSchur}. We want to show that it has the `restricted' property in the strong topology, that is, the sequences $(\mathbf m_{ik}(\boldalpha,\boldbeta))_k$ and $(\mathbf m_{kj}(\boldalpha,\boldbeta))_k$ converge as $k\to\infty$ in $\Kdelnnst$. By \cref{l:Kcontcrit}, convergence of a sequence $f_n$ in the strong topology is equivalent to the valuations $F_n$ stabilising and then uniform convergence for the $t^{-F_n}f_n$ on $[0,\delta]$ to a strictly positive function. Let us now focus on the sequence $(m_{ik})_k$ given by $m_{ik}=\mathbf m_{ik}(\boldalpha,\boldbeta)$.  The valuations $M_{ik}=\min(A_i,B_k)$ stabilise to $A_i$ as $k\to\infty$, as follows from the assumption that $\bigA$ and $\bigB$ are  interlacing. Moreover, we have pointwise convergence of $(m_{ik})_k$ to $\boldalpha_i$, by \cref{p:Edreilimits}. Let us assume $k>i$ and consider the vertex label $v_{ik}$ from \cref{p:ToeplitzviaQ}.  It turns out to be useful to consider the related sequence $(v_{ik}\inv)_k$. Let us express $v_{ik}\inv$ in terms of the $m_{ij}$ as follows,
\begin{equation}\label{e:vikinv}
v_{ik}\inv=m_{ik}\left(1+\frac{m_{ik}}{m_{i-1,k-1}}+\frac{m_{ik}}{m_{i-2,k-2}}+\dotsc +\frac{m_{ik}}{m_{1,k-i+1}}\right)\inv.
\end{equation}
Note that the term in the bracket has $i$ summands, independently of $k$.
Also, since the valuations $M_{ij}$ of the $m_{ij}$ are weakly increasing in $i$ and $j$ we have that $\Val(v_{ik}\inv)=M_{ik}$, which means that the valuations of the $v_{ik}\inv$ stabilise to $A_i$ as $k\to\infty$.  We now consider the sequence $(t^{-A_i}v_{ik}\inv)_k$ and summarise its properties.
\begin{itemize}
\item For $k$ large enough $t^{-A_i}v_{ik}\inv$ has valuation $0$ and is therefore a continuous strictly positive function on $[0,\delta]$.
\item The sequence $(v_{ik})_k$ is strictly monotonely increasing in $k$ by \cref{p:ToeplitzviaQ} and therefore the sequence $(t^{-A_i}v_{ik}\inv)_k$ is strictly monotonely decreasing. 
\item We have pointwise convergence of $(t^{-A_i}v_{ik}\inv )_k$ on $(0,\delta]$ to 
\begin{equation}\label{e:elli}
\ell_i:=t^{-A_i}\boldalpha_{i}\left(1+\frac{\boldalpha_{i}}{\boldalpha_{i-1}}+\frac{\boldalpha_{i}}{\boldalpha_{i-2}}+\dotsc +\frac{\boldalpha_{i}}{\boldalpha_{1}}\right)\inv.
\end{equation}
\end{itemize}
It now follows that the sequence $(t^{-A_i}v_{ik}\inv)_k$ is uniformly convergent (by Dini's theorem), and its limit is a continuous function $[0,\delta]$ that is positive on $(0,\delta]$ by \cref{e:elli}. Moreover, since $\Val(\frac{\boldalpha_i}{\boldalpha_{i-s}})=A_i-A_{i-s}\ge 0$ we have that 
\begin{equation}\label{e:Uniformell}
\lim_{t\to 0}\frac{\boldalpha_i(t)}{\boldalpha_{i-s}(t)}=:c_s\in \R_{\ge 0}\qquad \text{ and }\qquad
\ell_i(0)=\lim_{t\to 0}\ell_i(t)=\lc(\boldalpha_i)(1+c_1+\dotsc+c_{i-1})\inv\in \R_{>0},
\end{equation}
so that $\ell_i$ is a strictly positive function on $[0,\delta]$. Therefore we may invert $\ell_i$ in $C^0([0,\delta])$, and we note that we also have uniform convergence 
\[
\lim_{k\to\infty}(t^{A_i}v_{ik})=\ell_i\inv\quad \text{ and }\quad  \lim_{k\to\infty}(t^{A_i}v_{i-1,k-1})=t^{A_{i}-A_{i-1}}\ell_{i-1}\inv \quad \text{in $C^0([0,\delta])$,}
\]
again keeping in mind $A_i\ge A_{i-1}$. Now from \eqref{e:vikinv} and \eqref{e:elli} it follows that 
\begin{eqnarray}\label{e:row1}
t^{A_i}v_{ik}(t) - t^{A_i}v_{i-1,k-1}(t) &=&
t^{A_i}m_{ik}(t)\inv,\\ \label{e:row2}
\ell_i(t)\inv - t^{A_i-A_{i-1}}\ell_{i-1}(t)\inv &=&
t^{A_i}\boldalpha_i(t)\inv.
\end{eqnarray}
Since just above we have proved uniform convergence of \eqref{e:row1} to \eqref{e:row2}, and both are strictly positive functions on $[0,\delta]$ (for $k$ sufficiently large), we can invert once more and deduce that also $t^{-A_i}m_{ik}$ converges uniformly to $t^{-A_i}\boldalpha_i$. Now the conditions of \cref{l:Kcontcrit} are met and we see that 
$m_{ik}$ converges to $\boldalpha_i$ in the strong topology. The proof of strong convergence for $m_{kj}$ follows by symmetry, therefore, we have shown that $(\boldalpha,\boldbeta)$ lies in $\Omega^{res}(\Kdelposst)$ and lifts $(\bigA,\bigB)$. 
Finally, the remaining vertical maps are surjections using that $\mathbb E$ and $\mathbb E^s$ are bijections. 
\end{proof}

\appendix

\section{Proof of  \cref{p:ToeplitzviaQ}}\label{a:ToeplitzviaQ}

Let $(\RR,\RR_{>0})$ be a ring with a positive structure. In this Appendix we prove the quiver description of totally positive Toeplitz matrices in $U_+^{(n+1)}(\RR_{>0})$ stated in  \cref{p:ToeplitzviaQ}. Namely, we give a direct proof in  terms of standard coordinates. We note that there is also an  alternative way one could prove this proposition using planar duality and an earlier description of totally positive Toeplitz matrices~\cite{rietschNagoya} in terms of the Givental quiver~\cite{Givental:QToda}, by making use a coordinate transformation constructed in \cite{Ludenbach}.

Recall first that the standard coordinates $(m_{ij})_{(i,j)\in\mathcal S_{\le n+1}}$ for $u\in U^{(n+1)}_+(\RR_{>0})$
are given by the factorisation 
\begin{multline}\label{e:idealfactRecall}
u=x_{\mathbf i_0}((m_{ij})):=\phi^{\mathbf i_0}_{>0}((m_{ij})_{i,j\in\mathcal S_{\le n+1}})=\big(x_{n}\left({m_{n,1}}\right)x_{n-1}\left({m_{n-1,1}}\right)\cdots\cdots x_2\left({m_{2,1}}\right) x_1\left({m_{1,1}}\right)\big)\\
\big(x_{n}\left({m_{n-1,2}}\right)\cdots x_2\left({m_{1,2}}\right)\big)\big(x_{n}\left({m_{n-1,2}}\right)\cdots x_3\left({m_{1,2}}\right)\big)\dots \big(x_{n}\left({m_{2,n-1}}\right)x_{n-1}\left({m_{1,n-1}}\right)\big) x_n\left({m_{1,n}}\right),
\end{multline}
where $\mathbf i_0=\mathbf i_0^{(n)}=(n,n-1,\dotsc,1,n,\dotsc,2,\dotsc, n,n-1,n)$.  We also consider $U_+^{\mathbf i_0}(\RR^\times)$ whose elements have standard coordinates $m_{ij}\in\RR^\times$, and think of $U_+^{(n+1)}(\RR_{>0})$ as the positive part of $U_+^{\mathbf i_0}(\RR)$ (where the coordinates $m_{ij}$ lie in $\RR_{>0}$). We will prove a proposition that characterises the Toeplitz matrices in $U_+^{\mathbf i_0}(\RR^\times)$, and hence within that the totally positive Toeplitz matrices, in terms of standard coordinates. Since $(\RR,\RR_{>0})$ necessarily contains $(\Q,\Q_{>0})$, see \cref{r:RcontainsQ}, it suffices to work over $\Q$.

Consider for any $(i,j)\in\mathcal S_{\le n+1}$ the expression
\[
v_{i,j}:= \sum_{\ell=0}^{\min(i,j)-1}\frac{1}{m_{i-\ell,j-\ell}}=\frac 1{m_{i,j}}+\frac{1}{m_{i-1,j-1}}+\dotsc.
\]
in $\Q[\, m_{i,j}^{\pm 1}\mid i,j\in\mathcal S_{\le n+1}]$. Since $v_{i,j}-v_{i-1,j-1}=\frac{1}{m_{i,j}}$ we have that $\Q[\, m_{i,j}^{\pm 1}\mid (i,j)\in\mathcal S_{\le n+1}]=\Q[\, (v_{i,j}-v_{i-1,j-1})^{\pm 1}\mid (i,j)\in\mathcal S_{\le n+1}]$. We consider the quiver $Q_{n+1}=(\mathcal V,\mathcal A)$ from \cref{d:Qn+1} with vertices indexed by $\{(0,0)\}\cup \mathcal S_{n+1}$ and labelled by the $v_{ij}$ for $(i,j)\in \mathcal S_{\le n+1}$ and $0$ for $(0,0)$ as shown in \cref{f:Qn+1}. We also recall that there are two types of vertices, $\mathcal V=\mathcal V_\circ\sqcup\mathcal V_{\bullet}$. Namely, those which are sources or sinks that we call $\circ$-vertices, and all other vertices, that we call $\bullet$-vertices.  

We associate to every arrow $a$ in the labelled quiver $Q_{n+1}$ the expression  in standard coordinates given by $\kappa_a=v_{h(a)}-v_{t(a)}$, where $h(a)$ is the head and $t(a)$ is the tail of the arrow $a$. To every $\bullet$-vertex $b$ we associate
\[
f_b((m_{ij})_{i,j})=\left(\prod_{a\in\operatorname{in}(b)}\kappa_a\right)-
\left(\prod_{a\in\operatorname{out}(b)}\kappa_a\right)\qquad \in \Q\left[\{m_{ij}^{\pm 1}\mid (i, j)\in \mathcal S_{\le n+1}\}\right],
\]
where $\operatorname{in}(b)$ and $\operatorname{out}(b)$ are the sets of incoming and outgoing arrows, respectively, associated to the vertex~$b$.

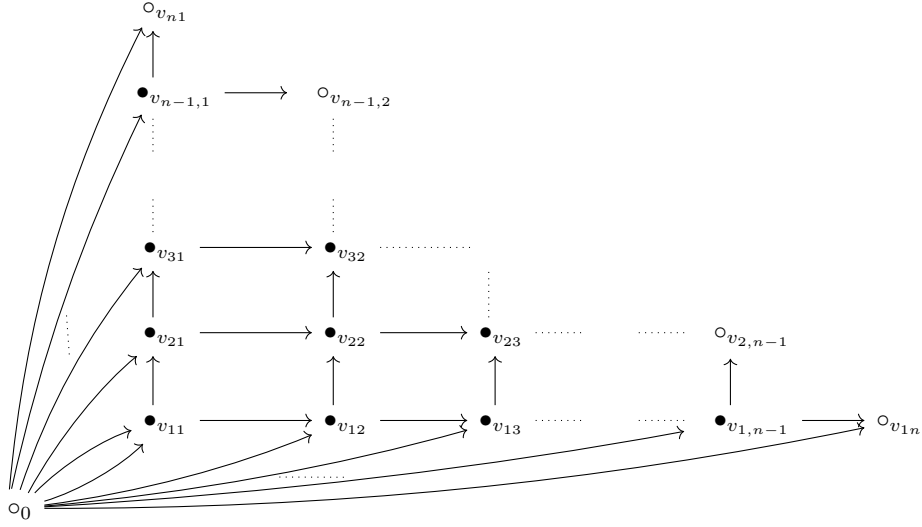
\begin{figure}
\[\begin{tikzcd}
	& {\circ_{v_{n1}}} \\
	& {\ \ \bullet_{v_{n-1,1}}} & {\ \ \circ_{v_{n-1,2}}} \\
	& {} & {} & {} \\
	& {\bullet_{v_{31}}} & {\bullet_{v_{32}}} & {} & {} \\
	& {\bullet_{v_{21}}} & {\bullet_{v_{22}}} & {\bullet_{v_{23}}} & {} & {\ \ \circ_{v_{2,n-1}}} \\
	& {\bullet_{v_{11}}} & {\bullet_{v_{12}}} & {\bullet_{v_{13}}} & {} & {\ \ \bullet_{v_{1,n-1}}} & {\circ_{v_{1n}}} \\
	{\circ_0}
	\arrow[shift left=2, from=2-2, to=1-2]
	\arrow[from=2-2, to=2-3]
	\arrow[shift right=2, shorten >=4pt, dotted, no head, from=2-2, to=3-2]
	\arrow[shift right=2, shorten >=4pt, dotted, no head, from=2-3, to=3-3]
	\arrow[shift left=2, shorten >=4pt, dotted, no head, from=4-2, to=3-2]
	\arrow[from=4-2, to=4-3]
	\arrow[shift left=2, shorten >=4pt, dotted, no head, from=4-3, to=3-3]
	\arrow[shorten >=6pt, dotted, no head, from=4-3, to=4-4]
	\arrow[shift left=2, from=5-2, to=4-2]
	\arrow[from=5-2, to=5-3]
	\arrow[shift left=2, from=5-3, to=4-3]
	\arrow[from=5-3, to=5-4]
	\arrow[shift left=2, shorten >=2pt, dotted, no head, from=5-4, to=4-4]
	\arrow[shorten >=6pt, dotted, no head, from=5-4, to=5-5]
	\arrow[shorten >=6pt, dotted, no head, from=5-6, to=5-5]
	\arrow[shift left=2, from=6-2, to=5-2]
	\arrow[from=6-2, to=6-3]
	\arrow[shift left=2, from=6-3, to=5-3]
	\arrow[from=6-3, to=6-4]
	\arrow[shift left, from=6-4, to=5-4]
	\arrow[shorten >=6pt, dotted, no head, from=6-4, to=6-5]
	\arrow[shift left=2, from=6-6, to=5-6]
	\arrow[shorten >=6pt, dotted, no head, from=6-6, to=6-5]
	\arrow[from=6-6, to=6-7]
	\arrow[shift left=2, curve={height=-12pt}, from=7-1, to=1-2]
	\arrow[""{name=0, anchor=center, inner sep=0}, shift left=2, curve={height=-6pt}, from=7-1, to=2-2]
	\arrow[""{name=1, anchor=center, inner sep=0}, shift left, curve={height=-6pt}, from=7-1, to=4-2]
	\arrow[curve={height=-6pt}, from=7-1, to=5-2]
	\arrow[curve={height=-6pt}, from=7-1, to=6-2]
	\arrow[curve={height=6pt}, from=7-1, to=6-2]
	\arrow[curve={height=6pt}, from=7-1, to=6-3]
	\arrow[""{name=2, anchor=center, inner sep=0}, curve={height=6pt}, from=7-1, to=6-4]
	\arrow[""{name=3, anchor=center, inner sep=0}, curve={height=6pt}, from=7-1, to=6-6]
	\arrow[curve={height=12pt}, from=7-1, to=6-7]
	\arrow[shorten <=6pt, shorten >=6pt, dotted, no head, from=0, to=1]
	\arrow[shorten <=6pt, shorten >=12pt, dotted, no head, from=2, to=3]
\end{tikzcd}\]
\caption{The quiver $Q_{n+1}$\label{f:Qn+1}}
\end{figure}
The following proposition implies \cref{p:ToeplitzviaQ}.  The explicit formula for the arrow-labels is in fact an incarnation of the coordinate change on arrow-labels of the dual (Givental) quiver introduced in \cite{Ludenbach}.
\begin{prop}\label{p:ToeplitzEquiv} Suppose $(m_{ij})_{i,j}\in(\RR^\times)^{\mathcal S_{\le n+1}}$. The following conditions are equivalent. 
\begin{enumerate}
\item
$u=x_{\mathbf i_0}((m_{ij}))$ is a Toeplitz matrix.
 \item
 The following identities hold for $(i,j)\in\mathcal S_{\le n+1}$,
 \begin{equation*}
 \begin{array}{ccc}
 v_{ij}-v_{i,j-1}&=&\frac 1{m_{ij}}\prod_{\ell=1}^{j-1}\frac{m_{i+1,\ell}}{m_{i,\ell}},\\
v_{ij}-v_{i-1,j}&=&\frac 1{m_{ij}}\prod_{\ell=1}^{i-1}\frac{m_{\ell,j+1}}{m_{\ell,j}},
 \end{array}
 \end{equation*}
with $v_{i,0}=v_{0,j}:=0$.
 \item
 For every $\bullet$-vertex $b$ of the quiver $Q_{n+1}$ we have $f_b((m_{ij}))=0$, that is, the product of the $\kappa_a$ associated to incoming arrows $a$ equals to  the product of the $\kappa_a$ associated to the outgoing arrows. 
\end{enumerate}
\end{prop}
\begin{remark}\label{r:ArrowLabelsandqi}
The proposition implies in particular that the arrow labels $\kappa_a$ lie in $\RR^\times$. Moreover, they lie in $\RR_{>0}$ if the $m_{ij}$ all lie in $\RR_{>0}$.  The characterisation (2) together with \cref{l:mijFormula} also implies 
\[
(v_{i,n+1-i}-v_{i,n-i})(v_{i,n+1-i}-v_{i-1,n+1-i})=\frac{\Delta_{i-1}(u)\Delta_{i+1}(u)}{\Delta_i^2(u)}=q_i(u).
\]
\end{remark}
We first prove the following lemma. Recall that $\Minor^{J}_K$ denotes the minor with row set $J$ and column set $K$. We set $\Minor^{J}_K=1$ if $J=\emptyset$, and $\Minor^{J}_K=0$ if either $J$ or $K$ is not a subset of $[1,n+1]$ (for an $(n+1)\times(n+1)$-matrix). Furthermore, recall that we use the notation $[i]=[1,i]$ and let $[0]=\emptyset$. 
\begin{lemma}\label{l:ToeplitzMinorChar}
The matrix $u=x_{\mathbf i_0}((m_{ij}))$ in $U_+^{\mathbf i_0}(\RR^\times)$ is Toeplitz if and only if
\begin{equation}\label{e:MinorChar}
\left(\Minor^{[i]}_{[i]+j}(u)\right)^2=\Minor^{[i]}_{[i]+j-1}(u)\Minor^{[i]}_{[i]+j+1}(u)+\Minor^{[i+1]}_{[i+1]+j}(u)\Minor^{[i-1]}_{[i-1]+j}(u)
\end{equation}
for all $(i,j)\in\mathcal S_{\le n+1}$. \end{lemma}
\begin{remark}
For example, if $i=1$ and $u=(a_{rs})_{r,s}$, then $\Minor^{[i]}_{[i]+j}(u)=a_{1,j+1}$, and \eqref{e:MinorChar} reads
\begin{equation}\label{e:jis1example}
a_{1,j+1}^2=a_{1,j}a_{1,j+2}+ \det\begin{pmatrix}a_{1,j+1}&a_{1,j+2}\\ a_{2,j+1}&a_{2,j+2}\end{pmatrix}.
\end{equation}
If $j=1$ then this reads $a_{12}^2=a_{13}+(a_{12}a_{23}-a_{13})=a_{12}a_{23}$. If $u\in U_+^{\mathbf i_0}(\RR^\times)$ as in the lemma, it follows that $a_{12}=m_{11}$ is invertible, thus we may deduce that $a_{12}=a_{23}$. Recursively, we may obtain that $a_{1,j-1}=a_{2,j}$ from the equations \eqref{e:jis1example} for general $j$. This is the initial part of the Toeplitz condition for $u$.  
\end{remark}
\begin{remark}\label{r:plethysm} The formula~\eqref{e:MinorChar} has a representation-theoretic interpretation. If $u=(u_{rs})_{r,s}$ is a Toeplitz matrix we may interpret its entries as complete homogeneous symmetric polynomials  $u_{ij}=h_{i-j}$ and interpret the minors above as Schur polynomials in $n$ variables. For example $\Minor^{[i]}_{[i]+j}(u)=s_{i\times j}$, the Schur polynomial associated to the partition $(j,\dotsc, j)$ with $i$ parts. Let $V_{i\times j}$ denote the irreducible representation of $GL_n$ with character $s_{i\times j}$. Then the left-hand side of \eqref{e:MinorChar} is the character of the tensor square $V_{i\times j}\otimes V_{i\times j}$, and the lemma tells us that
\[
V_{i\times j}\otimes V_{i\times j}=(V_{i\times (j-1)}\otimes V_{i\times (j+1)})\ \oplus\  (V_{(i-1)\times j}\otimes V_{(i+1)\times j})\]
Note that the tensor square also has the natural decomposition $V_{i\times j}\otimes V_{i\times j}=S^2(V_{i\times j})\oplus {\bigwedge}^2(V_{i\times j})$. However, these two decompositions do not coincide already for $i\times j=1\times 2$.
\end{remark}
\begin{proof}[Proof of \cref{l:ToeplitzMinorChar}]
Let us first show that \eqref{e:MinorChar} holds whenever $u$ is a Toeplitz matrix. Namely, this can be deduced from the Desnanot-Jacobi identity, which for any $(i+1)\times (i+1)$ matrix $U$ says that 
\begin{equation}\label{e:DesnanotJacobi}
\det(U)\Minor^{[2,i]}_{[2,i]}(U)=\Minor^{[i]}_{[i]}(U)\Minor^{[2,i+1]}_{[2,i+1]}(U) - \Minor^{[2,i+1]}_{[i]}(U)\Minor^{[i]}_{[2,i+1]}(U).
\end{equation}  
Suppose $U=u^{[i+1]}_{[i+1]+j}$, the submatrix of $u$ consisting of rows $1$ through $i+1$ and columns $j+1$ through $i+j+1$. Then if $u$ is Toeplitz we have
\begin{equation}\begin{array}{lclccl}
\Minor^{[2,i]}_{[2,i]}(U)&=&\Minor^{[1,i-1]}_{[1,i-1]}(U)&=&\Minor^{[i-1]}_{[i-1]+j}(u),
\\
\Minor^{[2,i+1]}_{[2,i+1]}(U)&=&\Minor^{[1,i]}_{[1,i]}(U)&=&
\Minor^{[i]}_{[i]+j}(u),
\end{array}
\end{equation}
 and similarly
 \begin{equation}
 \Minor^{[2,i+1]}_{[i]}(U)=\Minor^{[i]}_{[i]+j-1}(u),\quad\text{and}\quad
  \Minor^{[i]}_{[2,i+1]}(U)=\Minor^{[i]}_{[i]+j+1}(u).
 \end{equation}
Thus the Desnanot-Jacobi identity 
for $U=u^{[i+1]}_{[i+1]+j}$ implies \eqref{e:MinorChar} if $u$ is Toeplitz. 

For the converse, suppose that \eqref{e:MinorChar} holds for all $(i,j)\in\mathcal S_{\le n+1}$. Let $a_j:=u_{1,j+1}$. We may suppose we have proved $u_{rs}=a_{s-r}$ for all $(r,s)\le (i+1, i+j)$ (for the usual lexicographical order). See \eqref{e:partialu}. It remains to prove that $u_{i+1,i+j+1}=a_j$. 
\begin{equation}\label{e:partialu}
\begin{pmatrix}
1&a_1&a_2&\cdots &
\cdots&\cdots &\cdots&a_{j+i}&\cdots &a_{n}\\
0&1&a_1&\cdots &
&\ddots&& \vdots&&\\
0&\ddots&\ddots& \ddots&&&a_j& a_{j+1}&&\\
0&&&1&a_1&\cdots &a_{j-1}&u_{i+1,i+j+1}&\cdots &\\
0 &&&&1&u_{i+2,i+3}&\cdots&\cdots&&\\
0 &&&&&\ \ddots&&&&\\
\vdots &&&&&&&&&\\
\vdots &&&&&&&
&&\\
\vdots &&&&&&&&\ddots\quad&\vdots \\
0 &\cdots&&&&&&&&1
\end{pmatrix}
\end{equation} 
We can now use the assumption \eqref{e:MinorChar}. Since part of the matrix $u$ is already Toeplitz we have 
\[
\Minor^{[i-1]}_{[i-1]+j}(u)=\Minor^{[2,i]}_{[2,i]+j+1}(u)\quad\text{and}\quad
\Minor^{[i]}_{[i]+j+1}(u)=\Minor^{[2,i+1]}_{[2,i+1]+j},
\]
so that we can rewrite \eqref{e:MinorChar} and obtain 
\begin{multline*}
\left(\Minor^{[i]}_{[i]+j}(u)\right)^2
=\Minor^{[i]}_{[i]+j-1}(u)\Minor^{[2,i+1]}_{[2,i+1]+j}(u)+\Minor^{[i+1]}_{[i+1]+j}(u)\Minor^{[2,i]}_{[2,i]+j+1}(u)\\
=\Minor^{[i]}_{[i]+j}(u) \Minor^{[2,i+1]}_{[2,i+1]+j}(u),
\end{multline*}
where the second equality follows from  the Desnanot-Jacobi identity applied to the second summand. Next, note that $\Minor^{[i]}_{[i]+j}(u)$ is a monomial in terms of the $m_{ij}$ and therefore we can divide by it. Therefore it follows that 
\[
0=\Minor^{[i]}_{[i]+j}(u)-\Minor^{[2,i+1]}_{[2,i+1]+j}(u)=(a_j-u_{i+1,i+j+1})\Minor^{[i-1]}_{[i-1]+j}(u),
\]
where the second equality comes from expanding the minors along the final column. Since $\Minor^{[i-1]}_{[i-1]+j}(u)$ is again a monomial in the $m_{ij}$, we divide by it and deduce that $u_{i+1,i+j+1}=a_j$, completing the proof.   
\end{proof}

\begin{proof}[Proof of Proposition~\ref{p:ToeplitzEquiv}]  Let us set 
\begin{equation}\label{e:ne}
e_{ij}:=\frac 1{m_{ij}}\prod_{\ell=1}^{j-1}\frac{m_{i+1,\ell}}{m_{i,\ell}}\qquad\text{and} \qquad
n_{ij}:=\frac 1{m_{ij}}\prod_{\ell=1}^{i-1}\frac{m_{\ell,j+1}}{m_{\ell,j}}.
\end{equation}
Then (2) says that for the arrow $a$ pointing to the right from $(i,j-1)$ to $(i,j)$, we have $\kappa_a=e_{ij}$, and for the arrow $a$ pointing upwards from $(i-1,j)$ to $(i,j)$ we have $\kappa_a=n_{ij}$. 
It is straightforward that $e_{ij}n_{ij}=e_{i,j+1}n_{i+1,j}$ for $i+j<n+1$. Thus (2) implies (3).

Consider the total order on the $(i,j)\in\mathcal S_{\le n+1}$, which is the lexicographic order based on $(i+j,j)$. Namely,  
\[
(1,1)<(2,1)<(1,2)<(3,1)<(2,2)<(1,3)<(4,1)<(3,2)<\dotsc .
\]
We will prove that $(3)$ implies $(2)$ using induction. Let us check the first non-trivial case by explicitly. Our assumption (3) at the vertex labelled $v_{11}$ gives
\begin{equation}\label{e:v11}
(v_{12}-v_{11})(v_{21}-v_{11})=v_{11}^2 \quad \implies\quad v_{12}v_{21}=v_{11}v_{12}+v_{12}v_{11}\quad \implies \quad \frac{1}{v_{11}}=\frac{1}{v_{12}}+\frac{1}{v_{21}},
\end{equation}
noting that $v_{11}=\frac 1{m_{11}}$, $v_{21}=\frac 1{m_{21}}$ and $v_{12}=\frac{1}{m_{12}}$ are all invertible. 
In terms of standard coordinates we therefore  have $m_{11}=m_{12}+m_{21}$ which implies  
\begin{equation}\label{e:v11-arrows}
\begin{array}{ccl}
v_{12}-v_{11}&=&\frac{1}{m_{12}}-\frac{1}{m_{11}}=\frac{m_{11}-m_{12}}{m_{11}m_{12}}=\frac{1}{m_{12}}\frac{m_{21}}{m_{11}}\text{,\quad and }\\ 
v_{21}-v_{11}&=&\frac{1}{m_{21}}-\frac{1}{m_{11}}=\frac{m_{11}-m_{21}}{m_{11}m_{21}}=\frac{1}{m_{21}}\frac{m_{12}}{m_{11}}.
\end{array}
\end{equation}
Thus the arrow pointing east from $v_{11}$ to $v_{12}$ has label $v_{12}-v_{11}=e_{12}$, and the northward pointing arrow to $v_{21}$ has label $v_{21}-v_{11}=n_{21}$. This verifies (2) in the first non-trivial case. 

Suppose now that we have proved $(2)$ whenever $i+j\le k$. We have $v_{k1}=\frac{1}{m_{k1}}=e_{k1}$ for the label of the arrow from $(0,0)$ to $(k,1)$, automatically, and this can be considered the first instance of (2) with $i+j=k+1$. Therefore, we may assume that we have already proved
\[
v_{k-j+1,j}-v_{k-j+1,j-1}=e_{k-j+1,j}
\]
for some $j\ge 1$, as well as (2) for any $i+j\le k$. We consider the relevant part of the quiver as shown below. Here $x$ and $y$ are the arrow-labels for which the formulas from (2) have not yet been verified. 
\begin{equation}\label{e:Aquiversegment}\begin{tikzcd}
	{} & {} \\
	{} & {v_{k-j+1,j-1}} & {v_{k-j+1,j}} \\
	{} & {v_{k-j,j-1}} & {v_{k-j,j}} & {v_{k-j,j+1}} \\
	& {} & {v_{k-j-1,j}} & {}
	\arrow[draw=none, from=1-1, to=1-2]
	\arrow[from=2-1, to=2-2]
	\arrow[from=2-2, to=1-2]
	\arrow["{e_{k-j+1,j}}", from=2-2, to=2-3]
	\arrow[from=3-1, to=3-2]
	\arrow["{n_{k-j+1,j-1}}", from=3-2, to=2-2]
	\arrow["{e_{k-j,j}}", from=3-2, to=3-3]
	\arrow[""{name=0, anchor=center, inner sep=0}, "x", from=3-3, to=2-3]
	\arrow["y", from=3-3, to=3-4]
	\arrow[from=4-2, to=3-2]
	\arrow["{n_{k-j,j}}", from=4-3, to=3-3]
	\arrow[draw=none, from=4-4, to=3-4]
\end{tikzcd}\end{equation}
We now prove using (3) that $x=n_{k-j+1,j}$ and $y=e_{k-j,j+1}$ as expected. This is the induction step that will complete the proof of (3)$\implies$(2). 

From the definition of the $v_{ij}$ we have that $v_{ij}-v_{i-1,j-1}=\frac{1}{m_{ij}}$. On the other hand, $v_{ij}-v_{i-1,j-1}$ can be described as the sum of two arrow labels (in two ways). We use this observation to write 
\[
\begin{array}{ccccc}
v_{k-j+1,j}-v_{k-j,j-1}&=&\frac{1}{m_{k-j+1,j}}&=&x+e_{k-j,j}\\
v_{k-j,j+1}-v_{k-j-1,j}&=&\frac{1}{m_{k-j,j+1}}&=&y+n_{k-j,j}.
\end{array}
\]
Therefore we have
\[
x=\frac{1}{m_{k-j+1,j}}-e_{k-j,j}\qquad\text{and}\qquad
y=\frac{1}{m_{k-j,j+1}}-n_{k-j,j}.
\]
Now $(k-j,j)$ is a $\bullet$-vertex, and (3) combined with the above identities give us  that
\[
e_{k-j,j}n_{k-j,j}=xy=\left(\frac{1}{m_{k-j+1,j}}-e_{k-j,j}\right)\left(\frac{1}{m_{k-j,j+1}}-n_{k-j,j}\right).
\]
Multiplying out and cancelling it follows that 
\[
\frac{1}{m_{k-j+1,j}m_{k-j,j+1}}=\frac{e_{k-j,j}}{m_{k-j,j+1}}+\frac{n_{k-j,j}}{m_{k-j+1,j}},
\]
which implies the middle equality in each of the identities below
\[
\begin{array}{ccc}
x=\frac{1}{m_{k-j+1,j}}-e_{k-j,j}&=&\frac{m_{k-j,j+1}n_{k-j,j}}{m_{k-j+1,j}}=n_{k-j+1,j},\\
y=\frac{1}{m_{k-j,j+1}}-n_{k-j,j}&=&\frac{m_{k-j+1,j}e_{k-j,j}}{m_{k-j,j+1}}=e_{k-j,j+1}.
\end{array}
\]
This concludes the proof of (3)$\implies$(2).
It remains to prove the equivalence of (1) and (2). 

We first rewrite the $n_{ij}$ and $e_{ij}$ from \eqref{e:ne} in terms of minors using \cref{l:mijFormula}. It turns out that both $\prod_{\ell=1}^{i-1}m_{i,\ell}$ and $\prod_{\ell=1}^{j-1}=m_{\ell,j}$ are telescopic products, and  $e_{ij}$ and $n_{ij}$ simplify to 
\begin{equation}
e_{ij}=\frac{\Minor^{[i+1]}_{[i+1]+j-1}(u)\Minor^{[i-1]}_{[i-1]+j}(u)}{\Minor^{[i]}_{[i]+j}(u)\Minor^{[i]}_{[i]+j-1}(u)},\qquad
n_{ij}=
\frac{\Minor^{[i-1]}_{[i-1]+j+1}(u)\Minor^{[i]}_{[i]+j-1}(u)}
{\Minor^{[i]}_{[i]+j}(u)\Minor^{[i-1]}_{[i-1]+j}(u)}.
\end{equation}
Suppose (1) holds, i.e. $u$ is a Toeplitz matrix. Since (2) is trivially satisfied for $(i,j)=(1,1)$ we may assume we have proved (2) up to the diagonal given by $i+j=k$. Then we have
\begin{multline}\label{e:kappaMinor}
v_{i,j+1}-v_{i,j}=(v_{i,j+1}-v_{i-1,j})-(v_{i,j}-v_{i-1,j})\\=\frac{1}{m_{i,j+1}}-n_{i,j}=
\frac{\Minor^{[i]}_{[i]+j}(u)\Minor^{[i-1]}_{[i-1]+j+1}(u)} {\Minor^{[i]}_{[i]+j+1}(u)\Minor^{[i-1]}_{[i-1]+j}(u)}-
\frac{\Minor^{[i-1]}_{[i-1]+j+1}(u)\Minor^{[i]}_{[i]+j-1}(u)}
{\Minor^{[i]}_{[i]+j}(u)\Minor^{[i-1]}_{[i-1]+j}(u)}\\=
\frac{\Minor^{[i-1]}_{[i-1]+j+1}(u)} {\Minor^{[i]}_{[i]+j}(u)\Minor^{[i]}_{[i]+j+1}(u)\Minor^{[i-1]}_{[i-1]+j}(u)}\left(\left(\Minor^{[i]}_{[i]+j}(u)\right)^2-\Minor^{[i]}_{[i]+j-1}(u)\Minor^{[i]}_{[i]+j+1}(u)\right).
\end{multline}
Since $u$ is Toeplitz we may replace the expression in brackets using \cref{l:ToeplitzMinorChar} to obtain
\begin{equation*}
v_{i,j+1}-v_{i,j}=\frac{\Minor^{[i-1]}_{[i-1]+j+1}(u)\Minor^{[i+1]}_{[i+1]+j}(u)} {\Minor^{[i]}_{[i]+j}(u)\Minor^{[i]}_{[i]+j+1}(u)}=e_{i,j+1}.
\end{equation*}
The remaining statement, that $v_{i+1,j}-v_{i,j}=n_{i+1,j}$ follows by symmetry. Thus we have proved (1)$\implies$(2).

Finally, suppose that (2) holds and, indirectly, that $u$ is not Toeplitz matrix. Then by \cref{l:ToeplitzMinorChar} we have that the minor identity \eqref{e:MinorChar} fails for some $(i,j)$.  But using (2) we may rewrite $v_{i,j+1}-v_{i,j}$ as in \eqref{e:kappaMinor}. Now the identity $v_{i,j+1}-v_{i,j}=e_{i,j+1}$ (which is assertion of (2) for $(i,j+1)$), is seen to be false if the identity \cref{l:ToeplitzMinorChar} for $(i,j)$ doesn't hold. Thus we have arrived at a contradiction.  So (2)$\implies$(1). 
\end{proof}

\begin{remark}
Let $J$ be the ideal in $\Q[\{m_{ij}^{\pm 1}\mid (i, j)\in \mathcal S_{\le n+1}\}]$ generated by the $f_b$,
\begin{equation}\label{e:IdealFromQuiver}
J=J_{n+1}=\left(f_b(m_{ij})_{i,j}\mid b\in\mathcal V_{\bullet}\right).
\end{equation}
The proof of \cref{p:ToeplitzEquiv} in fact shows that $J$ is the ideal associated to the subvariety of Toeplitz matrices in $U_+^{\mathbf i_0}(\Q)$. 
\end{remark}
\bibliographystyle{amsalpha}
\bibliography{biblio}

\end{document}